\documentclass[11pt, a4paper]{article}

\usepackage{fullpage}
\usepackage[margin=2.5cm]{geometry}
\usepackage{setspace}
\usepackage[section]{placeins}
\usepackage{authblk}

\usepackage{xcolor}           
\usepackage[colorlinks=true, allcolors= blue]{hyperref}
\hypersetup{
    bookmarksdepth=3 %
}

\usepackage{todonotes}
\usepackage[draft, deletedmarkup=sout]{changes} 
\definechangesauthor[name=Rai, color=green]{R}
\definechangesauthor[name=David, color=blue]{D}
\definechangesauthor[name=Krish, color=purple]{K}
\definechangesauthor[name=Ali, color=red]{A}

\usepackage{amsmath,amssymb,amsthm,mathtools,mathrsfs}
\usepackage{bbold}             
\usepackage{stackengine}       
\usepackage{nicefrac}   
\usepackage{cleveref}

\newtheorem{Theorem}{Theorem}[section]

\newtheorem{Lemma}[Theorem]{Lemma}

\newtheorem{Proposition}[Theorem]{Proposition}
\newtheorem{Definition}[Theorem]{Definition}
\newtheorem{remark}[Theorem]{Remark}

\renewcommand{\H}{\mathcal{H}}

\newcommand{\C}{\mathcal{C}}
\newcommand{\D}{\mathcal{D}}
\newcommand{\I}{\mathcal{I}}

\newcommand{\E}{\mathcal{E}}
\newcommand{\F}{\mathcal{F}}
\newcommand{\G}{\mathcal{G}}
\newcommand{\M}{\mathcal{M}}
\newcommand{\K}{\mathcal{K}}
\newcommand{\Q}{\mathcal{Q}}
\newcommand{\beliefsupport}{2^\K_\emptyset}
\renewcommand{\P}{\mathcal{P}}

\newcommand{\R}{\mathcal{R}}
\renewcommand{\S}{\mathcal{S}}

\newcommand{\U}{\mathcal{U}}

\newcommand{\W}{\mathcal{W}}
\newcommand{\X}{\mathcal{X}}

\newcommand{\NN}{\mathbb{N}}
\newcommand{\PP}{\mathbb{P}}
\newcommand{\EE}{\mathbb{E}}

\newcommand{\RR}{\mathbb{R}}

\newcommand{\supp}{\textnormal{supp}}   

\newcommand{\Post}{\textnormal{Post}}
\newcommand{\until}{\,..\,}
\newcommand{\eps}{\varepsilon}
\newcommand{\given}{\,|\,}
\newcommand{\givenm}{\,\middle|\,}

\DeclareMathOperator*{\argmin}{arg\,min}
\newcommand{\defas}{\coloneqq}
\newcommand{\play}{\rho}
\newcommand{\Plays}{\Omega}

\newcommand{\another}[1]{\widetilde{#1}}
\newcommand{\POMDP}{P}
\newcommand{\MDP}{M}
\newcommand{\com}{\mathsf{com}}
\newcommand{\rev}{\mathsf{rev}}
\newcommand{\syn}{\mathsf{syn}}
\newcommand{\safe}{\mathsf{safe}}
\newcommand{\rep}{\mathsf{rep}}
\newcommand{\reset}{\mathsf{reset}}
\newcommand{\unsafe}{\mathsf{unsafe}}
\newcommand{\erase}{\mathtt{?}}

\allowdisplaybreaks

\title{The Complexity of Approximating the Value in\\ Revealing POMDPs with Long-Run Average Objectives}

\author[1]{Ali Asadi}
\author[1]{Krishnendu Chatterjee}
\author[1]{David Lurie}

\affil[1]{Institute of Science and Technology Austria}

\date{\today}

\begin{document}
\maketitle

\begin{abstract}
    We study partially observable Markov decision processes (POMDPs) with long-run average objectives, where the payoff is defined as the limit inferior of the expected average rewards. We consider the computational problem of approximating the long-run average value of a POMDP. In general, the long-run average value of a POMDP is neither computable nor approximable. We therefore consider the subclass of revealing POMDPs. Informally, a POMDP is revealing when the controller observes the underlying state with positive probability at each stage of the process. Our main contributions are threefold. First, we illustrate the practical relevance of this class of POMDPs through an application in control and optimization. Second, we present an exponential-time algorithm for the value approximation problem. Third, we establish EXPTIME-hardness by a reduction from the problem of almost-sure safety in POMDPs. Together, these results show that the problem of approximating the long-run average value in revealing POMDPs is EXPTIME-complete.\\

\end{abstract}

\newpage


\section{Introduction}

    \emph{Partially observable Markov decision processes} (POMDPs) are a classical model for sequential decision-making under partial information~\cite{bertsekas1976DynamicProgrammingStochastic,papadimitriou1987complexity,kaelbling1998planning}.
    At each stage, the environment is in a hidden state and the controller chooses an action.
    Together, the current state and the chosen action determine a probability distribution over the successor state and an observed signal.
    Since the state is not directly observed, the controller must choose actions using only past actions and signals.
    All the information available is summarized by the \emph{belief}, which is the conditional probability distribution of the current state given the observed history of actions and signals.
    Two well-known models arise as special cases: \emph{Markov decision processes} (MDPs), in which the current state is fully observed~\cite{puterman1994}, and \emph{blind MDPs}, in which no information about the state is observed and which correspond to probabilistic finite automata~\cite{rabin1963probabilistic,paz1971introduction}.
    
    POMDPs arise in many applications such as communication networks and queueing systems~\cite{arapostathis1993discrete}, machine and inventory problems~\cite{smallwood1973optimal,wang2019inventory}, and reinforcement learning~\cite{kaelbling1996reinforcement}.
    Many of these applications model systems that operate over huge horizons, making long-run performance a natural optimization criterion.
    We consider the \emph{long-run average objective}, defined as $\liminf_{n\to\infty} \EE \left(\tfrac{1}{n}\sum_{m=1}^n G_m \right)$, where $\EE$ denotes the expectation and $G_m$ is the reward obtained at stage $m$.
    The corresponding \emph{value} is the supremum of this quantity over all strategies of the controller. 
    Notably, the value coincides with several classical definitions of long-run value, such as the asymptotic and uniform value~\cite{rosenberg2002blackwell}, uncertain duration value~\cite{NS10}, general uniform value~\cite{renault2017long}, and history-dependent value~\cite{VZ21}.
    
    In this paper, we focus on whether this value can be \emph{computed} in POMDPs.
    For MDPs, the value can be computed efficiently~\cite{puterman1994}.
    The situation changes radically under partial observation.
    For POMDPs, Madani et al.~\cite{madani2003undecidability} proved that no algorithm can compute, or even approximate, the value in general.
    This raises the natural question of identifying structural subclasses of POMDPs for which approximation becomes possible.

    We consider the subclass of \emph{revealing POMDPs}.
    Informally, a revealing POMDP requires that the controller directly observes the underlying state with positive probability at each stage.
    Consequently, although the state remains hidden in general, the belief occasionally collapses to a Dirac distribution on the revealed state.
    Our main contributions are the following:
    \begin{itemize}
        \item 
            First, \Cref{Section: Application to sensing} demonstrates how the abstract class of revealing POMDPs captures practically relevant problems in control and optimization.
        \item 
            Second, \Cref{Result: Approximating the uniform value for revealing POMDPs is decidable} establishes that approximating the long-run average value of revealing POMDPs is $\mathrm{EXPTIME}$-complete.
    \end{itemize}
    For the upper complexity bound, the proof proceeds in two steps.

    The first step identifies regions of the belief dynamics in which the long-run behavior is well defined.
    Since the set of beliefs is infinite, we work instead with the \emph{belief-support MDP}, which records only the support of the current belief. 
    Because this MDP has a finite set of states, it decomposes into \emph{end-components}.
    An end-component is a set of belief-supports and allowed actions in which the controller can remain forever.
    We fix a maximal end-component and restrict attention to strategies that keep the belief-support inside it.
    For these strategies, we prove that the corresponding finite-horizon values converge at an explicit rate to a long-run value that depends only on the end-component and not on the initial belief.

    The main argument is an asynchronous coupling between two copies of the process starting from arbitrary beliefs with support in the same maximal end-component.
    The first copy follows the given strategy until a state is revealed.
    The second copy waits until a state is revealed and then uses paths within the end-component to reach the state revealed in the first copy, possibly at a later stage.
    Once this happens, the two copies follow the same continuation strategy and obtain the same sequence of rewards from then on.
    Only the stages before this synchronization can differ, and since it occurs after an explicit number of stages with high probability, their contribution to the average vanishes at an explicit rate.
    A block argument then yields the convergence of the finite-horizon values.
    Finally, the bound obtained is uniform in the two initial beliefs, which shows that the limit depends on the end-component alone.
    
    The second step turns these local long-run average values into a reachability problem.
    For each state that can be revealed, we consider the best long-run value that the controller can secure by remaining in an end-component containing that state.
    We then construct a revealing POMDP, called \emph{Commit POMDP}, in which, whenever a state is revealed, the controller may commit to it.
    The corresponding commit action leads to a target state with probability equal to its best end-component value and to a sink state otherwise.
    We prove that the value of the Commit POMDP with reachability objectives coincides with the value of revealing POMDPs with the long-run average objectives. 
    Intuitively, committing replaces an infinite continuation within an end-component by a one-step bet with the same value.
    To prove this reduction, we use Chatterjee et al.~\cite[Lemma 5.3, p.~109]{chatterjee2022finite} and Venel and Ziliotto~\cite[Lemma 33, p.~2004]{venel2016strong}.
    These results show that, after finitely many stages, there exists a random belief together with continuation strategies whose long-run average objectives converge almost surely from every state in the support of that random belief.
    We use this property to relate these almost-sure long-run average objectives to the maximum safe values in the Commit POMDP.
    Since the end-component values are approximated by finite-horizon values, we actually construct an \emph{Approximate Commit POMDP} and show that this approximation changes its reachability value by at most $\eps$.
    Applying the existing approximation algorithm for revealing POMDPs with reachability objectives~\cite{asadi2026revealing} then yields the claimed $\mathrm{EXPTIME}$ upper bound.
    
    For the lower bound, we give a reduction from almost-sure safety in POMDPs, inspired by the hardness proof for revealing POMDPs with parity objectives~\cite{belly2025revelations}.
    The reduction makes unsafe states absorbing and assigns reward one to safe states and reward zero to unsafe states.
    It then adds, alongside every original transition, a transition that reveals the exact successor state and a revealing transition that resets the process to its initial state.
    If the original POMDP is almost-sure safe, the constructed revealing POMDP has value one.
    Otherwise, every strategy reaches an unsafe state almost surely, and hence the constructed POMDP has value zero.
    Thus, even a constant-error approximation distinguishes the two cases.
    Together with the reduction from almost-sure safety, this proves $\mathrm{EXPTIME}$-completeness.

\paragraph{Related work}
    Our result lies at the intersection of the literature on revealing POMDPs and that on stochastic optimization under partial observation.

    Our revealing condition coincides with the strongly revealing condition introduced by Belly et al.~\cite{belly2025revelations}.
    Belly et al.~\cite{belly2025revelations} proved that the almost-sure analysis for revealing POMDPs with parity objectives is $\mathrm{EXPTIME}$-complete.
    Asadi et al.~\cite{asadi2026revealing} extended this analysis by proving that the limit-sure and quantitative analysis for revealing POMDPs with parity objectives is in $\mathrm{EXPTIME}$.
    Chen and Liew studied \emph{intermittently observable MDPs}~\cite{chenIntermittentlyObservableMarkov2023} with discounted reward, in which the state is either perfectly observed or not observed at all.
    Avrachenkov et al.~\cite{avrachenkov2025constrained} studied constrained average rewards and proved, under a recurrence assumption on the underlying dynamics, that the belief MDP is unichain and satisfies strong duality.

    Computing or approximating the value of POMDPs has been extensively studied.
    However, most results~\cite{belly2025revelations, chatterjee2014partial,chatterjee2013survey,chatterjee2010probabilistic, gimbert2014deciding,asadi2026revealing} concern logical objectives, a different class of objectives; see Chatterjee et al.~\cite{chatterjee2012survey} for a survey.
    By contrast, algorithmic results for POMDPs with long-run average objectives are scarce.
    In general, no algorithm can compute, or even approximate, the long-run average value~\cite{madani2003undecidability}.
    Recently, Chatterjee et al.~\cite{chatterjee2022finite} proved that approximately optimal finite-memory strategies always exist.
    This implies that the approximation problem is recursively enumerable, but it does not provide an explicit approximation algorithm.
    Positive results exist for POMDPs under additional assumptions on the belief dynamics.
    Chatterjee et al.~\cite{chatterjee2025ergodic} proved that ergodic blind MDPs with long-run average objectives can be approximated algorithmically.
    Chatterjee et al.~\cite{chatterjee2026approximating} extended this result to POMDPs under a Doeblin condition, which imposes a uniform reset property on the belief dynamics.

\paragraph{Novelty}
    Our contribution is novel in the following respects.
    First, revealing POMDPs include partially observed nonergodic POMDPs, that is, POMDPs whose long-run average value depends on the initial belief.
    To the best of our knowledge, this is the first approximation algorithm for a general subclass of nonergodic POMDPs with long-run average objectives.
    Second, to the best of our knowledge, this is the first general subclass of POMDPs with long-run average objectives for which the approximation problem is $\mathrm{EXPTIME}$-complete.
    Third, the revealing condition is stated directly in terms of the data of the model and does not impose any constraints on the set of beliefs.
    Finally, the application presented in \Cref{Section: Application to sensing} shows that the revealing condition arises naturally in stochastic control and optimization models.
    
\paragraph{Organization of the paper}
    \Cref{Section: Preliminaries} introduces POMDPs and the long-run average objective.
    \Cref{Section: Revealing POMDPs} defines revealing POMDPs and states the main result.
    \Cref{Section: Application to sensing} describes a sensing architecture that induces a revealing POMDP with a long-run average objective.
    \Cref{Seection: proof upper bound} proves the $\mathrm{EXPTIME}$ upper bound for approximating the long-run average value in revealing POMDPs.
    Finally, \Cref{Section: Hardness} proves the $\mathrm{EXPTIME}$ lower bound for approximating the long-run average value in revealing POMDPs
    

\section{Preliminaries}\label{Section: Preliminaries}

This section introduces partially observable Markov decision processes (POMDPs).

\paragraph{Notation}
    Calligraphic letters (e.g., $\I,\H,\K,\S$) denote sets, their elements (e.g., $i$, $h$, $k$, $s$) appear in lowercase, and random elements use uppercase (e.g., $I$, $H$, $K$, $S$). 
    Given a finite set $\C$, we denote the set of probability distributions on $\C$ by $\Delta(\C)$.
    Given an element $c \in \C$, we denote the Dirac measure on $c$ by $\delta_{c}$.
    We write $[a \until b]$ for the integer set $\{a,a+1,\ldots,b\}$, where $a$ and $b$ are integers.
    The set of real numbers is denoted by $\RR$, while $\NN$ and $\NN^*$ represent the sets of natural numbers and nonzero natural numbers, respectively. 
    Given a vector $b \in \RR^{|\K|}$, we denote its transpose by $b^\top$.
    The maximum over an empty set is $0$.
    
\paragraph{Model} 
    A \textit{POMDP}, denoted by $\POMDP$, is defined by a tuple $\POMDP=(\K,\I,\S,p,g),$ where:
    \begin{itemize}
        \item 
            $\K$ is the finite set of states;
        \item 
            $\I$ is the finite set of actions;
        \item 
            $\S$ is the finite set of signals;
        \item 
            $p \colon \K\times \I\rightarrow\Delta(\K\times \S)$ is the transition probability function;
        \item 
            $g \colon \K\times\I\rightarrow[0,1]$ is the stage reward function.
    \end{itemize}

\paragraph{Related models}
    Markov decision processes~\cite{puterman1994} are POMDPs in which the observed signal is the successor state.
    Formally, an MDP is denoted by $M = (\K, \I, p,g)$ with transition function $p \colon \K \times \I \to \Delta(\K)$. 
    Blind MDPs~\cite{madani2003undecidability} are POMDPs in which the controller is said to be \emph{blind}, i.e., the signal set is a singleton.
    
\paragraph{Dynamic}
    An \emph{initial distribution} is a probability distribution $b_1\in \Delta(\K)$ over the states, according to which the initial state $K_1$ is drawn.
    A POMDP starting from the initial distribution $b_1\in \Delta(\K)$, denoted by $\POMDP(b_1)$, evolves as follows.
    The controller knows $b_1$, but does not know $k_1$, the realization of $K_1$.
    At each stage $m\in \NN^*$: 
    \begin{enumerate}
        \item 
            The controller selects an action $I_m$;
        \item 
            A stage reward $G_m\defas g(K_m,I_m)$ is generated, but \emph{not} observed by the controller;
        \item 
            Next, the successor state $K_{m+1}$ and the public signal $S_{m+1}$ are drawn according to $p( \,\cdot \given K_m, I_m)$.
        \item 
            Finally, the controller observes the tuple $(I_m,S_{m+1})$ but neither $K_{m+1}$ nor $G_m$.
    \end{enumerate}   
    
\paragraph{Matrices}
    For every action $i\in \I$ and signal $s\in\S$, define the matrix $P(i,s)$ by setting, for all states $k,k'\in\K$,
    \[
        P_{k,k'}(i,s) \defas p(k',s \given k,i) \defas p(k,i)(k',s).
    \]
    Denote by $\P\defas\left\{P(i,s)\colon\; (i,s)\in\I\times \S\right\}$ the set of all such matrices.
    Each matrix $P\in\P$ represents the joint probabilities of transitioning from the current state $k\in \K$ to a successor state $k'\in \K$ and observing signal $s\in\S$, given the chosen action $i\in\I$. 

\paragraph{History}
    A \emph{history} before stage $m$ is a sequence $(i_1,s_2,\ldots,i_{m-1},s_m)$. 
    The set of histories before stage $m$ is denoted by $\H_m\defas (\I\times\S)^{m-1}$, with $(\I\times \S)^{0}\defas\{\emptyset\}$. 
    Given a history $h_m=(i_1,s_2,\ldots,s_m)$ and $(i,s)\in \I\times \S$, we denote their concatenation by
    \[
        h_m\times(i,s)\defas(i_1,s_2,\ldots,s_m,i,s).
    \]
    
\paragraph{Play}
    A play in a POMDP is an infinite sequence $\play = (k_1, i_1, s_2, k_2, i_2, s_3, k_3, \ldots)$ of states, actions, and signals such that, for all $m \ge 1$, $P_{k_m,k_{m+1}}(i_m,s_{m + 1}) > 0$. 
    The set of all plays is denoted by $\Plays$.
    
\paragraph{Strategy} 
    A (history-dependent) \emph{strategy} is a mapping $\sigma\colon\bigcup_{m\ge 1}\H_m\to \I$ with $\sigma(i\given h_m) \defas \sigma(h_m)(i)$ the probability of choosing action $i\in \I$ given the history $h_m \in \H_m$.
    We denote the set of strategies by $\Sigma$.

\begin{remark}
    A strategy is \emph{pure} if, for every history, it assigns a Dirac measure to some action $i\in \mathcal{I}$.
    A strategy is \emph{behavioural} if, for every history, it assigns a probability distribution over~$\I$. 
    By~\cite{feinberg1996,venel2016strong}, allowing behavioural strategies does not change the value of the POMDP.
    Therefore, restricting attention to pure strategies is without loss of generality and our results remain valid for behavioural strategies.
\end{remark}

\paragraph{Random history}
    Given an initial distribution $b_1 \in\Delta(\K)$, a strategy $\sigma \in \Sigma$, and a stage $m \in \NN^*$, define the \emph{random history} at stage $m$ by $H_m \defas(I_1,S_2,\ldots,I_{m-1},S_{m})$, which takes values in $\H_m$.

\paragraph{Continuation strategy}
    Given a history $h_m\in \H_m$ with $m\in \NN^*$, the $h_m$-shift of a strategy $\sigma\in \Sigma$, denoted by $\sigma[h_m]$, is defined by, for all $m'\in \NN^*$, $\sigma[h_m](h_{m'})\defas \sigma(h_m,h_{m'})$. 
    We call $\sigma[H_m]$ the random shift at stage $m$.
    In other words, $\sigma[h_m]$ corresponds to the continuation of the strategy $\sigma$ given that the history before stage $m$ was $h_m$.
    
\paragraph{Probability measure} 
    Let $\Omega \defas (\K\times\I\times \S)^{\NN^*}$ be the set of plays. 
    Given an initial distribution $b_1\in\Delta(\K)$ and a strategy $\sigma\in \Sigma$, denote by $\PP_{\sigma}^{b_1}$ the induced probability measure on $\Omega$, which is supported on $\Plays$.
    Similarly, denote the corresponding expectation under this measure by  $\EE^{b_1}_{\sigma}$.
    
\paragraph{Admissible history}
    Given an initial distribution $b_1 \in\Delta(\K)$ and a stage $m \in \NN^*$, define the set of admissible histories from $b_1 \in \Delta(\K)$ by
    \[
        \H_m(b_1) \defas \left\{ h_m \in \H_m \colon\; \exists \sigma \in \Sigma\quad \PP^{b_1}_{\sigma} (H_m=h_m)>0\right\}.
    \]

\paragraph{Belief}
    Given $\sigma\in \Sigma$, $b_1\in \Delta(\K)$ and $h_m\in \H_m(b_1)$, the \emph{belief} induced by $h_m$, denoted by $b^{b_1}_{h_m}\in \Delta(\K)$, is the conditional distribution of the current state given that history, i.e., for every state $k\in \K$,
    \[
        b^{b_1}_{h_m}(k)\defas \PP_\sigma^{b_1}\left(K_m=k\given H_m=h_m\right).
    \]
    The (random) belief at stage $m\in \NN^*$ is denoted by $B_m\defas b^{b_1}_{H_m}$ where $B_1=b_1$.

\paragraph{Belief update}
    Given $b\in \Delta(\K)$, $i\in \I$, and $s\in \S$ such that $\PP(s\given b,i)>0$, the belief update is given by Bayes' rule
    \[
        \Phi(b,i,s)(k')
        \defas
        \dfrac{\sum_{k\in \K}b(k)P_{k,k'}(i,s)}{\PP(s\given b,i)},
    \]
    where $\PP(s\given b,i)$ is the probability of observing the signal $s$ given the current belief $b$ and action~$i$, defined by $\PP(s\given b,i)\defas \sum_{k\in \K}\sum_{k'\in \K}b(k)P_{k,k'}(i,s)$.
                
\paragraph{Objective}
    We consider the following objective functions: 
    \begin{itemize}
        \item 
            \textbf{$n$-stage:}
            Given a finite horizon $n\in\NN^*$, an initial belief $b_1\in \Delta(\K)$, and a strategy $\sigma\in \Sigma$, the $n$-stage objective is 
            \[
                \gamma_n(b_1,\sigma)\defas\EE_{\sigma}^{b_1}\left(\dfrac{1}{n}\sum_{m=1}^{n} G_m\right).
            \]
            The $n$-stage value is $v_n(b_1)\coloneqq\sup_{\sigma\in\Sigma} \gamma_n(b_1,\sigma)$.
        \item 
            \textbf{Long-run average:}
            Given an initial belief $b_1\in \Delta(\K)$ and a strategy $\sigma\in \Sigma$, the long-run average objective is
            \[
                \gamma(b_1,\sigma)\defas\liminf_{n\to\infty}\EE_{\sigma}^{b_1}\left(\dfrac{1}{n}\sum_{m=1}^{n} G_m\right).
            \]
            The long-run average value is $v(b_1)\coloneqq\sup_{\sigma\in\Sigma} \gamma(b_1,\sigma)$.
        \item 
            \textbf{Reachability:}
            Given an initial belief $b_1\in \Delta(\K)$, a set of target states $\X\subseteq \K$ and a strategy $\sigma\in \Sigma$, the reachability objective is 
            \[
                \gamma_{R}(b_1,\sigma)\defas \PP_{\sigma}^{b_1}(\exists m\in \NN^*\colon\; K_m\in \X).
            \]
            The reachability value is $v_{R}(b_1)\defas\sup_{\sigma\in \Sigma} \gamma_{R}(b_1,\sigma)$.
    \end{itemize}

    \begin{remark}
        \label{remark: values coincides}
        By~\cite{rosenberg2002blackwell,NS10,renault2017long,VZ21}, the long-run average value coincides with other classical limit values in POMDPs, i.e., 
        \begin{align*}
            v(b_1) 
            &= \sup_{\sigma\in \Sigma}\EE_{\sigma}^{b_1}\left(\liminf_{n\to\infty}\dfrac{1}{n}\sum_{m=1}^nG_m\right) \\
            &= \lim_{n\to\infty}\sup_{\sigma\in \Sigma}\EE_{\sigma}^{b_1}\left(\dfrac{1}{n}\sum_{m=1}^{n} G_m\right)\\
            &= \lim_{\lambda\to 0}\sup_{\sigma\in \Sigma}\EE_{\sigma}^{b_1}\left(\sum_{m\in \NN^*}\lambda(1-\lambda)^{m-1}G_m\right).
        \end{align*}
    \end{remark}

    \paragraph{Approximation problem}
        Given a POMDP $\POMDP$, an initial belief $b_1\in \Delta(\K)$, and $\eps>0$, the approximation problem for the long-run average value is to compute $\overline{v}$ such that
        \[
            \left|\overline{v} - v(b_1)\right| \le \eps.
        \]


\section{Revealing POMDPs}\label{Section: Revealing POMDPs}

This section introduces \emph{revealing POMDPs} and our main contributions.

\paragraph{Class description}
    The subclass of \emph{revealing POMDPs} was previously studied in~\cite{asadi2026revealing,belly2025revelations}.
    Intuitively, the revealing property ensures that, whenever a state is visited, the controller is informed of that state with a positive probability.

    \begin{Definition}[Revealing POMDP]
        A POMDP is revealing if for every action $i\in \I$ and pair of states $k,k'\in \K$,
        \[
            \sum_{s\in \S}P_{k,k'}(i,s)>0\, \implies \,\exists s^*\in \S\text{ such that }P_{k,k'}(i,s^*)>0 \text{ and }\sum_{\underline{k}\in \K}\sum_{\overline{k}\in \K\setminus\{k'\}}P_{\underline{k},\overline{k}}(i,s^*)=0.
        \]
    \end{Definition}

    \begin{remark}
        The class of revealing POMDPs is nonergodic, i.e., the  long-run average value depends on the initial belief in general.
        Indeed, consider a revealing POMDP with two states $k_1$ and $k_2$, one action $i$, and two signals $s_1$ and $s_2$ such that $P_{k_1,k_1}(i,s_1)=P_{k_2,k_2}(i,s_2)=1$, with rewards $g(k_1,i)=0$ and $g(k_2,i)=1$.
        If the initial belief is $b_1=\delta_{k_1}$, then the long-run average value is $v(b_1)=0$, while if the initial belief is $b_1=\delta_{k_2}$, then the long-run average value is $v(b_1)=1$.
        Therefore, the long-run average value depends on the initial belief.
    \end{remark}
    
    Denote the minimum nonzero probability in the transition function by
    \begin{align*}
        p_{\min}\defas\min \left\{P_{k,k'}(i, s)\colon\; k,k' \in \K,\;i \in \I,\;s \in \S,\;P_{k,k'}(i, s) > 0\right\}.
    \end{align*}

\paragraph{Overview of results}

    The main results of this paper are summarized in the next theorem.
    
    \begin{Theorem}\label{Result: Approximating the uniform value for revealing POMDPs is decidable}
        Approximating the value for revealing POMDPs with long-run average objectives is $\mathrm{EXPTIME}$-complete. 
        In particular,
        \begin{itemize}
            \item[1.] 
                Approximating the value for revealing POMDPs with long-run average objectives is in $\mathrm{EXPTIME}$.
            \item[2.]
                Approximating the value for revealing POMDPs with long-run average objectives is $\mathrm{EXPTIME}$-hard, even for every fixed approximation error $\eps<1/2$.
        \end{itemize}
        
    \end{Theorem}
    \noindent The first item is proved in \Cref{Seection: proof upper bound} and the second item is proved in \Cref{Section: Hardness}.


\section{Application to Unreliable Sensing Architectures}
\label{Section: Application to sensing}
    
    Controller synthesis problems aim to design a controller for a stochastic environment whose state is not directly observed~\cite{kaelbling1998planning}.
    The controller learns about the state through sensors or communication channels.
    However, sensors may be unreliable in several ways~\cite{ni2009sensor}. We model erasures, in which a sensor either reports the exact value it monitors or reports nothing.
    This section describes a sensing architecture that induces a revealing POMDP with a long-run average objective.
    Finally, we explain how our approach assesses sensing architectures against their cost, which has many applications beyond the long-run average objective.

\paragraph{Controlled system}
    Consider the interaction between a controller and a stochastic system.
    The system is described by $d$ variables $u^1,\ldots,u^d$, where each variable $u^r$ takes values in a finite set $\U^r$.
    A state of the system is the vector of the current values of all the variables, so the state space is a set $\K\subseteq\U^1\times\cdots\times\U^d$.
    For example, if every variable is Boolean, then $\U^r=\{0,1\}$ for every $r$ and $\K\subseteq\{0,1\}^d$.
    At each stage:
    \begin{itemize}
        \item 
            The controller selects an action from a finite set $\I$.
        \item 
            The state is updated stochastically as a function of the current state and of the selected action.
        \item
            The stage reward $g(k,i)\in[0,1]$ measures the current performance of the system, such as production efficiency, availability, service quality, or a safety cost.
    \end{itemize}

\paragraph{Sensing architecture}
    The controller does not observe the state directly, as each variable is instead monitored by its own probabilistic sensor (or communication channel).
    At each stage, the $r$-th sensor either reports the current value of the $r$-th variable or fails and returns the erasure symbol $\erase$, in which case the controller knows that this variable was not reported.
    The signal observed by the controller is the vector of the $d$ reports.
    Hence, the set of signals satisfies $\S\subseteq\prod_{r=1}^d\left(\U^r\cup\{\erase\}\right)$.
    Formally, conditionally on the successor state $K_{m+1}=k'$, the signal $S_{m+1}=\left(S_{m+1}^1,\ldots,S_{m+1}^d\right)$ satisfies $S_{m+1}^r\in\left\{{k'}^r,\erase\right\}$ for every $r\in[1\until d]$.
    A signal therefore reveals only those components whose sensors did not fail.
    The exception is the complete signal $s_{k'}\defas\left({k'}^1,\ldots,{k'}^d\right)$, in which every sensor reports and which identifies the successor state $k'$ uniquely.
    The interaction between the controller and the system thus evolves as a POMDP $\POMDP=(\K,\I,\S,p,g)$.

\paragraph{Revealing architecture}
    We now give a structural condition on the sensors under which the POMDP becomes revealing.
    Assume that every feasible transition produces the complete signal with positive probability, that is, for every $k,k'\in\K$ and $i\in\I$,
    \[
        \sum_{s\in\S}P_{k,k'}(i,s)>0\quad\Longrightarrow\quad P_{k,k'}(i,s_{k'})>0.
    \]
    This condition holds in particular when the sensors fail independently: if, conditionally on a feasible transition from $k$ to $k'$ under action $i$, each sensor $r$ reports ${k'}^r$ with probability $\rep_r(k'\given k,i)>0$, independently of the other sensors, then
    \[
        P_{k,k'}(i,s_{k'})=\left(\sum_{s\in\S}P_{k,k'}(i,s)\right)\prod_{r=1}^d\rep_r(k'\given k,i)>0.
    \]
    Under this condition, the POMDP is revealing: whenever the transition from $k$ to $k'$ is feasible, the complete signal $s_{k'}$ occurs with positive probability and is produced only when the successor state is $k'$.
    This is a natural and practical scenario for controller synthesis against a stochastic environment: the observation is partial and the sensors fail intermittently, yet the resulting model is revealing.

\paragraph{Synthesis for a given architecture}
    The long-run average objective evaluates the persistent operating performance of the system, rather than the probability of satisfying a qualitative objective such as reachability or parity.
    By \Cref{Result: Approximating the uniform value for revealing POMDPs is decidable}, for every $\eps>0$, the value $v(b_1)$ can be approximated within $\eps$ in exponential time.
    Our result therefore provides an algorithmic method for evaluating the best long-run performance that a controller can achieve under intermittent but occasionally exact observations.

\paragraph{Choosing a sensing architecture}
    The reliability of the sensors is often a design parameter rather than a given.
    For example, more precise sensors report more often, which changes the information available to the controller and the performance it can achieve, but they also come at a higher price.
    Our result makes this trade-off quantitative.
    Let $\Theta$ be a finite collection of candidate architectures sharing the same state space $\K$, action set $\I$, and reward function $g$.
    Each architecture $\theta\in\Theta$ determines a signal set $\S_\theta$ and a transition-and-signal function $p_\theta\colon\K\times\I\to\Delta(\K\times\S_\theta)$, hence a POMDP $\POMDP_\theta=(\K,\I,\S_\theta,p_\theta,g)$.
    Assume that every $\POMDP_\theta$ is revealing, and denote by $v_\theta(b_1)$ its long-run average value from the initial belief $b_1$.
    For two architectures $\theta,\theta'\in\Theta$, the difference
    \[
        v_{\theta'}(b_1)-v_\theta(b_1)
    \]
    measures the gain in long-run performance obtained by replacing $\theta$ with $\theta'$.
    Suppose that each architecture $\theta$ carries a cost $c(\theta)\in [0,1]$.
    Therefore, replacing $\theta$ with $\theta'$ becomes profitable when the gain in performance exceeds the additional cost of the sensors, that is, 
    \[
        v_{\theta'}(b_1)-v_\theta(b_1)>c(\theta')-c(\theta).
    \]
    Our result thus answers two distinct questions: how to control the system with the sensors at hand and whether more precise sensors are worth their price.


\section{Proof of the EXPTIME Upper Bound}

\label{Seection: proof upper bound}

\paragraph{Proof sketch}
    The proof of the first item of \Cref{Result: Approximating the uniform value for revealing POMDPs is decidable} proceeds as follows:
    \begin{itemize}
        \item
            In \Cref{Section: Safetiness}, we introduce safe strategies in maximal end-components of the belief-support MDP and establish an explicit convergence rate at which their finite-horizon values converge to the corresponding long-run average values.
        \item 
            In \Cref{Section: Reduction to Reachability Objectives}, we construct a Commit POMDP and prove that its reachability value coincides with the long-run average value of the original POMDP.
        \item 
            In \Cref{Section: Proof of Theorem}, we combine the reduction with the approximation algorithm for reachability objectives and establish the $\mathrm{EXPTIME}$ upper bound.
    \end{itemize}

\subsection{Safe Values in Revealing POMDPs}\label{Section: Safetiness}
    
    This section introduces safe strategies within end-components of the belief-support MDP.
    We then establish an explicit convergence rate at which their finite-horizon values converge to the corresponding long-run average values.
    
    \subsubsection{End-components in POMDPs}
    
    \paragraph{Belief support}
        The belief support of $b\in \Delta(\K)$ is $\supp(b)\defas\{k\in \K\colon\; b(k)>0\}$.
    
    \paragraph{Belief-support MDP}
        Consider a POMDP $\POMDP$ with initial belief $b_1\in \Delta(\K)$.
        The \emph{belief-support MDP} (BS-MDP), denoted by $\MDP_B$, is defined by the tuple $\MDP_B=\left(\beliefsupport,\I,p_B,\supp(b_1)\right)$, where:
        \begin{itemize}
            \item 
                $\beliefsupport$ is the set of nonempty subsets of $\K$;
            \item 
                $\I$ is the set of actions;
            \item 
                $p_B\colon \beliefsupport\times \I\to \Delta\left(\beliefsupport\right)$ is the transition function defined by, for every $q,q'\in \beliefsupport$ and $i\in \I$, 
                \[
                    p_B(q'\given q,i)\defas \tfrac{\textbf{1}_{\left\{q'\in \Post(q,i)\right\}}}{|\Post(q,i)|},
                \]
                where $\Post$ is the set of reachable posterior supports defined by, for every $q\in \beliefsupport$ and $i\in \I$, $\Post(q,i)\defas \left\{\psi(q,i,s)\colon\; s\in \S \text{ and } \psi(q,i,s)\neq\emptyset\right\}$ with the support-update function defined by, for every $q\in \beliefsupport$, $i\in \I$, and $s\in \S$, $\psi(q,i,s)\defas\left\{k' \in \K \colon\; \exists k \in q,\; P_{k,k'}(i,s)>0\right\}$.
            \item 
                $\supp(b_1)\in \beliefsupport$ is the initial state of $\MDP_B$.
        \end{itemize}
        
    \paragraph{End-component}
        Consider an MDP $M = (\K, \I, p,g)$ with initial belief $b_1\in \Delta(\K)$.
        An end-component, denoted by $\C$, is defined by a pair $\C = (\Q, \E)$, where $\Q \subseteq \K$ is a subset of states and $\E \colon \Q \rightrightarrows \I$ assigns a nonempty set of actions to each state with the following properties:
        \begin{itemize}
            \item 
                \emph{Closedness}:
                For every state $q \in \Q$ and action $i\in \E(q)$, we have that $\supp(p(q, i)) \subseteq \Q$;
            \item
                \emph{Strong connectivity}:
                For every pair of states $q, q' \in \Q$, there exist states $q_1=q, q_2, \dots, q_n=q'$ such that, for every $m \in [1 \until n-1$], there exists an action $i_m \in \E(q_m)$ with $p(q_{m+1}\given q_m, i_m)> 0$.
                
        \end{itemize}
        The finite set of end-components is denoted by 
        \[
            \mathfrak{C}\defas\left\{\C=(\Q,\E)\colon \C \textnormal{ is an end-component of }\MDP\right\}.
        \]
    
        \paragraph{Revelation in end-component}
        Consider a POMDP and an end-component $\C=(\Q,\E)$ of the BS-MDP $\MDP_B$.
        A state $k\in \K$ can be revealed in $\C$ if $\{k\}\in \Q$.
        The set of Dirac beliefs on the states that can be revealed in $\C$ is denoted by
        \[
            \D_\C\defas\left\{\delta_k\colon\; \{k\}\in \Q\right\}.
        \]
    
        The next proposition shows that, in a revealing POMDP, every state occurring in some belief-support of an end-component can be revealed in that end-component.
        
        \begin{Proposition}\label{Result: revelation in end-component}
            Consider a revealing POMDP $\POMDP$ and its BS-MDP $\MDP_B$.
            Fix an end-component $\C=(\Q,\E)$ of $\MDP_B$.
            Then, for every $q\in \Q$ and $k\in q$, we have that $\{k\}\in \Q$.
            In particular,
            \[
                \bigcup_{q\in \Q}q=\left\{k\in \K\colon\; \{k\}\in \Q\right\}\qquad\text{ and }\qquad\D_\C\neq \emptyset.
            \]
        \end{Proposition}
        
        \begin{proof}[Proof of Proposition~\ref{Result: revelation in end-component}]
            Consider a revealing POMDP $\POMDP$ and its BS-MDP $\MDP_B$.
            Fix an end-component $\C=(\Q,\E)$ of $\MDP_B$, a belief-support $q\in \Q$, and a state $k\in q$.
            We first prove that
            \begin{align}
                \bigcup_{q\in \Q}q\subseteq \left\{k\in \K\colon\; \{k\}\in \Q\right\}.
                \label{equation first inclusion}
            \end{align}
    
            We have that there exist a belief-support $\another{q}\in \Q$ and an action $\another{i}\in \E(\another{q})$ such that $q\in \Post(\another{q},\another{i})$.
            Indeed, if $|\Q|>1$, then strong connectivity yields a path in $\Q$ ending at $q$.
            If $\Q=\{q\}$, then the claim follows from closedness and the nonemptiness of $\Post(q,i)$ for every $i\in\E(q)$.
                
            Since $q\in \Post(\another{q},\another{i})$, there exists $s\in \S$ such that $q=\psi(\another{q},\another{i},s)$.
            Since $k\in q$, the definition of $\psi$ gives that there exists $\another{k}\in \another{q}$ such that $P_{\another{k},k}(\another{i},s)>0$ and $\sum_{s'\in \S}P_{\another{k},k}(\another{i},s')>0$.
            Therefore, by the revealing property, there exists a signal $s^*\in \S$ such that
            $\psi(\another{q},\another{i},s^*) = \{k\}$ and thus $\{k\}\in \Post(\another{q},\another{i})$.
            Since $\C$ is closed and $\another{i}\in \E(\another{q})$, we get that $\{k\}\in \Q$.
            Since $q\in \Q$ and $k\in q$ were arbitrary, this proves \eqref{equation first inclusion}.
            The reverse inclusion holds trivially because $k\in \{k\}$ for every $k\in \K$ such that $\{k\}\in \Q$.
            Finally, since $\Q\neq \emptyset$ and every belief-support in $\beliefsupport$ is nonempty, there exist $q\in \Q$ and $k\in q$; hence, $\{k\}\in \Q$ and $\delta_k\in \D_\C$, which concludes the proof.
        \end{proof}
    
    \paragraph{Maximal end-components}
        Given two end-components $\C=(\Q,\E)$ and $\another{\C}=(\another{\Q},\another{\E})$ of the BS-MDP $\MDP_B$, write $\C\preceq\another{\C}$ if
        \[
            \Q\subseteq\another{\Q}
            \qquad\text{and}\qquad
            \E(q)\subseteq\another{\E}(q)
            \quad\text{for every }q\in\Q.
        \]
        An end-component $\C=(\Q,\E)$ is \emph{maximal} if there exists no end-component $\another{\C}=(\another{\Q},\another{\E})$ such that $\C\preceq\another{\C}$ and $\C\neq\another{\C}$.
        Denote the set of maximal end-components of $\MDP_B$ by $\mathfrak{M}$.\\
        
        The next lemma shows that maximal end-components are stable under revelations.
    
        \begin{Lemma}
            \label{Result: safe actions on singleton supports}
            Consider a revealing POMDP $\POMDP$ and its BS-MDP $\MDP_B$.
            Fix a maximal end-component $\C=(\Q,\E)\in\mathfrak{M}$.
            Then, for every $q\in\Q$, $k\in q$, and $i\in\E(q)$, we have that $\Post(\{k\},i)\subseteq\Q$ and $i\in\E(\{k\})$.
        \end{Lemma}
        
        \begin{proof}[Proof of Lemma~\ref{Result: safe actions on singleton supports}]
            Consider a revealing POMDP $\POMDP$ and its BS-MDP $\MDP_B$.
            Fix a maximal end-component $\C=(\Q,\E)\in\mathfrak{M}$, a belief-support $q\in\Q$, a state $k\in q$, and an action $i\in\E(q)$.
            By Proposition~\ref{Result: revelation in end-component}, we have that $\{k\}\in\Q$.
    
            Define the downward closure of $\Q$ by
            \[
                \overline{\Q}\defas\left\{r\in\beliefsupport\colon\;\exists q'\in\Q,\ r\subseteq q'\right\}.
            \]
            For every $r\in\overline{\Q}\setminus\Q$, fix a belief-support $q_r\in\Q$ such that $r\subseteq q_r$ and an action $i_r\in\E(q_r)$.
            Define the action correspondence $\overline{\E}$ on $\overline{\Q}$ by
            \[
                \overline{\E}(r)\defas
                \begin{cases}
                    \E(\{k\})\cup\{i\},
                        &\text{if }r=\{k\},\\
                    \E(r),
                        &\text{if }r\in\Q\setminus\{\{k\}\},\\
                    \{i_r\},
                        &\text{if }r\in\overline{\Q}\setminus\Q.
                \end{cases}
            \]
    
            We first show that, for every $r\in\overline{\Q}$ and every $\another{i}\in\overline{\E}(r)$,
            \begin{equation}
                \Post(r,\another{i})\subseteq\overline{\Q}.
                \label{equation: alternative proof downward closure}
            \end{equation}
            If $r\in\Q$ and $\another{i}\in\E(r)$, this follows from the closedness of $\C$.
            It remains to consider the added actions.
            First, fix $r'\in\Post(\{k\},i)$ and a signal $s\in\S$ such that $r'=\psi(\{k\},i,s)$.
            Since $\{k\}\subseteq q$, monotonicity of $\psi$ gives that $r'\subseteq\psi(q,i,s)$.
            Since $\Post(q,i)\subseteq\Q$, we have $r'\in\overline{\Q}$.
            Second, fix $r\in\overline{\Q}\setminus\Q$, $r'\in\Post(r,i_r)$, and a signal $s\in\S$ such that $r'=\psi(r,i_r,s)$.
            Since $r\subseteq q_r$, monotonicity of $\psi$ and closedness of $\C$ give that $r'\subseteq\psi(q_r,i_r,s)\in\Post(q_r,i_r)\subseteq\Q$.
            Hence, $r'\in\overline{\Q}$, which proves equation~\eqref{equation: alternative proof downward closure}.
    
            Let $\R$ be the set of belief-supports reachable from $\{k\}$ using actions prescribed by $\overline{\E}$ and let $\E_{\R}$ be the restriction of $\overline{\E}$ to $\R$.
            Since $\{k\}\in\Q$ and $\C$ is strongly connected, every belief-support in $\Q$ is reachable from $\{k\}$ using actions prescribed by $\E$.
            Moreover, the definition of $\overline{\E}$ gives that $\E(r)\subseteq\overline{\E}(r)$ for every $r\in\Q$.
            Therefore, $\Q\subseteq\R$, and hence $\E(r)\subseteq\E_{\R}(r)$ for every $r\in\Q$.
    
            We show that $(\R,\E_{\R})$ is an end-component of $\MDP_B$.
            First, it is closed.
            Indeed, if $r\in\R$, $\another{i}\in\E_{\R}(r)$, and $r'\in\Post(r,\another{i})$, then $r'\in\overline{\Q}$ by \eqref{equation: alternative proof downward closure} and $r'$ is reachable from $\{k\}$ by the definition of $\R$.
            Therefore, $r'\in\R$.
    
            We next show that $(\R,\E_{\R})$ is strongly connected.
            By definition of $\R$, every belief-support $r\in\R$ is reachable from $\{k\}$.
            Conversely, fix $r\in\R$.
            If $r\in\Q$, then $r$ can reach $\{k\}$ using actions prescribed by $\E_{\R}$ because $\C$ is strongly connected and $\E(r')\subseteq\E_{\R}(r')$ for every $r'\in\Q$.
            Suppose that $r\in\R\setminus\Q$.
            Fix a state $\another{k}\in r$.
            By the revealing property, there exist $k'\in\K$ and $s^*\in\S$ such that $P_{\another{k},k'}(i_r,s^*)>0$ and $s^*$ reveals $k'$.
            Since $\another{k}\in r\subseteq q_r$, we obtain $\psi(r,i_r,s^*)=\psi(q_r,i_r,s^*)=\{k'\}$ and thus, $\{k'\}\in\Post(q_r,i_r)\subseteq\Q$.
            Since $i_r\in\E_{\R}(r)$, the belief-support $r$ has an edge to $\{k'\}\in\Q$, which can reach $\{k\}$ using actions prescribed by $\E_{\R}$ because $\C$ is strongly connected and $\E(r')\subseteq\E_{\R}(r')$ for every $r'\in\Q$.
            Therefore, every belief-support in $\R$ can reach $\{k\}$, which proves that $(\R,\E_{\R})$ is strongly connected.
    
            Therefore, $\C\preceq(\R,\E_{\R})$.
            Since $(\R,\E_{\R})$ is an end-component and $\C$ is maximal, we have that $\R=\Q$ and $\E_{\R}(r)=\E(r)$ for every $r\in \Q$.
            Since $i\in\overline{\E}(\{k\})$ and $\{k\}\in\R$, the definition of $\E_{\R}$ gives that $i\in\E_{\R}(\{k\})=\E(\{k\})$.
            Moreover, the definition of $\R$ gives that $\Post(\{k\},i)\subseteq\R=\Q$.
            Hence, $i\in\E(\{k\})$ and $\Post(\{k\},i)\subseteq\Q$, which concludes the proof.
        \end{proof}
    
        The next lemma shows that, in a maximal end-component $\C=(\Q,\E)$, any two singleton belief-supports of $\Q$ are connected by a path of length at most $|\K|-1$ that visits only singleton belief-supports of $\Q$ and uses only actions prescribed by $\E$.
        
        \begin{Lemma}
            \label{Result: short paths between singleton supports}
            Consider a revealing POMDP $\POMDP$ and its BS-MDP $\MDP_B$.
            Fix $\C=(\Q,\E)\in\mathfrak{M}$ and states $k,k'\in\K$ such that $\{k\},\{k'\}\in\Q$.
            Then, there exist $r\leq|\K|-1$, $k_0=k,k_1,\ldots,k_r=k'$ with $\{k_m\}\in\Q$ for every $m\in[0\until r]$, $i_0,\ldots,i_{r-1}\in\I$, and $s_1,\ldots,s_r\in\S$ such that, for every $m\in[0\until r-1]$, $i_m\in\E(\{k_m\})$ and $\psi(\{k_m\},i_m,s_{m+1})=\{k_{m+1}\}$.
            In particular, $P_{k_m,k_{m+1}}(i_m,s_{m+1})>0$ for every $m\in[0\until r-1]$.
        \end{Lemma}
        
        \begin{proof}[Proof of Lemma~\ref{Result: short paths between singleton supports}]
            Consider a revealing POMDP $\POMDP$ and its BS-MDP $\MDP_B$.
            Fix a maximal end-component $\C=(\Q,\E)\in\mathfrak{M}$ and define $\Q_{\mathrm{sing}}\defas\left\{\{k\}\in\Q\;\middle|\;k\in\K\right\}$.
            Fix $k,k'\in\K$ such that $\{k\},\{k'\}\in\Q_{\mathrm{sing}}$.
            Because $\C$ is strongly connected, there exists an integer $\ell\in\NN$, supports
            \[
                q_0=\{k\},q_1,\ldots,q_\ell=\{k'\},
            \]
            actions $i_0,\ldots,i_{\ell-1}\in\I$, and signals $\another{s}_1,\ldots,\another{s}_\ell\in\S$ such that, for every $m\in[0\until\ell-1]$,
            \[
                i_m\in\E(q_m)
                \qquad\text{and}\qquad
                q_{m+1}=\psi(q_m,i_m,\another{s}_{m+1}).
            \]
        
            We construct states $k_\ell,k_{\ell-1},\ldots,k_0$ backwards such that $k_\ell=k'$ and, for every $m\in[0\until\ell]$, $k_m\in q_m$ and, if $m\le \ell-1$, $P_{k_m,k_{m+1}}(i_m,\another{s}_{m+1})>0$.
            We initialize the construction by setting $k_\ell\defas k'$ with $q_\ell=\{k'\}$.
            For the recursive step, fix $m\in[0\until \ell-1]$ and suppose that the states $k_\ell,\ldots,k_{m+1}$ have already been constructed, so that $k_{m+1}\in q_{m+1}$.
            Since $q_{m+1}=\psi(q_m,i_m,\another{s}_{m+1})$, the definition of $\psi$ gives a state $k_m\in q_m$ such that $P_{k_m,k_{m+1}}(i_m,\another{s}_{m+1})>0$, and we fix such a state $k_m$.
            After $\ell$ steps, the construction produces a state $k_0\in q_0$.
            Since $q_0=\{k\}$, we get that $k_0=k$.
        
            Moreover, by Lemma~\ref{Result: safe actions on singleton supports}, we have that, for every $m\in[0\until\ell-1]$, $i_m\in\E(\{k_m\})$ and $\sum_{s\in\S}P_{k_m,k_{m+1}}(i_m,s)>0$.
            Because $\POMDP$ is revealing, there exists a signal $s_{m+1}\in\S$ such that $P_{k_m,k_{m+1}}(i_m,s_{m+1})>0$ and $\psi(\{k_m\},i_m,s_{m+1})=\{k_{m+1}\}$.
            Therefore, we have constructed a path from $\{k\}$ to $\{k'\}$ in $\Q_{\mathrm{sing}}$.
        
            If this path visits a singleton support more than once, we delete the portion between two consecutive occurrences of the same singleton support.
            Repeating this operation yields a path whose singleton supports are pairwise distinct.
            Therefore, we deduce that $r\leq|\Q_{\mathrm{sing}}|-1\leq|\K|-1$, which concludes the proof.
        \end{proof}
    
    \subsubsection{Approximating the Maximum Safe Value}
    
        This subsection introduces the strategies that keep the belief-support inside a fixed end-component of the BS-MDP, called $\C$-safe strategies, and the maximum safe value.
        Using an asynchronous coupling, we prove for maximal end-components that the finite-horizon values of such strategies converge, at an explicit rate, to a long-run average value that depends on the end-component alone, and we deduce that the maximum safe value can be approximated in $\mathrm{EXPTIME}$.
    
        \paragraph{$\C$-Safe strategy}
            Consider a POMDP $\POMDP$ and its BS-MDP $\MDP_B$ with end-component $\C=(\Q,\E)$.
            The set of $\C$-safe beliefs is defined by
            \[
                \Delta_\C\defas\left\{b\in \Delta(\K)\colon\; \supp(b)\in \Q\right\}.
            \]
            Given an initial belief and a strategy, denote the support of the belief process at stage $m\in \NN^*$ by $Q_m\defas\supp(B_m)$.
            A strategy $\sigma$ is called  $\C$\emph{-safe} from $b_1 \in \Delta_\C$ if
            \[
                \PP_\sigma^{b_1}\left(\forall m \in \NN^*,\; Q_m\in \Q \text{ and } I_m\in \E(Q_m)\right) = 1.
            \]
            The set of $\C$-safe strategies from $b_1 \in \Delta_\C$, denoted by $\Sigma_\C(b_1)$, is defined by
            \[
                \Sigma_\C(b_1)\defas\left\{\sigma \in \Sigma\colon\;    \PP_\sigma^{b_1}\left(\forall m \in \NN^*,\; Q_m\in \Q \text{ and } I_m\in \E(Q_m)\right) = 1\right\}.
            \]
    
        \paragraph{$\C$-Safe objectives}
            Consider a POMDP $\POMDP$ and its BS-MDP $\MDP_B$.
            Consider an end-component $\C=(\Q,\E)$ of $\MDP_B$ and an initial belief $b_1\in \Delta_\C$.
            \begin{itemize}
                \item 
                    The $\C$-safe $n$-stage (with $n\in \NN^*$) objective given by the $\C$-safe strategy $\sigma\in \Sigma_\C(b_1)$ is defined by
                    \[
                        \gamma_{n,\C}(b_1,\sigma)\defas \EE_\sigma^{b_1}\left(\dfrac{1}{n}\sum_{m=1}^nG_m\right).
                    \]
                    The $\C$-safe $n$-stage value is defined by $v_{n,\C}(b_1)\defas\sup\limits_{\sigma\in \Sigma_\C(b_1)}\gamma_{n,\C}(b_1,\sigma)$.
                    
                \item 
                    The $\C$-safe long-run average objective given by the $\C$-safe strategy $\sigma\in \Sigma_\C(b_1)$ is defined by
                    \[
                        \gamma_\C(b_1,\sigma)\defas \liminf_{n\to\infty}\EE_\sigma^{b_1}\left(\dfrac{1}{n}\sum_{m=1}^n G_m\right).
                    \]
                    The $\C$-safe value is defined by $v_\C(b_1)\defas\sup\limits_{\sigma\in \Sigma_\C(b_1)}\gamma_\C(b_1,\sigma)$.
            \end{itemize}
            
        \paragraph{Maximum safe value}
            \noindent Given $k\in \K$, define the maximum safe value by
            \[
                v_{\safe}(k)\defas\max\left\{v_\C(\delta_k)\colon \C=(\Q,\E)\in \mathfrak{C} \text{ and } \{k\}\in \Q\right\}.
            \]
        
        \paragraph{Main Result}    
            The main result of this section is the following theorem.
            \begin{Theorem}
                \label{Theorem: approximation of max safe value is EXPTIME}
                Approximating the maximum safe value is in EXPTIME.
            \end{Theorem}
        
            In the next lemma, we first show that the $\C$-safe $n$-stage values from two initial beliefs in $\Delta_\C$ become uniformly close as the horizon grows.
            
            \begin{Lemma}
                \label{Result: n-stage Safe value}
                Consider a revealing POMDP $\POMDP$ and its BS-MDP $\MDP_B$.
                Fix a maximal end-component $\C=(\Q,\E)\in \mathfrak{M}$ of $\MDP_B$ and $\eps\in (0,1)$, and define the horizon
                \begin{equation}
                    n_\eps\defas\left\lceil\dfrac{6}{\eps}\left(1+\left(\left\lceil\dfrac{\log(3/\eps)}{p_{\min}}\right\rceil+|\K|\right)\left\lceil\dfrac{3\log(3/\eps)}{2p_{\min}^{|\K|}}\right\rceil\right)\right\rceil.
                    \label{equation: explicit horizon}
                \end{equation}
                Then, for every $n\geq n_\eps$ and pair of initial beliefs $b_1,b_1'\in \Delta_\C$,
                \begin{equation}
                    \left|v_{n,\C}(b_1)-v_{n,\C}(b_1')\right|\le \eps.
                    \label{Equation: n-stage safe value}
                \end{equation}
            \end{Lemma}
    
            \paragraph{Proof overview of Lemma~\ref{Result: n-stage Safe value}}
                The two initial beliefs are compared through an asynchronous coupling of two copies of $\POMDP$.
                Starting from $b_1\in \Delta_\C$, the first copy is played according to an arbitrary $\C$-safe strategy.
                Starting from $b_1'\in \Delta_\C$, the second copy is played according to a $\C$-safe strategy that we construct.
                Once a state is revealed in each copy, the second copy follows a path inside $\C$ to reach the state revealed in the first copy.
                From that stage on, both copies use the same continuation strategy and obtain the same rewards.
                Therefore, only the stages before this synchronization contribute to the difference of the two $n$-stage values.
                Since the synchronization occurs before an explicit stage with high probability, this contribution vanishes at an explicit rate.
                Accordingly, the proof first establishes the revelation property and a uniform lower bound on the probability of reaching a given revealed belief within an explicit number of stages, then constructs the coupling strategy from $b_1'$ and the coupling itself, and finally compares the two $n$-stage values.
    
            \begin{proof}[Proof of Lemma~\ref{Result: n-stage Safe value}]
            
                Consider a revealing POMDP $\POMDP$ and its BS-MDP $\MDP_B$.
                Fix a maximal end-component $\C=(\Q,\E)\in \mathfrak{M}$ of $\MDP_B$ and $\eps\in (0,1)$.
                Fix an arbitrary pair of initial beliefs $b_1, b_1'\in \Delta_\C$ and a $\C$-safe strategy $\sigma\in \Sigma_\C(b_1)$.
            
                \vspace{1em}
                \smallskip\noindent\textit{Revelation property.}
                    Since $\sigma$ is $\C$-safe, we have that the support of the belief process remains inside $\C$, i.e., $Q_m\in \Q$ and $I_m\in \E(Q_m)$ for every $m\in \NN^*$.
                    Fix a stage $m\in \NN^*$ and an admissible history $h_m\in \H_m(b_1)$ with $\PP_\sigma^{b_1}(H_m=h_m)>0$.
                    Consider an action $i\in \I$ such that $\sigma(i\given h_m)>0$.
                    Since $\sigma$ is $\C$-safe, we have $i\in \E\left(\supp(b_{h_m}^{b_1})\right)$.
                    Fix a state $k\in \supp\left(b_{h_m}^{b_1}\right)$.
                    Since $p(k,i)\in \Delta(\K\times \S)$, there exists a successor state $k'\in \K$ such that $\sum_{s\in \S}P_{k,k'}(i,s)>0$.
                    Since $\POMDP$ is revealing, there exists at least one signal $s^*\in \S$ that reveals $k'$, i.e.,
                    \[
                        P_{k,k'}(i,s^*)\ge p_{\min}>0
                        \qquad\text{and}\qquad
                        P_{\underline{k},\overline{k}}(i,s^*)=0
                            \quad\text{for every }\underline{k}\in \K\text{ and }\overline{k}\in \K\setminus\{k'\}.
                    \]
                    Therefore, we have $\psi\left(\supp(b_{h_m}^{b_1}),i,s^*\right)=\{k'\}$.
                    Since $i\in \E\left(\supp(b_{h_m}^{b_1})\right)$ and $\C$ is closed, we get that $\{k'\}\in \Q$ and $\delta_{k'}\in \D_\C$.
                    Consequently, on the event $\{K_m=k,H_m=h_m\}$, observing the signal $s^*$ at stage $m$ leads to the revealed belief $\delta_{k'}$, and thus
                    \[
                        \PP_\sigma^{b_1}\left(B_{m+1}\in \D_\C\givenm K_m=k, H_m=h_m,I_m=i\right)\ge P_{k,k'}(i,s^*)\ge p_{\min}.
                    \]
                    Averaging over the actions prescribed by $\sigma$ and over the states of the support of $b_{h_m}^{b_1}$, we deduce that, for every $b_1\in \Delta_\C$ and admissible history $h_m\in \H_m(b_1)$ with $m\in \NN^*$,
                    \begin{align}
                        \PP_\sigma^{b_1}\left(B_{m+1}\in \D_\C\givenm H_m=h_m\right)
                            &\geq \sum_{i\in \I}\sigma(i\given h_m)\sum_{k\in \K}b_{h_m}^{b_1}(k)p_{\min}
                            =p_{\min}.
                        \label{equation: alternative one-step revelation}
                    \end{align}
                    Since inequality \eqref{equation: alternative one-step revelation} holds for every initial belief in $\Delta_\C$, every $\C$-safe strategy from that belief, and every admissible history, we obtain, by induction on $n\in \NN$, that, for every stage $m\in \NN^*$,
                    \begin{align}
                        \PP_\sigma^{b_1}\left(\bigcap_{j=1}^{n}\{B_{m+j}\notin \D_\C\}\givenm H_m\right)\le (1-p_{\min})^{n}.
                        \label{equation: alternative geometric revelation}
                    \end{align}
                    Indeed, for the base case $n=0$, the intersection over the empty index set is $\Plays$, and thus both sides of \eqref{equation: alternative geometric revelation} are equal to $1$.
                    Assume now that \eqref{equation: alternative geometric revelation} holds for some $n\in \NN$.
                    We have that
                    \begin{align*}
                        &\PP_\sigma^{b_1}\left(\bigcap_{j=1}^{n+1}\{B_{m+j}\notin \D_\C\}\givenm H_m\right)\\
                            &\qquad = \EE_\sigma^{b_1}\left(\mathbf{1}_{\bigcap_{j=1}^{n}\{B_{m+j}\notin \D_\C\}}\PP_\sigma^{b_1}\left(B_{m+n+1}\notin \D_\C\givenm H_{m+n}\right)\givenm H_m\right)
                                &&\text{(tower rule)}\\
                            &\qquad \le (1-p_{\min})\PP_\sigma^{b_1}\left(\bigcap_{j=1}^{n}\{B_{m+j}\notin \D_\C\}\givenm H_m\right)
                                &&\text{(by Eq.~\eqref{equation: alternative one-step revelation})}\\
                            &\qquad \le (1-p_{\min})^{n+1},
                                &&\text{(induction hyp.)}
                    \end{align*}
                    which proves inequality \eqref{equation: alternative geometric revelation}.
            
                    Define 
                    \begin{equation}
                        \ell_\eps\defas
                            \begin{cases}
                                1,
                                    &\text{if } p_{\min} = 1,\\
                                \left\lceil\dfrac{\log(\eps/3)}{\log(1-p_{\min})}\right\rceil,
                                    &\text{if } 0 < p_{\min} < 1.
                            \end{cases}
                            \label{equation: alternative waiting length}
                    \end{equation}
                    Then, $(1-p_{\min})^{\ell_\eps}\le \eps/3$.
                    Define the first revelation time by
                    \[
                        T_{\rev}\defas\inf\{m\in \NN^*\colon\; B_m\in \D_\C\}.
                    \]
                    Since $\{T_{\rev}>1+n\}\subseteq \bigcap_{j=1}^{n}\{B_{1+j}\notin \D_\C\}$ for every $n\in \NN$, we obtain that
                    \begin{align}
                        \PP_\sigma^{b_1}(T_{\rev}>1+\ell_\eps)
                            & \le \PP_\sigma^{b_1}\left(\bigcap_{j=1}^{\ell_\eps}\{B_{1+j}\notin \D_\C\}\right)
                                &&\text{(def. of $T_{\rev}$)}\nonumber\\
                            & = \EE_\sigma^{b_1}\left(\PP_\sigma^{b_1}\left(\bigcap_{j=1}^{\ell_\eps}\{B_{1+j}\notin \D_\C\}\givenm H_1\right)\right)
                                &&\text{(tower rule)}\nonumber\\
                            & \le (1-p_{\min})^{\ell_\eps}
                                &&\text{(by Eq.~\eqref{equation: alternative geometric revelation})}\nonumber\\
                            & \le \dfrac{\eps}{3}.
                                &&\text{(def. of $\ell_\eps$)}
                        \label{equation: alternative first revelation}
                    \end{align}
                    Moreover, letting $n$ tend to infinity in the same bound, we get that
                    \begin{align*}
                        \PP_\sigma^{b_1}(T_{\rev}=\infty)
                            & = \lim_{n\to\infty}\PP_\sigma^{b_1}(T_{\rev}>1+n)
                                &&\text{(continuity from above)}\\
                            & \le \lim_{n\to\infty}(1-p_{\min})^{n}
                                &&\text{(by Eq.~\eqref{equation: alternative geometric revelation})}\\
                            & = 0.
                                &&\text{($p_{\min}>0$)}
                    \end{align*}
                    In particular, we have that $\PP_\sigma^{b_1}(T_{\rev}<\infty)=1$.
                    On the event $\{T_{\rev}<\infty\}$, denote by $H_{\rev}\defas H_{T_{\rev}}$ the \emph{first revealed history} and by
                    \[
                        \H_{\rev}^{\sigma}(b_1)\defas\left\{h\in \bigcup_{m\in \NN^*}\H_m(b_1)\colon\; \PP_\sigma^{b_1}\left(H_{\rev}=h\right)>0\right\}
                    \]
                    the set of its realizations.
                    Every realization $h_\rev\in \H_{\rev}^{\sigma}(b_1)$ reveals a unique state $k_\rev\in \K$ with $\{k_\rev\}\in \Q$ such that $b_{h_\rev}^{b_1}=\delta_{k_\rev}$.
                    Denote by $\mu$ the law of the first revealed history, i.e., $\mu(h_\rev)\defas\PP_{\sigma}^{b_1}\left(H_{\rev}=h_\rev\right)$.
                    Throughout the proof, every sum indexed by $h_\rev$ ranges over $\H_{\rev}^{\sigma}(b_1)$, and thus $\sum_{h_\rev}\mu(h_\rev)=\PP_\sigma^{b_1}(T_{\rev}<\infty)=1$.
                    
                \vspace{1em}
                \smallskip\noindent\textit{Uniform reset lower bound.}
                    Fix a pair of states $k,k'\in \K$ with $\{k\},\{k'\}\in \Q$ and let $r_{k,k'}$ denote the length of a shortest belief-support path in $\C$ from $\{k\}$ to $\{k'\}$.
                    Fix such a path
                    \[
                        q_0=\{k\},q_1,\ldots,q_{r_{k,k'}}=\{k'\},
                    \]
                    together with actions $i_0,\ldots,i_{r_{k,k'}-1}$ and signals $s_1,\ldots,s_{r_{k,k'}}$ satisfying $q_m\in \Q$, $i_m\in \E(q_m)$, and $q_{m+1}=\psi(q_m,i_m,s_{m+1})$ for every $m\in [0\until r_{k,k'}-1]$.
                    Choosing the states along the path backwards from $k_{r_{k,k'}}\defas k'$, the definition of $\psi$ provides states $k_0,\ldots,k_{r_{k,k'}}$ such that $k_0=k$, $k_{r_{k,k'}}=k'$, and, for every $m\in [0\until r_{k,k'}-1]$,
                    \[
                        k_m\in q_m
                        \qquad\text{and}\qquad
                        P_{k_m,k_{m+1}}(i_m,s_{m+1})>0.
                    \]
                    Since every nonzero entry of the transition function is at least $p_{\min}$, the probability of success of such a path is at least
                    \begin{align}
                        \prod_{m=0}^{r_{k,k'}-1}P_{k_m,k_{m+1}}(i_m,s_{m+1})\ge p_{\min}^{r_{k,k'}}.
                        \label{equation: alternative path probability}
                    \end{align}
                    Since $\C$ is maximal, Lemma~\ref{Result: short paths between singleton supports} gives such a path made of singleton belief-supports such that $r_{k,k'}\le |\K|-1\le |\K|$.
    
                    We choose the following parameters for the rest of the proof: the trial length $L_\eps$, the number of trials $N_\eps$, and the reset horizon $m_{\rev}$ defined by
                    \begin{equation}
                        L_\eps\defas \ell_\eps+|\K|,
                        \qquad
                        N_\eps\defas\left\lceil\dfrac{3\log(3/\eps)}{2p_{\min}^{|\K|}}\right\rceil,
                        \qquad
                        m_{\rev}\defas 1+L_\eps N_\eps.
                        \label{equation: alternative reset horizon}
                    \end{equation}
                    We partition the horizon into trials of length $L_\eps$, i.e., trial $j\in \NN^*$ is the block of stages $[m_j\until m_{j+1}-1]$ with $m_j\defas 1+(j-1)L_\eps$.
                    In particular, $m_{j+1}=m_j+L_\eps$ for every $j\in \NN^*$ and $m_{N_\eps+1}=m_{\rev}$.
            
                    Fix a first revealed history $h_\rev\in \H_{\rev}^{\sigma}(b_1)$ and recall that $k_\rev$ denotes the state it reveals.
                    We construct a $\C$-safe strategy $\sigma_{h_\rev}$ from the second belief $b_1'\in \Delta_\C$, whose objective is to reach the same revealed belief $\delta_{k_\rev}$.
                    The strategy $\sigma_{h_\rev}$ proceeds trial by trial as follows:
                    \begin{itemize}
                        \item
                            If, at the beginning of trial $j$, the current belief is a revealed belief $\delta_k\in \D_\C$, then $\sigma_{h_\rev}$ starts a reset attempt toward $\delta_{k_\rev}$, i.e., it plays the actions $i_0,\ldots,i_{r_{k,k_\rev}-1}$ along the fixed path from $\{k\}$ to $\{k_\rev\}$, as long as the observed belief-supports agree with the sequence $q_0,\ldots,q_{r_{k,k_\rev}}$.
                        \item
                            Otherwise, $\sigma_{h_\rev}$ enters a waiting phase, in which it plays $\C$-safe actions for at most $\ell_\eps$ stages, until a revealed belief $\delta_{k}\in \D_\C$ is observed; in the latter case, it starts a reset attempt toward $\delta_{k_\rev}$ from $\delta_k$, as above.
                        \item
                            If no revealed belief is observed during the $\ell_\eps$ stages of the waiting phase, or if the observed belief-supports leave the fixed path during a reset attempt, then $\sigma_{h_\rev}$ plays $\C$-safe actions until the end of trial $j$, i.e., until stage $m_{j+1}$, at which a new trial begins.
                        \item
                            As soon as the target belief $\delta_{k_\rev}$ is reached, $\sigma_{h_\rev}$ switches to the continuation strategy $\sigma[h_\rev]$ of the first copy.
                    \end{itemize}
                    Since $\sigma\in \Sigma_\C(b_1)$ and every action prescribed along a path of $\C$ or during a waiting phase is $\C$-safe, carrying out the same construction from any $b\in\Delta_\C$ yields a strategy, still denoted by $\sigma_{h_\rev}$, in $\Sigma_\C(b)$.
                    Define the first time at which the target state $k_\rev$ is revealed by
                    \[
                        T_{k_\rev}\defas \inf\left\{m\in \NN^*\colon\; B_m=\delta_{k_\rev}\right\}.
                    \]
                    We prove that, for every $b\in \Delta_\C$,
                    \begin{align}
                        \PP^{b}_{\sigma_{h_\rev}}(T_{k_\rev}> m_{\rev})\le \left(1-\left(1-\eps/3\right)p_{\min}^{|\K|}\right)^{N_\eps}\le \dfrac{\eps}{3}.
                        \label{equation: alternative reset bound}
                    \end{align}
            
                    We first establish the following one-trial inequality.
                    For every trial index $j\in \NN^*$ and admissible history $h_{m_j}\in \H_{m_j}(b)$ such that $\PP_{\sigma_{h_\rev}}^{b}\left(H_{m_j}=h_{m_j},T_{k_\rev}>m_j\right)>0$,
                    \begin{align}
                        \PP_{\sigma_{h_\rev}}^{b}\left(T_{k_\rev}\le m_{j+1}\givenm H_{m_j}=h_{m_j},T_{k_\rev}>m_j\right)\ge \left(1-\eps/3\right)p_{\min}^{|\K|}.
                        \label{equation: alternative one-trial bound}
                    \end{align}
                    Note that the event $\{T_{k_\rev}>m_j\}$ is determined by the history $H_{m_j}$.
                    Since the target belief $\delta_{k_\rev}$ has not been reached before stage $m_j$, the strategy $\sigma_{h_\rev}$ is, at stage $m_j$, in one of the following two cases:
                    \begin{itemize}
                        \item
                            \emph{A reset attempt starts at stage $m_j$.}
                            Suppose that $b_{h_{m_j}}^{b}=\delta_k$ for some revealed belief $\delta_k\in \D_\C$.
                            Since $T_{k_\rev}>m_j$, we have that $k\neq k_\rev$.
                            By construction of $\sigma_{h_\rev}$, the reset attempt reaches the target belief $\delta_{k_\rev}$ whenever the successive states and signals follow the fixed path from $\{k\}$ to $\{k_\rev\}$, which takes $r_{k,k_\rev}\le |\K|\le L_\eps$ stages and thus terminates before stage $m_{j+1}=m_j+L_\eps$.
                            Therefore, we deduce that
                            \begin{align*}
                                &\PP_{\sigma_{h_\rev}}^{b}\left(T_{k_\rev}\le m_{j+1}\givenm H_{m_j}=h_{m_j},T_{k_\rev}>m_j\right)\\
                                    &\qquad \ge \prod_{a=0}^{r_{k,k_\rev}-1}P_{k_a,k_{a+1}}(i_a,s_{a+1})
                                        &&\text{(construction of $\sigma_{h_\rev}$)}\\
                                    &\qquad \ge p_{\min}^{r_{k,k_\rev}}
                                        &&\text{(by Eq.~\eqref{equation: alternative path probability})}\\
                                    &\qquad \ge p_{\min}^{|\K|}\ge \left(1-\eps/3\right)p_{\min}^{|\K|},
                                        &&\text{($r_{k,k_\rev}\le |\K|$)}
                            \end{align*}
                            and thus inequality \eqref{equation: alternative one-trial bound} follows.
                        \item
                            \emph{A waiting phase starts at stage $m_j$.}
                            Suppose that $b_{h_{m_j}}^{b}\notin \D_\C$.
                            The controller plays a $\C$-safe strategy until a revealed belief in $\D_\C$ is observed, for at most $\ell_\eps$ stages.
                            By Eq.~\eqref{equation: alternative geometric revelation} applied with $(m,n)=(m_j,\ell_\eps)$, the waiting phase ends with a revealed belief with high probability, i.e.,
                            \begin{align}
                                \PP_{\sigma_{h_\rev}}^{b}\left(\bigcup_{r=1}^{\ell_\eps}\{B_{m_j+r}\in \D_\C\}\givenm H_{m_j}=h_{m_j},T_{k_\rev}>m_j\right)
                                &\ge 1-(1-p_{\min})^{\ell_\eps}\nonumber\\
                                &\ge 1-\dfrac{\eps}{3}.
                                \label{equation: alternative waiting bound}
                            \end{align}
                            We decompose this union into the pairwise disjoint events indexed by the first stage $m_j+r$, with $r\in [1\until \ell_\eps]$, at which a revealed belief is observed, and by the revealed state.
                            On the event where $B_{m_j+r}=\delta_{k_\rev}$, we already have that $T_{k_\rev}\le m_j+r\le m_{j+1}$.
                            On the event where $B_{m_j+r}=\delta_k$ with $k\neq k_\rev$, the strategy $\sigma_{h_\rev}$ starts a reset attempt toward $\delta_{k_\rev}$ at stage $m_j+r$, which terminates before stage $m_j+r+r_{k,k_\rev}\le m_j+\ell_\eps+|\K|=m_{j+1}$ and, as in the previous case, succeeds with conditional probability at least $p_{\min}^{r_{k,k_\rev}}\ge p_{\min}^{|\K|}$.
                            Therefore, we deduce that
                            \begin{align*}
                                &\PP_{\sigma_{h_\rev}}^{b}\left(T_{k_\rev}\le m_{j+1}\givenm H_{m_j}=h_{m_j},T_{k_\rev}>m_j\right)\\
                                    &\qquad \ge p_{\min}^{|\K|}\PP_{\sigma_{h_\rev}}^{b}\left(\bigcup_{r=1}^{\ell_\eps}\{B_{m_j+r}\in \D_\C\}\givenm H_{m_j}=h_{m_j},T_{k_\rev}>m_j\right)
                                        &&\text{(decomposition)}\\
                                    &\qquad \ge \left(1-\eps/3\right)p_{\min}^{|\K|},
                                        &&\text{(by Eq.~\eqref{equation: alternative waiting bound})}
                            \end{align*}
                            which proves \eqref{equation: alternative one-trial bound}.
                    \end{itemize}
            
                    We now prove that, for every trial index $j\in \NN^*$,
                    \begin{align}
                        \PP_{\sigma_{h_\rev}}^{b}(T_{k_\rev}>m_j)\le \left(1-\left(1-\eps/3\right)p_{\min}^{|\K|}\right)^{j-1}.
                        \label{equation: alternative trial inequality}
                    \end{align}
                    We start by observing that the base case holds.
                    When $j=1$, we have that 
                    \[
                        \PP_{\sigma_{h_\rev}}^{b}(T_{k_\rev}>m_1)\le 1=\left(1-\left(1-\eps/3\right)p_{\min}^{|\K|}\right)^0
                    \]
                    and thus \eqref{equation: alternative trial inequality} holds.
                    Assume now that \eqref{equation: alternative trial inequality} holds for some $j\in \NN^*$.
                    We prove that it holds for trial $j+1$.
                    Since $\{T_{k_\rev}>m_{j+1}\}\subseteq \{T_{k_\rev}>m_j\}$, if $\PP_{\sigma_{h_\rev}}^{b}(T_{k_\rev}>m_j)=0$, then $\PP_{\sigma_{h_\rev}}^{b}(T_{k_\rev}>m_{j+1})=0$ and \eqref{equation: alternative trial inequality} holds for $j+1$.
                    Otherwise, using again the inclusion $\{T_{k_\rev}>m_{j+1}\}\subseteq \{T_{k_\rev}>m_j\}$ and the tower rule, we have that
                    \begin{align*}
                        &\PP_{\sigma_{h_\rev}}^{b}\left(T_{k_\rev}>m_{j+1}\right)\\
                            &\qquad = \PP_{\sigma_{h_\rev}}^{b}\left(T_{k_\rev}>m_{j+1}\givenm T_{k_\rev}>m_j\right)\PP_{\sigma_{h_\rev}}^{b}\left(T_{k_\rev}>m_j\right)\\
                            &\qquad = \left(1-\PP_{\sigma_{h_\rev}}^{b}\left(T_{k_\rev}\le m_{j+1}\givenm T_{k_\rev}>m_j\right)\right)\\
                            &\qquad\qquad \cdot\PP_{\sigma_{h_\rev}}^{b}\left(T_{k_\rev}>m_j\right)\\
                            &\qquad \le \left(1-\left(1-\eps/3\right)p_{\min}^{|\K|}\right)\PP_{\sigma_{h_\rev}}^{b}\left(T_{k_\rev}>m_j\right)
                                &&\text{(by Eq.~\eqref{equation: alternative one-trial bound})}\\
                            &\qquad \le \left(1-\left(1-\eps/3\right)p_{\min}^{|\K|}\right)^{j},
                                &&\text{(induction hyp.)}
                    \end{align*}
                    which proves inequality \eqref{equation: alternative trial inequality}.
                    Then, since $m_{N_\eps+1}=m_{\rev}$, we deduce that, for every $b\in \Delta_\C$,
                    \begin{align*}
                        \PP_{\sigma_{h_\rev}}^{b}(T_{k_\rev}>m_{\rev})
                            & = \PP_{\sigma_{h_\rev}}^{b}(T_{k_\rev}>m_{N_\eps+1})\\
                            & \le \left(1-\left(1-\eps/3\right)p_{\min}^{|\K|}\right)^{N_\eps}
                                &&\text{(by Eq.~\eqref{equation: alternative trial inequality})}\\
                            & \le \dfrac{\eps}{3},
                                &&\text{($1-x\le e^{-x}$, $1-\eps/3\ge 2/3$, and def. of $N_\eps$)}
                    \end{align*}
                    which proves inequality \eqref{equation: alternative reset bound}.
                    Moreover, letting $j$ tend to infinity, since $m_{j+1}=1+jL_\eps$ tends to infinity, continuity from above yields
                    \[
                        \PP_{\sigma_{h_\rev}}^{b}(T_{k_\rev}=\infty)=\lim_{j\to\infty}\PP_{\sigma_{h_\rev}}^{b}(T_{k_\rev}>m_{j+1})\le \lim_{j\to\infty}\left(1-\left(1-\eps/3\right)p_{\min}^{|\K|}\right)^{j}=0.
                    \]
                    In particular, we have that $\PP_{\sigma_{h_\rev}}^{b}(T_{k_\rev}<\infty)=1$.
            
                \vspace{1em}
                \smallskip\noindent\textit{Coupling strategy from $b_1'$.}
                    We now define the coupling strategy $\sigma'$ by averaging the strategies $\sigma_{h_\rev}$ according to the law $\mu$ of the first revealed history.
                    For every $m\in \NN^*$, history $h_m\in \H_m$, and action $i\in \I$, define
                    \[
                        \sigma'(h_m)(i)\defas
                            \begin{cases}
                                \dfrac{\sum_{h_\rev}\mu(h_\rev)\PP_{\sigma_{h_\rev}}^{b_1'}(H_m=h_m,I_m=i)}{\sum_{h_\rev}\mu(h_\rev)\PP_{\sigma_{h_\rev}}^{b_1'}(H_m=h_m)},
                                    &\text{if }\sum_{h_\rev}\mu(h_\rev)\PP_{\sigma_{h_\rev}}^{b_1'}(H_m=h_m)>0,\\[3ex]
                                \dfrac{\mathbf{1}_{\left\{i\in \E\left(\supp(b_{h_m}^{b_1'})\right)\right\}}}{\left|\E\left(\supp(b_{h_m}^{b_1'})\right)\right|},
                                    &\text{else if }h_m\in \H_m(b_1')\text{ and }\supp(b_{h_m}^{b_1'})\in \Q,\\[3ex]
                                \dfrac{1}{|\I|},
                                    &\text{otherwise}.
                            \end{cases}
                    \]
                    In other words, the strategy $\sigma'$ follows the mixture of the strategies $\sigma_{h_\rev}$ on the histories that they generate with positive probability, plays a uniformly chosen $\C$-safe action on the remaining $\C$-safe histories, and is defined arbitrarily on the histories that lead to beliefs with support that does not belong to $\Q$.
                    These last histories have a probability zero of occurring under the coupling construction.
                    By construction, the strategy $\sigma'$ depends only on the history of the second copy $\POMDP(b_1')$, and is therefore a well-defined strategy in $\POMDP(b_1')$.
            
                    We first prove that $\sigma'$ is $\C$-safe from $b_1'$.
                    Fix a history $h_m\in \H_m(b_1')$ such that $q\defas \supp(b_{h_m}^{b_1'})\in \Q$ and an action $i\notin \E(q)$.
                    Since $\sigma_{h_\rev}\in \Sigma_\C(b_1')$ for every $h_\rev$ with $\mu(h_\rev)>0$, we have that $\sigma_{h_\rev}(h_m)(i)=0$ whenever $\PP_{\sigma_{h_\rev}}^{b_1'}(H_m=h_m)>0$, and therefore
                    \[
                        \PP_{\sigma_{h_\rev}}^{b_1'}(H_m=h_m,I_m=i)=\PP_{\sigma_{h_\rev}}^{b_1'}(H_m=h_m)\sigma_{h_\rev}(h_m)(i)=0.
                    \]
                    Hence, the first two branches in the definition of $\sigma'$ give that $\sigma'(h_m)(i)=0$ for every action $i\notin \E(q)$, and since $\sigma'(h_m)\in \Delta(\I)$, we get that $\sum_{i\in \E(q)}\sigma'(h_m)(i)=1$.
                    Moreover, since $\C$ is closed, playing an action $i\in \E(q)$ keeps the belief-support inside $\Q$, i.e., for every signal $s\in \S$ such that $\psi(q,i,s)\neq \emptyset$,
                    \[
                        \supp\left(b_{h_m\times (i,s)}^{b_1'}\right)=\psi(q,i,s)\in \Post(q,i)\subseteq \Q.
                    \]
                    Since $\supp(b_1')\in \Q$, an induction on $m\in \NN^*$ yields that $\sigma'\in \Sigma_\C(b_1')$.
            
                    By using an induction argument on $m\in \NN^*$, we next prove that, for every $h_m\in \H_m(b_1')$ with $m\in \NN^*$,
                    \begin{align}
                        \PP_{\sigma'}^{b_1'}(H_m=h_m)=\sum_{h_\rev}\mu(h_\rev)\PP_{\sigma_{h_\rev}}^{b_1'}(H_m=h_m).
                        \label{equation: alternative mixture identity}
                    \end{align}
                    For the base case $m=1$, we have that $\H_1(b_1')=\{\emptyset\}$ and, since $\sum_{h_\rev}\mu(h_\rev)=1$, both sides of \eqref{equation: alternative mixture identity} are equal to $1$.
                    For the induction case, assume that the statement holds at some stage $m\in \NN^*$.
                    Consider an admissible history $h_{m+1}=h_m\times (i_m,s_{m+1})$ and set $b_m\defas b_{h_m}^{b_1'}$.
                    Recall that the conditional signal probability
                    \[
                        \PP(s_{m+1}\given b_m,i_m)=\sum_{k\in \K}\sum_{k'\in \K}b_m(k)P_{k,k'}(i_m,s_{m+1})
                    \]
                    depends only on the belief $b_m$ and on the action $i_m$, and not on the strategy.
                    If the weighted probability $\sum_{h_\rev}\mu(h_\rev)\PP_{\sigma_{h_\rev}}^{b_1'}(H_m=h_m)$ is equal to zero, then both sides of \eqref{equation: alternative mixture identity} at $h_{m+1}$ are equal to zero, since $\{H_{m+1}=h_{m+1}\}\subseteq\{H_m=h_m\}$ and by the induction hypothesis.
                    Otherwise, by the induction hypothesis, we have that
                    \begin{align*}
                        &\PP_{\sigma'}^{b_1'}(H_{m+1}=h_{m+1})\\
                            &\qquad = \PP_{\sigma'}^{b_1'}(H_m=h_m)\sigma'(h_m)(i_m)\PP(s_{m+1}\given b_m,i_m)\\
                            &\qquad = \sum_{h_\rev}\mu(h_\rev)\PP_{\sigma_{h_\rev}}^{b_1'}(H_m=h_m,I_m=i_m)\PP(s_{m+1}\given b_m,i_m)
                                &\text{(def. of $\sigma'$)}\\
                            &\qquad = \sum_{h_\rev}\mu(h_\rev)\PP_{\sigma_{h_\rev}}^{b_1'}(H_{m+1}=h_{m+1}),
                    \end{align*}
                    which proves the induction.
            
                \vspace{1em}
                \smallskip\noindent\textit{Asynchronous coupling.}
                    We now couple two copies of $\POMDP$: the first copy starts from $b_1$ and is played according to $\sigma$, while the second copy starts from $b_1'$ and is played according to $\sigma'$.
                    Denote by $(H_m,B_m,G_m)_{m\in \NN^*}$ (resp., $(H_m',B_m',G_m')_{m\in \NN^*}$) the history, belief, and reward process in the first (resp., second) copy $\POMDP(b_1)$ (resp., $\POMDP(b_1')$).
                    Formally, let $\nu$ be the probability measure on $\Omega\times\Omega$ obtained as follows.
                    Draw $h_\rev$ according to $\mu$.
                    Conditionally on $h_\rev$, draw the first play up to $T_\rev$ according to its conditional law under $\PP_\sigma^{b_1}$, and draw the second play up to $T_{k_\rev}$ according to the law induced by $\sigma_{h_\rev}$ from $b_1'$.
                    At these respective times, both copies have belief $\delta_{k_\rev}$; complete the two plays with the same continuation generated by $\sigma[h_\rev]$.
                    In particular:
                    \begin{itemize}
                        \item
                            The first marginal of $\nu$ is the law of the process in $\POMDP(b_1)$ induced by $\sigma$.
                        \item
                            The second marginal of $\nu$ is the law of the process in $\POMDP(b_1')$ induced by $\sigma'$, by Eq.~\eqref{equation: alternative mixture identity}.
                    \end{itemize}
                    We write $\EE_\nu$ and $\PP_\nu$ for the respective expectation and probability measures.
            
                    By the coupling construction, the first (resp., second) copy of the POMDP is on a (random) Dirac belief $\delta_{K_\rev}$ with $\{K_\rev\}\in \Q$ at stage $T_{\rev}$ (resp., at the first stage at which the second copy reveals $K_\rev$).
                    Formally, let $K_\rev$ be the random state such that $B_{T_{\rev}}=\delta_{K_\rev}$ and define
                    \[
                        T_{K_\rev}\defas \inf\left\{m\in \NN^*\colon\; B_m'=\delta_{K_\rev}\right\}.
                    \]
                    Since $m_{\rev}\ge 1+\ell_\eps$ by the coupling construction, and since $K_\rev$ may be random, we have that
                    \begin{align}
                        \PP_{\nu}(T_{\rev}>m_{\rev})
                            & \le \PP_{\sigma}^{b_1}(T_{\rev}>1+\ell_\eps)\le \dfrac{\eps}{3},
                                &&\text{(first marginal and Eq.~\eqref{equation: alternative first revelation})}\nonumber\\
                        \PP_{\nu}(T_{K_\rev}>m_{\rev})
                            & = \sum_{h_\rev}\mu(h_\rev)\PP_{\sigma_{h_\rev}}^{b_1'}(T_{k_\rev}>m_{\rev})\le \dfrac{\eps}{3}.
                                &&\text{(second marginal and Eq.~\eqref{equation: alternative reset bound})}
                        \label{equation: alternative two revelation times}
                    \end{align}
                    In particular, both revelation times are finite $\nu$-almost surely.
                    Define the synchronization time by $T_{\syn}\defas \max\{T_{\rev},T_{K_\rev}\}$.
                    Since $\{T_{\syn}>m_{\rev}\}=\{T_{\rev}>m_{\rev}\}\cup \{T_{K_\rev}>m_{\rev}\}$, we have that
                    \begin{align}
                        \PP_{\nu}(T_{\syn}>m_{\rev})\le \PP_{\nu}(T_{\rev}>m_{\rev})+\PP_{\nu}(T_{K_\rev}>m_{\rev})\le \dfrac{2\eps}{3}.
                        \label{equation: alternative synchronization}
                    \end{align}
                    Moreover, on the event $\{T_{\syn}<\infty\}$, both copies have reached the same revealed belief $\delta_{K_\rev}$, at stage $T_{\rev}$ for the first copy and at stage $T_{K_\rev}$ for the second one, and then use the same continuation strategy $\sigma[H_{\rev}]$.
                    Therefore, we deduce that
                    \begin{align}
                        \left(G_{T_\rev+r}\right)_{r\ge 0}=\left(G_{T_{K_\rev}+r}'\right)_{r\ge 0},\qquad \nu\text{-almost surely}.
                        \label{equation: alternative coupled tails}
                    \end{align}
            
                \vspace{1em}
                \smallskip\noindent\textit{$n$-stage value comparison.}
                    Define the auxiliary horizon $\widetilde n_\eps\defas\left\lceil 6m_{\rev}/\eps\right\rceil$.
                    Substituting the reset horizon $m_\rev$ of \eqref{equation: alternative reset horizon} and then using
                    $\ell_\eps\leq\left\lceil\log(3/\eps)/p_{\min}\right\rceil$, which is immediate when $p_{\min}=1$ and follows from $-\log(1-p_{\min})\geq p_{\min}$ when $p_{\min}<1$, we get that
                    \begin{align*}
                        \widetilde n_\eps
                            &=\left\lceil\dfrac{6}{\eps}\left(1+L_\eps N_\eps\right)\right\rceil
                            =\left\lceil\dfrac{6}{\eps}\left(1+\left(\ell_\eps+|\K|\right)N_\eps\right)\right\rceil\\
                            &\leq\left\lceil\dfrac{6}{\eps}\left(1+\left(\left\lceil\dfrac{\log(3/\eps)}{p_{\min}}\right\rceil+|\K|\right)\left\lceil\dfrac{3\log(3/\eps)}{2p_{\min}^{|\K|}}\right\rceil\right)\right\rceil
                            =n_\eps,
                    \end{align*}
                    where $n_\eps$ is the horizon \eqref{equation: explicit horizon} of the statement. Fix a horizon $n\ge n_\eps$ and note that $n\ge \widetilde n_\eps>m_{\rev}$.
                    On the event $\{T_{\syn}\le m_{\rev}\}$, the rewards of the two copies coincide from stage $T_{\rev}$ in the first copy and from stage $T_{K_\rev}$ in the second copy on, by Eq.~\eqref{equation: alternative coupled tails}.
                    Hence, since $g(\cdot)\in [0,1]$, only the stages before the synchronization and the $|T_{\rev}-T_{K_\rev}|$ stages of shift contribute to the difference.
                    Therefore, we obtain that, for every $n\ge n_\eps$,
                    \begin{align*}
                        &\left|\gamma_{n,\C}(b_1,\sigma)-\gamma_{n,\C}(b_1',\sigma')\right|\\
                            &\qquad = \left|\EE_{\nu}\left(\dfrac{1}{n}\sum_{m=1}^n(G_m-G_m')\right)\right|
                                &&\text{(coupling)}\\
                            &\qquad \le \EE_{\nu}\left(\textbf{1}_{\left\{T_{\syn}\le m_\rev\right\}}\dfrac{1}{n}\left|\sum_{m=1}^n (G_m - G_m')\right|\right)\\
                            &\qquad
                                \qquad + \EE_{\nu}\left(\textbf{1}_{\left\{T_{\syn}> m_\rev\right\}}\dfrac{1}{n}\left|\sum_{m=1}^n (G_m - G_m')\right|\right)
                                    &&\text{(triangle ineq. and decomposition)}\\
                            &\qquad \le \dfrac{1}{n}\EE_{\nu}\Bigl(\textbf{1}_{\left\{T_{\syn}\le m_\rev\right\}}\bigl((T_{\rev}-1)+(T_{K_\rev}-1)\\
                            &\qquad
                                \qquad +\left|T_{\rev}-T_{K_\rev}\right|\bigr)\Bigr)+ \PP_{\nu}(T_{\syn}>m_{\rev})
                                    &&\text{(coupled tails and $g(\cdot)\in [0,1]$)}\\
                            &\qquad = \dfrac{2}{n}\EE_{\nu}\left(\textbf{1}_{\left\{T_{\syn}\le m_\rev\right\}}\left(T_{\syn}-1\right)\right)+\PP_{\nu}(T_{\syn}>m_{\rev})
                                &&(a+b+|a-b|=2\max\{a,b\})\\
                            &\qquad \le \dfrac{2m_{\rev}}{n}+\PP_{\nu}(T_{\syn}>m_{\rev})
                                &&\text{(on $\{T_{\syn}\le m_{\rev}\}$)}\\
                            &\qquad \le \dfrac{2m_{\rev}}{n}+\dfrac{2\eps}{3}
                                &&\text{(by Eq.~\eqref{equation: alternative synchronization})}\\
                            &\qquad \le \eps.
                                &&\text{(def. of $\widetilde n_\eps$)}
                    \end{align*}
                    Since $\sigma'\in \Sigma_\C(b_1')$, we obtain that, for every $\sigma\in \Sigma_\C(b_1)$ and $n\ge n_\eps$,
                    \[
                        \gamma_{n,\C}(b_1,\sigma)\le \gamma_{n,\C}(b_1',\sigma')+\eps\le v_{n,\C}(b_1')+\eps.
                    \]
                    Taking the supremum over $\sigma$ yields
                    \begin{equation}
                        v_{n,\C}(b_1)\le v_{n,\C}(b_1')+\eps.
                        \label{equation: alternative inequality 1 coupling}
                    \end{equation}
                    Exchanging the roles of $b_1$ and $b_1'$, we obtain that, for every $n\ge n_\eps$,
                    \begin{equation}
                        v_{n,\C}(b_1')\le v_{n,\C}(b_1)+\eps.
                        \label{equation: alternative inequality 2 coupling}
                    \end{equation}
                    Finally, since the horizon $n_\eps$ depends only on $\eps$, on $p_{\min}$, and on $|\K|$, and not on the pair of initial beliefs nor on the strategy, by combining \eqref{equation: alternative inequality 1 coupling} and \eqref{equation: alternative inequality 2 coupling}, we get that, for every $n\ge n_\eps$ and pair of initial beliefs $b_1,b_1'\in \Delta_\C$,
                    \[
                        \left|v_{n,\C}(b_1)-v_{n,\C}(b_1')\right|\le \eps,
                    \]
                    which completes the proof.
            \end{proof}
            
            The next lemma proves an explicit convergence rate for the $\C$-safe value.
            
            \begin{Lemma}\label{Result: Approximation of the Safe Value}
                Consider a revealing POMDP $\POMDP$ and its BS-MDP $\MDP_B$.
                Fix a maximal end-component $\C=(\Q,\E)\in\mathfrak{M}$ of $\MDP_B$ and $\eps\in (0,1)$.
                Then, for every $n\geq n_\eps$, with $n_\eps$ given by Lemma~\ref{Result: n-stage Safe value}, and every initial belief $b_1\in \Delta_\C$,
                \[
                    \left|v_{n,\C}(b_1)-v_\C(b_1)\right|\le \eps.
                \]
            \end{Lemma}
            
            \begin{proof}[Proof of Lemma~\ref{Result: Approximation of the Safe Value}]
            
                Consider a revealing POMDP $\POMDP$ and its BS-MDP $\MDP_B$.
                Fix a maximal end-component $\C=(\Q,\E)\in\mathfrak{M}$ of $\MDP_B$, $\eps\in (0,1)$, and the horizon $n_\eps$ given by Lemma~\ref{Result: n-stage Safe value}.
                We prove that, for every horizon $n\ge n_\eps$ and initial belief $b\in \Delta_\C$,
                \[
                    \left|v_{n,\C}(b)-v_{\C}(b)\right|\le \eps.
                \]
            
                    We prove that 
                    \begin{align}
                        v_{n,\C}(b)\le v_\C(b) + \eps.
                        \label{equation: equation number one}
                    \end{align}
                    Fix $n\ge n_\eps$, $b\in \Delta_\C$, and $\eta>0$.
                    We construct a $\C$-safe strategy $\sigma_\eps$ block by block, each of length $n$.
                    At the beginning of block $j$, define
                    \[
                        m_j\defas jn+1, \qquad j\in \NN.
                    \]
                    At the beginning of block $j$, for every realized history $h_{m_j}$, let $b_{m_j}\defas b_{h_{m_j}}^b$ be the current belief.
                    Because the strategy constructed up to this stage is $\C$-safe, we have that $b_{m_j}\in \Delta_\C$.
                    From $b_{m_j}$, choose an $\eta$-optimal $\C$-safe strategy $\sigma_{j,h_{m_j}}$ for the $n$-stage value and follow it during block $j$. By definition of the supremum, it can be chosen such that
                    \begin{align*}
                        \EE_{\sigma_{j,h_{m_j}}}^{b_{m_j}}\left(\dfrac{1}{n}\sum_{m=1}^nG_m\right)
                            &\ge v_{n,\C}(b_{m_j})-\eta
                                &&\left(\text{def. of $\sigma_{j,h_{m_j}}$}\right)\\
                            &\ge v_{n,\C}(b)-\eps-\eta.
                                &&\left(\text{by Lemma~\ref{Result: n-stage Safe value}}\right)
                    \end{align*}
                    Given an arbitrary horizon $N$, write $N\defas qn+r$ with $q\in \NN$ and $r\in [0\until n-1]$.  
                    Then, 
                    \begin{align*}
                        &v_\C(b)\\
                            & \ge \gamma_{\C}(b,\sigma_\eps)
                                &&(\sigma_\eps\in \Sigma_\C(b))\\
                            & = \liminf_{N\to\infty}\EE_{\sigma_\eps}^{b}\left(\dfrac{1}{N}\sum_{m=1}^NG_m\right)
                                &&\text{(definition of $\gamma_\C$)}\\
                            & \ge \liminf_{N\to\infty}\EE_{\sigma_{\eps}}^{b}\left(\dfrac{1}{N}\sum_{m=1}^{qn}G_m\right)
                                &&(g(\cdot)\in [0,1])\\
                            & =  \liminf_{N\to\infty}\dfrac{1}{N}\sum_{j=0}^{q-1}\EE_{\sigma_{\eps}}^{b}\left(\sum_{m=m_j}^{(j+1)n}G_m\right)
                                &&\text{(block decomposition)}\\
                            & = \liminf_{N\to\infty}\dfrac{1}{N}\sum_{j=0}^{q-1}\EE_{\sigma_\eps}^{b}\left(\EE_{\sigma_\eps}^{b}\left(\sum_{m=m_j}^{(j+1)n}G_m\givenm H_{m_j}\right)\right)
                                &&\text{(tower rule)}\\
                            & = \liminf_{N\to\infty}\dfrac{n}{N}\sum_{j=0}^{q-1}\EE_{\sigma_\eps}^{b}\left(\EE_{\sigma_\eps}^{b}\left(\dfrac{1}{n}\sum_{m=m_j}^{(j+1)n}G_m\givenm H_{m_j}\right)\right)
                                &&\text{(introduce $n$)}\\
                            & \ge \liminf_{N\to\infty}\dfrac{n}{N}\sum_{j=0}^{q-1}\EE_{\sigma_\eps}^{b}\left(v_{n,\C}(B_{m_j})-\eta\right)
                                &&(\eta\text{-optimality of the strategy in block $j$})\nonumber\\
                            & \ge \liminf_{N\to\infty}\dfrac{n}{N}\sum_{j=0}^{q-1}\left(v_{n,\C}(b)-\eps-\eta\right)
                                &&\text{(by Lemma~\ref{Result: n-stage Safe value})}\\
                            & = \liminf_{N\to\infty}\dfrac{qn}{N}\left(v_{n,\C}(b)-\eps-\eta\right)
                                &&\text{(sum of $q$ identical terms)}\\
                            & = \liminf_{q\to\infty}\dfrac{qn}{qn+r}\left(v_{n,\C}(b)-\eps-\eta\right)
                                &&(N=qn+r)\\
                            & = v_{n,\C}(b)-\eps-\eta.
                    \end{align*}
                    Since $\eta$ was taken arbitrary, we deduce that
                    \[
                        v_\C(b)\ge v_{n,\C}(b)-\eps.
                    \]
                    
                    We now prove that 
                    \begin{align}
                        v_\C(b)\le v_{n,\C}(b)+\eps.
                        \label{equation: equation number 2}
                    \end{align}
                    Fix an arbitrary strategy $\sigma\in \Sigma_\C(b)$.
                    The continuation of $\sigma$ at the beginning of a block $j$ is $\C$-safe from the current (random) belief $B_{m_j}$. Therefore, almost surely,
                    \begin{align*}
                        \EE_\sigma^{b}\left(\dfrac{1}{n}\sum_{m=m_j}^{m_j+n-1}G_m\givenm H_{m_j}\right)\le v_{n,\C}(B_{m_j})\le v_{n,\C}(b)+\eps.
                    \end{align*}
                    Writing $N=qn+r$ as above, we obtain 
                    \[
                        \EE_{\sigma}^{b}\left(\dfrac{1}{N}\sum_{m=1}^NG_m\right)\le \dfrac{qn}{N}v_{n,\C}(b)+\eps+\dfrac{r}{N}.
                    \]
                    Taking the limit inferior as $N\to \infty$,
                    \[
                        \gamma_\C(b,\sigma)\le v_{n,\C}(b)+\eps.
                    \]
                    Since $\sigma\in \Sigma_\C(b)$ was arbitrary, taking the supremum yields 
                    \[
                        v_\C(b)\le v_{n,\C}(b)+\eps.
                    \]
                    By combining \eqref{equation: equation number one} and \eqref{equation: equation number 2} we obtain that, for every $n\ge n_\eps$ and $b\in \Delta_\C$,
                    \begin{align}
                        \left|v_{n,\C}(b)-v_{\C}(b)\right|\le \eps,
                    \end{align}
                    which concludes the proof.
            \end{proof}
    
            \begin{Lemma}\label{Result: Independence of the Safe Value}
                Consider a revealing POMDP $\POMDP$ and its BS-MDP $\MDP_B$.
                Fix a maximal end-component $\C=(\Q,\E)\in\mathfrak{M}$ of $\MDP_B$.
                Then, the $\C$-safe value is independent of the initial belief, i.e., for every pair of initial beliefs $b_1,b_1'\in \Delta_\C$,
                \[
                    v_\C(b_1)=v_\C(b_1').
                \]
            \end{Lemma}
            
            \begin{proof}[Proof of Lemma~\ref{Result: Independence of the Safe Value}]
                We prove that, for every pair of initial beliefs $b_1,b_1'\in \Delta_\C$, the following equality holds
                \[  
                    v_\C(b_1)=v_\C(b_1').
                \]
                By Lemmas~\ref{Result: n-stage Safe value} and~\ref{Result: Approximation of the Safe Value}, for every $\eps\in (0,1)$, the horizon $n_\eps$ given by Lemma~\ref{Result: n-stage Safe value} satisfies that, for every $n\ge n_\eps$,
                \[
                    \left|v_{n,\C}(b_1)-v_{n,\C}(b_1')\right|\le \eps
                    \quad\text{and}\quad
                    \left|v_{n,\C}(b)-v_\C(b)\right|\le \eps
                    \quad\text{for every }b\in\Delta_\C.
                \]
                Then, 
                \[
                    \left|v_\C(b_1)-v_\C(b_1')\right|
                        \le \left|v_\C(b_1) - v_{n,\C}(b_1)\right|
                            + \left|v_{n,\C}(b_1) - v_{n,\C}(b_1')\right|
                            + \left|v_\C(b_1') - v_{n,\C}(b_1')\right|
                        \le 3\eps.
                \]
                Since $\eps\in (0,1)$ was arbitrary, it follows that $v_\C(b_1)=v_\C(b_1')$ for every $b_1,b_1'\in \Delta_\C$, which completes the proof.
            \end{proof}
            
        \paragraph{Restriction to maximal end-components and approximate maximum safe value}
            We show that the maximum safe value can be taken over maximal end-components. 
            Namely, for every state $k\in\K$,
            \begin{align}
                v_{\safe}(k)
                    =\max\left\{v_{\M}(\delta_k)\colon\; \M=(\Q',\E')\in \mathfrak{M} \text{ and } \{k\}\in \Q'\right\}.
                \label{equation: safe value over maximal end-components}
            \end{align}
            Since $\mathfrak{M}\subseteq\mathfrak{C}$, the left-hand side is at least the right-hand side. Conversely, fix an end-component $\C=(\Q,\E)\in\mathfrak{C}$ such that $\{k\}\in\Q$. Because $\MDP_B$ is finite, there exists a maximal end-component $\M=(\Q',\E')\in\mathfrak{M}$ such that $\C\preceq\M$. Every $\C$-safe strategy from $\delta_k$ is also $\M$-safe, and hence $v_\C(\delta_k)\leq v_\M(\delta_k)$. Taking the maximum over such end-components $\C$ proves the reverse inequality in \eqref{equation: safe value over maximal end-components}. If there is no such $\C$, both maxima are zero by convention.
    
            The horizon $n_\eps$ in Lemma~\ref{Result: n-stage Safe value} does not depend on the maximal end-component. For every $\eps\in(0,1)$, define
            \[
                v_{\safe,\eps}(k)\defas\max\left\{v_{n_\eps,\M}(\delta_k)\colon \M=(\Q,\E)\in \mathfrak{M} \text{ and } \{k\}\in \Q\right\}.
            \]
            Lemma~\ref{Result: Approximation of the Safe Value}, Eq.~\eqref{equation: safe value over maximal end-components}, and the inequality $|\max_j x_j-\max_j y_j|\le \max_j|x_j-y_j|$ imply that
            \begin{align}
                \max_{k\in \K}\left|v_{\safe}(k)-v_{\safe,\eps}(k)\right|\le \eps.
                \label{equation: approximationg of safe value bound}
            \end{align}
        
            We can now prove Theorem~\ref{Theorem: approximation of max safe value is EXPTIME}.
            
            \begin{proof}[Proof of Theorem~\ref{Theorem: approximation of max safe value is EXPTIME}]
                Consider a revealing POMDP $\POMDP$ with initial belief $b_1$ and its BS-MDP $\MDP_B$. 
                Fix $\eps>0$.
                We use a point-based approach~\cite{shani2013survey} to prove the EXPTIME complexity.
    
                \vspace{1em}
                \smallskip\noindent\textit{Support-preserving grid.}
                    Fix a maximal end-component $\C=(\Q,\E)\in \mathfrak{M}$, a horizon $n\in \NN^*$, and $\eta\in (0,1)$.
                    For every integer $\ell\geq|\K|$, define the $\ell$-uniform grid
                    \[
                        \G_{\ell}\defas\left\{b\in \Delta(\K)\colon\;\ell\, b(k)\in \NN\text{ for every }k\in \K\right\}.
                    \]
                    For every belief-support $q\in \beliefsupport$, denote the set of grid points with support $q$ by $\G_\ell(q)\defas\left\{b\in \G_\ell\colon\; \supp(b)=q\right\}$.
                    Every such grid point is determined by the positive integers $\left(\ell\, b(k)\right)_{k\in q}$, whose sum is $\ell$. Therefore, we have
                    \begin{equation}
                        \left|\G_{\ell}(q)\right|
                        =\binom{\ell-1}{|q|-1}
                        \leq \ell^{|q|-1}
                        \leq \ell^{|\K|}.
                        \label{equation: grid size}
                    \end{equation}
                    Denote by $\Pi_\ell\colon \Delta_\C\to \G_\ell$ a map that associates with every belief $b\in \Delta_\C$ a closest grid point of $\G_{\ell}$ such that the support remains the same, i.e.,
                    \begin{equation}
                        \Pi_\ell(b)\in\argmin_{\substack{b^\prime\in \G_\ell\\[2pt]\supp(b)=\supp(b^\prime)}}\|b-b^\prime\|_1.
                        \label{equation: rounding operator}
                    \end{equation}
                    A standard rounding argument shows that, for every belief $b\in\Delta_\C$ with $q=\supp(b)$, there exists a point $b'\in\G_\ell(q)$ such that
                    \[
                        \|b-b'\|_1\leq\dfrac{2(|q|-1)}{\ell}\leq\dfrac{2|\K|}{\ell}.
                    \]
                    Indeed, round the numbers $(\ell b(k))_{k\in q}$ to positive integers summing to $\ell$; the total mass rounded upward equals the total mass rounded downward, and each is at most $|q|-1$. Since $\Pi_\ell(b)$ is a closest support-preserving grid point, taking $\ell=\left\lceil\dfrac{2(n+1)|\K|}{\eta}\right\rceil$ yields, for every $b\in \Delta_\C$, $\|b-\Pi_\ell(b)\|_1\le \eta/(n+1)$ and $\supp(\Pi_\ell(b))=\supp(b)$.
    
                \vspace{1em}
                \smallskip\noindent\textit{Bellman equations.}
                    Fix a belief-support $q\in \Q$ and a stage $m\in [0\until n-1]$.
                    For every belief $b\in\Delta_\C$ and action $i\in\I$, write
                    $\S_+(b,i)\defas\{s\in\S\colon\PP(s\given b,i)>0\}$.
                    Then, for every belief $b\in \Delta_\C$ with $\supp(b)=q$, respectively for every grid point $b\in \G_\ell$ with $\supp(b)=q$,
                    \begin{align}
                        v_{m+1,\C}(b)
                            &=\max_{i\in \E(q)}\left[\dfrac{1}{m+1}\sum_{k\in \K}b(k)g(k,i)+\dfrac{m}{m+1}\sum_{s\in \S_+(b,i)}\PP(s\given b,i)\,v_{m,\C}\left(\Phi(b,i,s)\right)\right],
                        \label{equation: safe dynamic programming}\\
                        \widehat{v}_{m+1,\C}(b)
                            &\defas\max_{i\in \E(q)}\left[\dfrac{1}{m+1}\sum_{k\in \K}b(k)g(k,i)+\dfrac{m}{m+1}\sum_{s\in \S_+(b,i)}\PP(s\given b,i)\,\widehat{v}_{m,\C}\left(\Pi_\ell\left(\Phi(b,i,s)\right)\right)\right],
                        \label{equation: point-based backup}
                    \end{align}
                    with the conventions $v_{0,\C} = \widehat{v}_{0,\C} = 0$.
                    Both recursions stay in $\Delta_\C$ and select only $\C$-safe actions, since, for every $s\in\S_+(b,i)$, $\supp\left(\Phi(b,i,s)\right)=\psi(q,i,s)\in \Post(q,i)\subseteq \Q$, and $\Pi_\ell$ preserves belief-supports by \eqref{equation: rounding operator}.
                    
                \vspace{1em}
                \smallskip\noindent\textit{Lipschitz property.}
                    We establish the Lipschitz property on every fixed belief-support.
                    Fix a horizon $n\in \NN^*$ and two beliefs $b,\another{b}\in \Delta_\C$ such that $\supp(b)=\supp(\another{b})$.
                    Therefore, we have that $\H_m(b)=\H_m(\another{b})$, $\supp\left(b_{h_m}^{b}\right)=\supp\left(b_{h_m}^{\another{b}}\right)$ for every $h_m\in \H_m(b)$, and $\Sigma_\C(b)=\Sigma_\C(\another{b})$.
    
                    Fix a strategy $\sigma\in \Sigma_\C(b)$.
                    Then, for every $n\in \NN^*$, 
                    \begin{align*}
                        \left|\gamma_{n,\C}(b,\sigma)-\gamma_{n,\C}(\another{b},\sigma)\right|
                            &=\left|\EE_{\sigma}^{b}\left(\dfrac{1}{n}\sum_{m=1}^nG_m\right)-\EE_{\sigma}^{\another{b}}\left(\dfrac{1}{n}\sum_{m=1}^nG_m\right)\right|
                                &&(\text{def. of }\gamma_{n,\C})\\
                            &=\left|\sum_{k\in \K}(b(k)-\another{b}(k))\EE_{\sigma}^{\delta_k}\left(\dfrac{1}{n}\sum_{m=1}^nG_m\right)\right|   
                                &&(b,\another{b}\in \Delta_\C)\\
                            & \leq \left\|b-\another{b}\right\|_1.
                                &&(g\in [0,1])
                    \end{align*}
                    Taking the supremum over $\Sigma_\C(b)=\Sigma_\C(\another{b})$ yields,
                    \begin{equation}
                        \left|v_{n,\C}(b)-v_{n,\C}(\another{b})\right|
                        \leq
                        \left\|b-\another{b}\right\|_1,
                        \label{equation: finite horizon safe Lipschitz}
                    \end{equation}
                    for every $n\in \NN^*$.
    
                \vspace{1em}
                \smallskip\noindent\textit{Error analysis.}
                    The dynamic programming equations \eqref{equation: safe dynamic programming} cannot be solved exactly because $\Delta_\C$ is infinite.
                    We therefore evaluate them on a finite grid of beliefs, following a point-based approach~\cite{shani2013survey,asadi2026revealing}.
                    The grid is chosen so that it preserves the belief-supports, which guarantees that the approximation scheme only manipulates beliefs of $\Delta_\C$ and only selects $\C$-safe actions.
                    We prove by induction on $m\in [0\until n]$ that, for every belief $b\in \Delta_\C$,
                    \begin{equation}
                        \left|v_{m,\C}(b)-\widehat{v}_{m,\C}\left(\Pi_\ell(b)\right)\right|\leq \dfrac{m\eta}{n+1}.
                        \label{equation: projected dynamic programming error}
                    \end{equation}
                    For the base case, \eqref{equation: projected dynamic programming error} is immediate because $v_{0,\C}=\widehat{v}_{0,\C} = 0$.
                    For the induction case, suppose that \eqref{equation: projected dynamic programming error} holds for some $m\in [0\until n-1]$ and every belief in $\Delta_\C$, fix a belief $b\in \Delta_\C$, and write $q\defas\supp(b)$.
                    Then, 
                    \begin{align*}
                        &\left|v_{m+1,\C}(b)-\widehat{v}_{m+1,\C}\left(\Pi_\ell(b)\right)\right|\\
                            &\qquad\le \left|v_{m+1,\C}(b)-v_{m+1,\C}\left(\Pi_\ell(b)\right)\right|+\left|v_{m+1,\C}\left(\Pi_\ell(b)\right)-\widehat{v}_{m+1,\C}\left(\Pi_\ell(b)\right)\right|\\
                            &\qquad\le \left\|b-\Pi_\ell(b)\right\|_1+\left|v_{m+1,\C}\left(\Pi_\ell(b)\right)-\widehat{v}_{m+1,\C}\left(\Pi_\ell(b)\right)\right|\\
                            &\qquad\le \dfrac{\eta}{n+1}+\left|v_{m+1,\C}\left(\Pi_\ell(b)\right)-\widehat{v}_{m+1,\C}\left(\Pi_\ell(b)\right)\right|\\
                            &\qquad\le \dfrac{\eta}{n+1}+\dfrac{m}{m+1}\max_{i\in \E(q)}\Biggl|\sum_{s\in \S_+(\Pi_\ell(b),i)}\PP\left(s\given \Pi_\ell(b),i\right)\\
                            &\qquad\qquad\qquad\cdot\left(v_{m,\C}\left(\Phi\left(\Pi_\ell(b),i,s\right)\right)-\widehat{v}_{m,\C}\left(\Pi_\ell\left(\Phi\left(\Pi_\ell(b),i,s\right)\right)\right)\right)\Biggr|\\
                            &\qquad\le \dfrac{\eta}{n+1}+\dfrac{m}{m+1}\max_{i\in \E(q)}\sum_{s\in \S_+(\Pi_\ell(b),i)}\PP\left(s\given \Pi_\ell(b),i\right)\\
                            &\qquad\qquad\qquad\cdot\left|v_{m,\C}\left(\Phi\left(\Pi_\ell(b),i,s\right)\right)-\widehat{v}_{m,\C}\left(\Pi_\ell\left(\Phi\left(\Pi_\ell(b),i,s\right)\right)\right)\right|\\
                            &\qquad\le \dfrac{\eta}{n+1}+\dfrac{m}{m+1}\cdot\dfrac{m\eta}{n+1}\\
                            &\qquad\le \dfrac{(m+1)\eta}{n+1},
                    \end{align*}
                    where the first inequality follows from the triangle inequality, the second inequality follows from the Lipschitz property \eqref{equation: finite horizon safe Lipschitz}, the third inequality follows from the choice of the grid parameter $\ell$, the fourth inequality follows from the Bellman equations \eqref{equation: safe dynamic programming} and \eqref{equation: point-based backup} and the inequality $|\max_j x_j-\max_j y_j|\le \max_j|x_j-y_j|$, the fifth inequality follows from the triangle inequality, the sixth inequality follows from the induction hypothesis, and the last inequality follows from the inequality $m^2/(m+1)\le m$.
                    Therefore, taking $m=n$ in \eqref{equation: projected dynamic programming error}, we obtain that
                    \begin{equation}
                        \left|v_{n,\C}(b)-\widehat{v}_{n,\C}\left(\Pi_\ell(b)\right)\right|\leq \dfrac{n\eta}{n+1}\leq\eta
                        \qquad\text{for every }b\in \Delta_\C.
                        \label{equation: point-based error bound}
                    \end{equation}
                    In particular, for every state $k\in \K$ such that $\delta_k\in \Delta_\C$, we have that $\delta_k\in \G_\ell$ and thus $\Pi_\ell(\delta_k)=\delta_k$, so that $\left|v_{n,\C}(\delta_k)-\widehat{v}_{n,\C}(\delta_k)\right|\leq \eta$.
    
                \vspace{1em}
                \smallskip\noindent\textit{Horizon truncation.}
                    We now truncate the infinite horizon and prove that the horizon at which the $n$-stage maximal $\C$-safe value approximates the maximal $\C$-safe value is explicit and depends only on $\eta$, $p_{\min}$, and $|\K|$.
                    For $\eps\geq1$, the statement is trivial.
                    Therefore, we assume that $\eps\in (0,1)$ and set $\eta\defas\eps/2$.
    
                    Fix a maximal end-component $\C\in \mathfrak{M}$ of $\MDP_B$.
                    By Lemma~\ref{Result: Approximation of the Safe Value} applied with $\eta$, the value $v_\C(\delta_k)$, with $\{k\}\in \Q$, is approximated up to $\eta$ by $v_{n_\eta,\C}(\delta_k)$, where $n_\eta$ is the horizon \eqref{equation: explicit horizon} applied with $\eta$.
                    Since $\lceil a\rceil\leq a+1$, we obtain that
                    \begin{align}
                        n_\eta
                            &\leq1+\dfrac{6}{\eta}\left(1+\left(1+|\K|+\dfrac{\log(3/\eta)}{p_{\min}}\right)\left(1+\dfrac{3\log(3/\eta)}{2p_{\min}^{|\K|}}\right)\right)\notag\\
                            &=2^{O\left(|\K|\log(1/p_{\min})+\log|\K|+\log(1/\eta)\right)}.
                        \label{equation: explicit horizon bound}
                    \end{align}
    
                    In particular, $n_\eta$ does not depend on the maximal end-component $\C\in \mathfrak{M}$ and can therefore be taken as a common horizon for all maximal end-components.
                    By Eq.~\eqref{equation: approximationg of safe value bound} applied with $\eta$, we obtain
                    \begin{align}
                        \max_{k\in \K}\left|v_{\safe}(k)-v_{\safe,\eta}(k)\right|\le \eta.
                        \label{equation: horizon truncation error}
                    \end{align}
    
                \vspace{1em}
                \smallskip\noindent\textit{Complexity analysis.}
                    Each belief update $\Phi(b,i,s)$ is computed in $O(|\K|^2)$ operations, so one backup \eqref{equation: point-based backup} at a single grid point costs $O\left(|\K|^2|\I||\S|\right)$.
                    The backups are performed at the grid points whose belief-support belongs to $\Q$, and by \eqref{equation: grid size} their number is at most
                    \[
                        \sum_{q\in \Q}\left|\G_\ell(q)\right|
                        \leq |\Q|\,\ell^{|\K|}
                        \leq 2^{|\K|}\ell^{|\K|}.
                    \]
                    Therefore, since $\ell=O\left(n|\K|/\eta\right)$, backward induction over the $n$ stages computes $\widehat{v}_{n,\C}$ in
                    \begin{equation}
                        O\left(n\cdot 2^{|\K|}\ell^{|\K|}\cdot|\K|^2|\I||\S|\right)
                        =2^{O\left(|\K|\log\left(n|\K|/\eta\right)\right)}|\I||\S|
                        \label{equation: point-based running time}
                    \end{equation}
                    operations.
    
                \vspace{1em}
                \smallskip\noindent\textit{Conclusion.}
                    We now run the procedure on every $\C\in \mathfrak{M}$ with $n\defas n_\eta$, and set $\widehat{v}_{\safe,\eta}(k)\defas\max\left\{\widehat{v}_{n_\eta,\C}(\delta_k)\colon \C=(\Q,\E)\in \mathfrak{M},\;\{k\}\in \Q\right\}$.
                    Since $|\max_j x_j-\max_j y_j|\le \max_j|x_j-y_j|$, Eqs.~\eqref{equation: horizon truncation error} and~\eqref{equation: point-based error bound} yield, for every $k\in \K$,
                    \[
                        \left|v_{\safe}(k)-\widehat{v}_{\safe,\eta}(k)\right|
                        \leq \left|v_{\safe}(k)-v_{\safe,\eta}(k)\right|+\left|v_{\safe,\eta}(k)-\widehat{v}_{\safe,\eta}(k)\right|
                        \leq 2\eta=\eps.
                    \]
    
                    By~\eqref{equation: explicit horizon bound}, we have that $\log n_\eta=O\left(|\K|\log(1/p_{\min})+\log|\K|+\log(1/\eta)\right)$.
                    Therefore, summing \eqref{equation: point-based running time} at $n=n_\eta$ over the maximal end-components of $\MDP_B$, whose number is at most $\left|\beliefsupport\right|=2^{|\K|}-1$ because distinct maximal end-components have pairwise disjoint sets of belief-supports~\cite{de1997formal}, the total running time is
                    \begin{align*}
                        2^{|\K|}\cdot 2^{O\left(|\K|\log\left(n_\eta|\K|/\eta\right)\right)}|\I||\S|
                            &=2^{O\left(|\K|\log n_\eta+|\K|\log\left(|\K|/\eta\right)\right)}|\I||\S|\\
                            &=2^{O\left(|\K|^2\log(1/p_{\min})+|\K|\log\left(|\K|/\eta\right)\right)}|\I||\S|
                                &&\left(\text{by Eq. }\text{\eqref{equation: explicit horizon bound}}\right)\\
                            &=2^{O\left(|\K|^2\log(1/p_{\min})+|\K|\log\left(|\K|/\eps\right)\right)}|\I||\S|.
                                &&(\eta=\eps/2)
                    \end{align*}
                    Since $\log(1/p_{\min})$ is at most the number of bits used to encode the transition probabilities of $\POMDP$, the exponent is polynomial in the size of $\POMDP$ and in $\log(1/\eps)$, so that the total running time is exponential in the size of the input, which concludes the proof.
            \end{proof}
    

\subsection{Reduction to Reachability Objectives}\label{Section: Reduction to Reachability Objectives}
    
    This section introduces the class of \emph{Commit POMDPs}, in which the controller commits to the continuation value associated with a revealed state.
    We then prove the reduction of revealing POMDPs with long-run average objectives to Commit POMDPs with reachability objectives.
    
    \paragraph{Commit POMDP}
        Consider a revealing POMDP $\POMDP=(\K,\I,\S,p,g)$. 
        We define the Commit POMDP $\POMDP'=(\K',\I',\S',p',g')$ by
        \begin{itemize}
            \item 
                $\K'=\K\cup\{\top,\bot\}$, where $\{\top,\bot\}$ are absorbing states;
            \item 
                $\I'=\I\cup\I_{\mathsf{com}}$, where $\I_{\mathsf{com}}\defas\left\{\mathsf{com}_k\colon k\in\K\right\}$. 
                The action $\mathsf{com}_k$ means that the controller commits to the continuation value associated with the revealed state $k$:
            \item 
                $\S'=\S\cup\{\top,\bot\}$, where $\{\top,\bot\}$ reveal the absorbing states;
            \item 
                $p'\colon \K'\times \I'\to \Delta(\K'\times \S')$ is the transition function defined by
                \begin{itemize}
                    \item 
                        For every $k,k'\in \K$, $i\in \I$, and $s\in \S$, $p'(k',s\given k,i) = p(k',s\given k,i)$;
                    \item
                        For every $k\in \{\top,\bot\}$, and $i\in \I'$, $p'(k,k\given k,i)=1$;
                    \item 
                        For every original state $\another{k}\in\K$ and every commit action
                        $\mathsf{com}_k\in\I_{\mathsf{com}}$, set
                        \[
                            p'(\top,\top\given \another{k},\mathsf{com}_k) = v_{\safe}(k)\mathbf{1}_{\{\another{k}=k\}},
                        \]
                        and
                        \[
                            p'(\bot,\bot\given \another{k},\mathsf{com}_k) = 1 - v_{\safe}(k)\mathbf{1}_{\{\another{k}=k\}}.
                        \]       
                \end{itemize}  
            \item
                $g'\colon \K'\times \I'\to [0,1]$ is the stage reward defined by $g'\left(q,i'\right) = \mathbb{1}_{\left\{q=\top\right\}}$.
                
        \end{itemize}
        \noindent Every belief $b\in \Delta(\K)$ in $\POMDP$ extends to a belief $b'\in \Delta(\K')$ in $\POMDP'$, defined by $b'(k)\defas b(k)$ for every $k\in \K$ and $b'(\top)\defas b'(\bot)\defas 0$.
        In particular, an initial belief $b_1\in \Delta(\K)$ in $\POMDP$ yields the initial belief $b_1'\in \Delta(\K')$ in $\POMDP'$.
        A history before stage $m\in \NN^*$ is a sequence $h_m'=(i_1',s_2',\ldots,i_{m-1}',s_m')$. 
        The set of histories before stage $m\in \NN^*$ is denoted by $\H_m'\defas (\I'\times \S')^{m-1}$.
        A strategy is a mapping $\sigma\colon \bigcup_{m\in \NN^*}\H_m'\to \I'$.
        The set of strategies in $\POMDP'$ is denoted by $\Sigma'$.
        Given an initial belief $b_1\in \Delta(\K)$ and a strategy $\sigma\in \Sigma'$, we denote by $\PP_{\sigma}^{b_1'}$ the probability measure induced by $\sigma$ from $b_1'$ on the set of plays $\left(\K'\times\I'\times \S'\right)^{\NN^*}$, and by $\EE_{\sigma}^{b_1'}$ the corresponding expectation.
    
    \paragraph{Reachability in Commit POMDPs}
        Consider a revealing POMDP $\POMDP$ and its Commit POMDP $\POMDP'$.
        Given an initial belief $b_1\in \Delta(\K)$ and a strategy $\sigma\in \Sigma'$, the reachability objective to the target state $\top$ is defined by
        \[
            \PP_{\sigma}^{b_1'}(\exists m\in \NN^*\colon\; K_m'=\top).
        \]
        The reachability value is defined by
        \[
            v_{R}'(b_1')\defas\sup_{\sigma\in \Sigma'}\PP_{\sigma}^{b_1'}(\exists m\in \NN^*\colon\; K_m'=\top).
        \]
    
    \paragraph{Approximate Commit POMDP}
        Consider a revealing POMDP $\POMDP$ and its Commit POMDP $\POMDP'$.
        Given a vector $w\in[0,1]^{\K}$, define $\POMDP'[w]$ in the same way as $\POMDP'$, except that $v_{\safe}(k)$ is replaced by $w(k)$ for all $k\in\K$.
        Given $b_1\in\Delta(\K)$, the reachability value of $\POMDP'[w]$ is denoted by $v_{R,w}'(b_1')$.
        By construction, both $\POMDP'$ and $\POMDP'[w]$ satisfy the revealing property.
        Since they have the same actions and signals, they have the same set of strategies.
        Given a strategy $\sigma\in\Sigma'$, denote by $\PP_{\sigma}^{w,b_1'}$ the probability measure induced from $b_1'$ in $\POMDP'[w]$, and by $\EE_{\sigma}^{w,b_1'}$ the corresponding expectation.
    
    \paragraph{Previous result on revealing POMDPs with reachability objectives}
        By~\cite{madani2003undecidability}, the approximation problem for POMDPs with reachability objectives is undecidable in general.
        By~\cite{asadi2026revealing}, the following positive result holds for revealing POMDPs.
        
        \begin{Theorem}
            Approximating revealing POMDPs with reachability objectives is in $\mathrm{EXPTIME}$.
        \end{Theorem}
        
    \paragraph{Approximation of Commit POMDPs}
        The Commit POMDP and an Approximate Commit POMDP differ only in the transition probabilities associated with the commit actions.
        The next lemma shows that the reachability value is Lipschitz in the commit weights.
        
        \begin{Lemma}\label{Results: approximating reachability in Commit POMDP}
            Consider a revealing POMDP $\POMDP$, $w\in[0,1]^{\K}$, and $\eps>0$ such that
            $\max_{k\in\K}|w(k)-v_{\safe}(k)|\leq\eps$.
            Then, for every $b_1\in \Delta(\K)$,
            \[
                \left|v_{R}'\left(b_1'\right) - v_{R,w}'\left(b_1'\right)\right| \le \eps.
            \]
        \end{Lemma}
        
        \begin{proof}[Proof of Lemma \ref{Results: approximating reachability in Commit POMDP}]
            Consider the Commit POMDP $\POMDP'$ and the Approximate Commit POMDP $\POMDP'[w]$.
            Define the first commit time by
            \[
                T_{\com}\defas\inf\left\{m\in \NN^*\colon\; I_m'\in \I_{\com}\right\},
            \]
            with $T_{\com}=\infty$ if no commit action is ever played.
            On the event $\{T_{\com}=\infty\}$, define $I_{T_{\com}}'$ and $B_{T_{\com}}'$ arbitrarily.
            Observe that the two POMDPs have identical dynamics before the first commit action.
            Therefore, given a strategy $\sigma$, for every $m\in\NN^*$, the variables $(H_m',I_m',B_m')$ restricted to $\{T_{\com}=m\}$ have the same law under $\PP_{\sigma}^{b_1'}$ and $\PP_{\sigma}^{w,b_1'}$.
            Moreover, on the event $\{T_{\com}=m,I_{T_{\com}}'=\com_{k}\}$, the conditional probability of reaching the target state $\top$ at stage $m+1$ in $\POMDP'$ is
            \begin{align*}
                &\PP_{\sigma}^{b_1'}\left(K_{m+1}'=\top\givenm T_{\com}=m,H_m',I_m'=\com_k\right)\\
                    &\qquad = \sum_{\another{k}\in \K}\PP_{\sigma}^{b_1'}\left(K_m' = \another{k} \givenm T_{\com} = m,H_m',I_m'=\com_k\right)p'(\top,\top\given \another{k},\com_k)\\
                    &\qquad = \sum_{\another{k}\in \K}B_m'(\another{k})v_{\safe}(k)\mathbf{1}_{\{\another{k}=k\}}\\
                    &\qquad = B_m'(k)v_{\safe}(k).
            \end{align*}
            In the Commit POMDP $\POMDP'$, $\top$ can only be reached after a commit action, and $\top,\bot$ are absorbing
            \begin{align}
                \PP_{\sigma}^{b_1'}(\exists m\in\NN^*\colon\; K_m'=\top) = \EE_\sigma^{b_1'}\left(\mathbf{1}_{\{T_{\com}<\infty\}}\sum_{k\in \K}\mathbf{1}_{\left\{I_{T_{\com}}'=\com_k\right\}}B_{T_{\com}}'(k)v_{\safe}(k)\right).
                \label{equation: reachability explicit}
            \end{align}
            Similarly, the same argument in the Approximate Commit POMDP $\POMDP'[w]$ gives
            \begin{align}
                \PP_{\sigma}^{w,b_1'}(\exists m\in\NN^*\colon\; K_m'=\top) = \EE_\sigma^{w,b_1'}\left(\mathbf{1}_{\{T_{\com}<\infty\}}\sum_{k\in \K}\mathbf{1}_{\left\{I_{T_{\com}}'=\com_k\right\}}B_{T_{\com}}'(k)w(k)\right).
                \label{equation: reachability explicit2}
            \end{align}     
            Denoting by $\EE$ the expectation with respect to their common pre-commit law, we obtain that
            \begin{align*}
                &\left|\PP_{\sigma}^{b_1'}(\exists m\in\NN^*\colon\; K_m'=\top)-\PP_{\sigma}^{w,b_1'}(\exists m\in\NN^*\colon\; K_m'=\top)\right|\\
                    &\qquad=
                        \left|\EE\left(\mathbf{1}_{\{T_{\com}<\infty\}}\sum_{k\in \K}\mathbf{1}_{\left\{I_{T_{\com}}'=\com_k\right\}}B_{T_{\com}}'(k)\left(v_{\safe}(k)-w(k)\right)\right)\right|\\
                    &\qquad\le
                        \EE\left(\mathbf{1}_{\{T_{\com}<\infty\}}\sum_{k\in \K}\mathbf{1}_{\left\{I_{T_{\com}}'=\com_k\right\}}B_{T_{\com}}'(k)\left|v_{\safe}(k)-w(k)\right|\right)\\
                    &\qquad\le
                        \EE\left(\mathbf{1}_{\{T_{\com}<\infty\}}\sum_{k\in \K}\mathbf{1}_{\left\{I_{T_{\com}}'=\com_k\right\}}B_{T_{\com}}'(k)\eps\right)\\
                    &\qquad\leq \eps,
            \end{align*}
            where the equality follows from equations \eqref{equation: reachability explicit} and \eqref{equation: reachability explicit2}, the first inequality follows from the triangle inequality, the second inequality follows from the assumption on $w$, and the last inequality follows from $B_{T_\com}'\in \Delta(\K')$.
            Taking the supremum over strategies leads to
            \[
                \left|v_R'(b_1')-v_{R,w}'(b_1')\right|\leq \eps,
            \]
            which concludes the proof.
        \end{proof}
    
    \paragraph{Previous result}
        The next lemma follows from~\cite[Lemma 5.3, p. 109]{chatterjee2022finite} and~\cite[Lemma 33, p. 2004]{venel2016strong}, where it is originally stated for the expected liminf average objective. 
        
        \begin{Lemma}
            \label{result: CSZ theorem}
            Consider a POMDP $\POMDP$.
            Then, for every initial belief $b_1\in \Delta(\K)$ and $\eps>0$, there exist $m_\eps\in \NN^*$, a strategy $\sigma_\eps\in \Sigma$, and a random belief $B^*\in \Delta(\K)$, determined by the history up to stage $m_\eps$, such that 
            \begin{itemize}
                \item 
                    $\PP_{\sigma_\eps}^{b_1}\left(\left\|B_{m_\eps}-B^*\right\|_1\leq \eps\right)\geq 1-\eps$.
                \item
                    For every realization $b^*$ of $B^*$, there exists a strategy $\sigma_{b^*}\in\Sigma$ such that, for every $k\in\supp(b^*)$,
                    \[
                        \dfrac{1}{n}\sum_{m=1}^nG_m\xrightarrow[n\to\infty]{}\gamma(\delta_k,\sigma_{b^*})\qquad\PP_{\sigma_{b^*}}^{\delta_k}\text{-almost surely}.
                    \]
                    Moreover, $\gamma(b^*,\sigma_{b^*})=v(b^*)$ and $\EE_{\sigma_\eps}^{b_1}\left(v(B^*)\right)\geq v(b_1)-\eps$.
            \end{itemize}
        \end{Lemma}
    
        For completeness, the proof of Lemma~\ref{result: CSZ theorem} is deferred to Appendix~\ref{Appendix: Proof of Lemma result: CSZ theorem}.
    
    \paragraph{Lower bound on the reachability value}
        Lemma~\ref{result: CSZ theorem} provides a strategy whose long-run average converges almost surely from every state in the support of $B^*$.
        The next lemma converts such a strategy into a strategy of the Commit POMDP that reaches the target state $\top$ with probability at least the corresponding average of its almost-sure limits.
        
        \begin{Lemma}
            \label{Lemma: connect almost sure limit and reachability}
            Consider a revealing POMDP $\POMDP$ and its Commit POMDP $\POMDP'$.
            For every pair of initial beliefs $b,b^*\in \Delta(\K)$ and for every strategy $\sigma^*\in \Sigma$ satisfying that, for every $k\in \supp(b^*)$, 
            \begin{align}
                \dfrac{1}{n}\sum_{m=1}^nG_m\xrightarrow[n\to\infty]{}\gamma(\delta_k,\sigma^*)\qquad \PP_{\sigma^*}^{\delta_k}\text{-almost surely},
                \label{equation: limit almost-sure}
            \end{align}
            there exists a strategy $\sigma'\in \Sigma'$ such that
            \[
                \PP_{\sigma'}^{b'}(\exists m\in \NN^*\colon\; K_m'=\top)\ge \sum_{k\in \supp(b^*)}b(k)\gamma(\delta_k,\sigma^*).
            \]
        \end{Lemma}
    
        \begin{proof}[Proof of Lemma~\ref{Lemma: connect almost sure limit and reachability}]
            Consider a revealing POMDP $\POMDP$ and its Commit POMDP $\POMDP'$.
            Given $m\in \NN^*$, a history $h_m=(i_1,s_2,\ldots,i_{m-1},s_m)$ is called revealing for $k$ if its last action-signal pair $(i_{m-1},s_m)$ reveals the state $k$.
            Formally, for every $b_1\in \Delta(\K)$ and $\sigma\in \Sigma$ with $\PP_{\sigma}^{b_1}(H_m=h_m)>0$, we have that
            \[
                \PP_{\sigma}^{b_1}(K_m=k\given H_m=h_m)=1.
            \]
            A history is called \emph{revealing} if it is revealing for some state $k\in \K$, and the first revealing time is defined by $T_{\rev}\defas\inf\left\{m\in \NN^*\colon\; H_m\text{ is revealing}\right\}$.
            Fix $b,b^*\in \Delta(\K)$ and $\sigma^*\in \Sigma$ satisfying \eqref{equation: limit almost-sure}.
            For every revealing history $h_m$ with $m\in \NN^*$, define
            \[
                \D(h_m)\defas\left\{k\in \supp(b^*)\colon\; \PP_{\sigma^*}^{\delta_k}(H_m=h_m)>0\right\}.
            \]
    
            \vspace{1em}
            \smallskip\noindent\textit{Independence.}
                Fix a revealing history $h_m$ such that $\D(h_m)\neq \emptyset$ and denote by $k_{h_m}$ the state revealed by $h_m$.
                We prove that, for every pair of states $k,k'\in \D(h_m)$, the following equality holds
                \begin{align}
                    \gamma(\delta_k,\sigma^*)=\gamma(\delta_{k'},\sigma^*).
                    \label{result: almost sure value is indep}
                \end{align}
                Fix a state $k\in \D(h_m)$.
                By definition of $\D(h_m)$, we have that $\PP_{\sigma^*}^{\delta_k}(H_m=h_m)>0$.
                Therefore, conditioning the almost-sure limit in equation~\eqref{equation: limit almost-sure} on the event $\{H_m=h_m\}$ gives
                \[
                    \PP_{\sigma^*}^{\delta_k}\left(\lim_{n\to\infty}\dfrac{1}{n}\sum_{j=1}^nG_j=\gamma(\delta_k,\sigma^*)\givenm H_m=h_m\right) = 1.
                \]
                Since $h_m$ reveals $k_{h_m}$, the conditional law of the play after stage $m$ given the event $\{H_m=h_m\}$ is the law induced by the continuation strategy $\sigma^*[h_m]$ from the revealed belief $\delta_{k_{h_m}}$.
                Moreover, the long-run average does not depend on the rewards of the first $m-1$ stages.
                Therefore, we deduce that
                \begin{align}
                    \PP_{\sigma^*[h_m]}^{\delta_{k_{h_m}}}\left(\lim_{n\to\infty}\dfrac{1}{n}\sum_{j=1}^nG_j=\gamma(\delta_k,\sigma^*)\right)=1.
                    \label{equation: limit after revealing history}
                \end{align}
                Fix another state $k'\in \D(h_m)$.
                The probability measure in equation~\eqref{equation: limit after revealing history} does not depend on $k$.
                Hence, equation~\eqref{equation: limit after revealing history} holds for both $k$ and $k'$, i.e., the long-run average converges $\PP_{\sigma^*[h_m]}^{\delta_{k_{h_m}}}$-almost surely to $\gamma(\delta_k,\sigma^*)$ and to $\gamma(\delta_{k'},\sigma^*)$.
                Therefore, we get that $\gamma(\delta_k,\sigma^*)=\gamma(\delta_{k'},\sigma^*)$.
    
            \vspace{1em}
            \smallskip\noindent\textit{Belief-support end-component.}
                We construct end-components of the BS-MDP from the belief-supports reachable from $\delta_{k_{h_m}}$ under the continuation strategy $\sigma^*[h_m]$.
                For every history $h_{m'}\in \H_{m'}$ with $m'\in \NN^*$ such that
                \[
                    \PP_{\sigma^*[h_m]}^{\delta_{k_{h_m}}}(H_{m'}=h_{m'})>0,
                \]
                let $b_{h_{m'}}$ be the belief induced from $\delta_{k_{h_m}}$ by $h_{m'}$ under $\sigma^*[h_m]$.
                Define the set of reachable belief-supports from $\delta_{k_{h_m}}$ under the continuation strategy $\sigma^*[h_m]$ by
                \[
                    \Q_{h_m}\defas\left\{q\in \beliefsupport\colon\; \exists h_{m'}\in \H_{m'} \text{ with } m'\in \NN^* \text{ s.t. } \PP_{\sigma^*[h_m]}^{\delta_{k_{h_m}}}(H_{m'}=h_{m'})>0\text{ and } q = \supp\left(b_{h_{m'}}\right)\right\}.
                \]
                For every $q\in \Q_{h_m}$, define
                \begin{align*}
                    \E_{h_m}(q)\defas\left\{i\in \I\colon\;\exists h_{m'}\in \H_{m'}\text{ with }m'\in \NN^*\right.
                    &\left.\text{ such that }\right.\\
                    &\left.\PP_{\sigma^*[h_m]}^{\delta_{k_{h_m}}}(H_{m'}=h_{m'})>0,\right.\\
                    &\left.\supp(b_{h_{m'}})=q,\right.\\
                    &\left.\text{and }\sigma^*[h_m\times h_{m'}](i)>0\right\}.
                \end{align*}
                We have that $\{k_{h_m}\}\in \Q_{h_m}$ and $\E_{h_m}(q)\neq \emptyset$ for every $q\in \Q_{h_m}$.
    
                We prove that $(\Q_{h_m},\E_{h_m})$ is closed in the BS-MDP.
                Indeed, fix $q\in \Q_{h_m}$, $i\in \E_{h_m}(q)$, and $q'\in \Post(q,i)$.
                By definition of $\E_{h_m}(q)$, there exists $h_{m'}\in \H_{m'}$ with $m'\in \NN^*$ such that $\PP_{\sigma^*[h_m]}^{\delta_{k_{h_m}}}(H_{m'}=h_{m'})>0$, $\supp(b_{h_{m'}})=q$, and $\sigma^*[h_m\times h_{m'}](i)>0$.
                Since $q'\in \Post(q,i)$, there exists a signal $s\in \S$ such that $q'=\psi(q,i,s)$.
                By definition of $\psi$, there exist $k\in q$ and $k'\in q'$ such that $p(k',s\given k,i) > 0$.
                Since $\supp(b_{h_{m'}})=q$, we have that $b_{h_{m'}}(k)>0$.
                Therefore, we get $\PP_{\sigma^*[h_m]}^{\delta_{k_{h_m}}}(S_{m'+1}=s\given H_{m'}=h_{m'},I_{m'}=i)\ge b_{h_{m'}}(k)p(k',s\given k,i)>0$.
                Since $i\in \E_{h_m}(q)$, we obtain that $\PP_{\sigma^*[h_m]}^{\delta_{k_{h_m}}}(H_{m'+1}=(h_{m'},i,s))>0$.
                Therefore, $(h_{m'},i,s)$ is admissible from $\delta_{k_{h_m}}$ under $\sigma^*[h_m]$.
                Finally, by definition of $\Q_{h_m}$, $q'\in \Q_{h_m}$, which proves that $(\Q_{h_m},\E_{h_m})$ is closed.
    
                Consider the finite directed graph with vertex set $\Q_{h_m}$ and an edge from $q$ to $q'$ whenever there exists an action $i\in \E_{h_m}(q)$ such that $q'\in \Post(q,i)$.
                Since $(\Q_{h_m},\E_{h_m})$ is closed, we have that $\Post(q,i)\subseteq \Q_{h_m}$ for every $q\in \Q_{h_m}$ and $i\in \E_{h_m}(q)$.
                Since $\Q_{h_m}$ is finite, it contains nonempty closed subsets that contain no strictly smaller nonempty closed subset.
                Denote all such subsets by $\Q_1,\ldots,\Q_N$.
                Therefore, for every $n\in [1\until N]$, if $q\in \Q_n$ and $i\in \E_{h_m}(q)$, then $\Post(q,i)\subseteq \Q_n$.
                For every $q\in \Q_n$, define $\E_n(q)\defas\E_{h_m}(q)$ and let $\C_n\defas(\Q_n,\E_n)$.
    
                Given $n\in [1\until N]$, we now prove that $\C_n$ is an end-component of the BS-MDP.
                By the definition of $\Q_n$ and $\E_n$, we have that $\C_n=(\Q_n,\E_n)$ is closed.
                It remains to prove that $\C_n$ is strongly connected.
                Fix $q\in\Q_n$ and consider the set of belief-supports in $\Q_n$ reachable from $q$, denoted by $\R(q)$.
                Since $q\in \R(q)$, we have that $\R(q)\neq \emptyset$.
                Moreover, we have that $\R(q)$ is closed.
                Indeed, fix an arbitrary $\another{q}\in \R(q)$, $i\in \E_n(\another{q})$, and $\another{q}'\in \Post(\another{q},i)$.
                Since $\Q_n$ is closed, $\another{q}'\in \Q_n$ and there is an edge from $\another{q}$ to $\another{q}'$.
                Since $\another{q}$ is reachable from $q$, it follows that $\another{q}'$ is also reachable from $q$.
                Therefore $\another{q}'\in \R(q)$.
                Moreover, if $\R(q)\neq \Q_n$, then $\R(q)$ would be a strict closed subset of $\Q_n$, which contradicts the definition of $\Q_n$.
                Therefore $\R(q)=\Q_n$.
                Since $q\in \Q_n$ was taken arbitrary, we obtain that $\C_n\defas(\Q_n,\E_n)$ is an end-component of the BS-MDP.
                Since the BS-MDP is finite, for every $n\in[1\until N]$ fix a maximal end-component $\M_n$ such that $\C_n\preceq\M_n$.
    
            \vspace{1em}
            \smallskip\noindent\textit{Lower bound on the maximum safe value.}
                Fix $\overline{k}\in \D(h_m)$.
                We prove that, for every $n\in [1\until N]$ and $k\in \K$ such that $\{k\}\in \Q_n$,
                \begin{align}
                    v_{\safe}(k)\ge \gamma(\delta_{\overline{k}},\sigma^*).
                    \label{equation: relation of max-safe value with sigma*}
                \end{align}
                Fix $n\in [1\until N]$ and a history $h_{m'}$ such that $\PP_{\sigma^*[h_m]}^{\delta_{k_{h_m}}}(H_{m'}=h_{m'})>0$ and $\supp(b_{h_{m'}})\in \Q_n$.
    
                We first prove by induction on $m''\in \NN^*$ that, for every history $h_{m''}\in\H_{m''}$ such that
                $\PP_{\sigma^*[h_m\times h_{m'}]}^{b_{h_{m'}}}(H_{m''}=h_{m''})>0$, the induced belief-support $Q$ belongs to $\Q_n$ and every action $i$ satisfying $\sigma^*[h_m\times h_{m'}\times h_{m''}](i)>0$ belongs to $\E_n(Q)$.
                This implies that $\sigma^*[h_m\times h_{m'}]$ is $\C_n$-safe from $b_{h_{m'}}$.
                For the base case, $\H_1=\{\emptyset\}$ and the belief-support is in $\Q_n$.
                Moreover, if $\sigma^*[h_m\times h_{m'}](i)>0$, then, by definition of $\E_{h_m}$, we have that $i\in\E_{h_m}(\supp(b_{h_{m'}}))=\E_n(\supp(b_{h_{m'}}))$.
                For the induction case, suppose that the statement holds at some stage $m''$ and fix $h_{m''+1}=h_{m''}\times (i,s)$ such that
                $\PP_{\sigma^*[h_m\times h_{m'}]}^{b_{h_{m'}}}(H_{m''+1}=h_{m''+1})>0$.
                Consider the belief-support $Q$ induced by $h_{m''}$.
                The positive-probability assumption implies that $\PP_{\sigma^*[h_m\times h_{m'}]}^{b_{h_{m'}}}(H_{m''}=h_{m''})>0$ and $\sigma^*[h_m\times h_{m'}\times h_{m''}](i)>0$.
                Hence, by the induction hypothesis, $Q\in \Q_n$ and $i\in \E_n(Q)$.
                Moreover, the belief-support induced by $h_{m''+1}$ is $Q'=\psi(Q,i,s)\in \Post(Q,i)$.
                Since $\C_n$ is closed, we have that $\Post(Q,i)\subseteq \Q_n$ and thus $Q'\in \Q_n$.
                Finally, we have that
                \[
                    \PP_{\sigma^*[h_m]}^{\delta_{k_{h_m}}}\left(H_{m'+m''}=h_{m'}\times h_{m''+1}\right)>0.
                \]
                Therefore, if $\sigma^*[h_m\times h_{m'}\times h_{m''+1}](\another{i})>0$, then $\another{i}\in \E_{h_m}(Q')=\E_n(Q')$, which proves the induction.
    
                Conditioning \eqref{equation: limit after revealing history} on $\{H_{m'}=h_{m'}\}$ gives the same almost-sure limit under the finite mixture with belief $b_{h_{m'}}$.
                Since $b_{h_{m'}}(\hat{k})>0$ for every $\hat{k}\in\supp(b_{h_{m'}})$, the limit event has probability one under every component of that mixture. Hence,
                for every $\hat{k}\in \supp(b_{h_{m'}})$,
                \begin{align}
                    \dfrac{1}{n}\sum_{j=1}^nG_j\xrightarrow[n\to\infty]{}\gamma(\delta_{\overline{k}},\sigma^*)\qquad\PP_{\sigma^*[h_m\times h_{m'}]}^{\delta_{\hat{k}}}\text{-almost surely}.
                    \label{equation: limit after continuation history}
                \end{align}
                Therefore, the almost-sure limit under $\sigma^*[h_m\times h_{m'}]$ does not depend on the current state $\hat{k}\in \supp(b_{h_{m'}})$.
                We obtain that
                \begin{align*}
                     &v_{\safe}(k)\\
                        &\quad \ge v_{\M_n}(\delta_{k})
                            &&(\text{def. of }v_{\safe})\\
                        &\quad = v_{\M_n}(b_{h_{m'}})
                            &&(\text{by Lemma }\ref{Result: Independence of the Safe Value})\\
                        &\quad \ge \gamma_{\M_n}\left(b_{h_{m'}},\sigma^*[h_{m}\times h_{m'}]\right)
                            &&(\C_n\preceq\M_n\text{ and }\sigma^*[h_{m}\times h_{m'}] \text{ is }\C_n\text{-safe})\\
                        &\quad = \gamma\left(b_{h_{m'}},\sigma^*[h_{m}\times h_{m'}]\right)
                            &&(\text{def. of }\gamma_{\M_n} \text{ and }\gamma)\\
                        &\quad = \sum_{\hat{k}\in\K}b_{h_{m'}}(\hat{k})\gamma\left(\delta_{\hat{k}},\sigma^*[h_{m}\times h_{m'}]\right)
                            &&(\text{linearity and dominated convergence thm})\\
                        &\quad = \sum_{\hat{k}\in\K}b_{h_{m'}}(\hat{k})\gamma(\delta_{\overline{k}},\sigma^*)
                            &&(\text{by Eq. }\eqref{equation: limit after continuation history})\\
                        &\quad = \gamma(\delta_{\overline{k}},\sigma^*).
                            &&(b_{h_{m'}}\in \Delta(\K))
                \end{align*}
                Since $n\in [1\until N]$ was taken arbitrary, the inequality holds for every end-component $\C_n$ with $n\in[1\until N]$.
    
            \vspace{1em}
            \smallskip\noindent\textit{Paths to end-components.}
                Since $\Q_{h_m}$ is finite, every belief-support in $\Q_{h_m}$ has a path to one of the subsets $\Q_1,\ldots,\Q_N$.
                For every belief-support $q\in \Q_{h_m}$ such that $q\notin \bigcup_{n=1}^N\Q_n$, choose a shortest path from $q_0=q$ to $\bigcup_{n=1}^N\Q_n$.
                Formally,
                \[
                    q_0=q,\ldots,q_L
                \]
                with actions $i_0,\ldots,i_{L-1}$ and signals $s_1,\ldots,s_L$ such that $q_L\in \Q_n$ for some $n\in [1\until N]$, $i_a\in \E_{h_m}(q_a)$, $q_{a+1}=\psi(q_a,i_a,s_{a+1})$ for every $a\in [0\until L-1]$.
                Therefore, we have that $L\le |\Q_{h_m}|$.
    
            \vspace{1em}
            \smallskip\noindent\textit{Strategy construction.}
                We construct a strategy $\sigma'\in \Sigma'$ as follows:
                \begin{itemize}
                    \item
                        Until the first revealing history $h_m$ is observed, the strategy $\sigma'$ follows $\sigma^*$ using only actions from $\POMDP$.
                        Fix $k\in \supp(b^*)$.
                        Since $\POMDP$ is revealing, a revealing signal occurs with probability at least $p_{\min}$ at every stage, and thus, under $\PP_{\sigma^*}^{\delta_k}$, a revealing signal is observed almost surely.
                        Consider a revealing history $h_m$ having positive probability under $\PP_{\sigma^*}^{\delta_k}$.
                        By definition of $\D(h_m)$, we have that $k\in \D(h_m)$.
                        Therefore, by equation \eqref{result: almost sure value is indep}, for every $\another{k}\in \D(h_m)$,
                        \[
                            \gamma(\delta_{\another{k}},\sigma^*)=\gamma(\delta_k,\sigma^*).
                        \]
                    \item
                        For every revealing history $h_m$ such that $\D(h_m)\neq \emptyset$, fix $\overline{k}\in \D(h_m)$ and switch to the continuation strategy $\sigma_{h_m}'\in \Sigma'$, defined as follows.
                        Since $h_m$ reveals $k_{h_m}$, the belief at the beginning of $\sigma_{h_m}'$ is $\delta_{k_{h_m}}$, with belief-support $q_0=\{k_{h_m}\}$.
                        If $q_0\in\bigcup_{n=1}^N\Q_n$, then the strategy has already reached an end-component and proceeds as described below.
                        Otherwise, consider the path $q_0,\ldots,q_L$, with actions $i_0,\ldots,i_{L-1}$ and signals $s_1,\ldots,s_L$, fixed above for the belief-support $q_0$.
                        Starting from $q_0$, the strategy $\sigma_{h_m}'$ plays the pure action $i_a$ whenever the current belief-support is $q_a$ and all the previously observed belief-supports in the current attempt were $q_0,\ldots,q_a$.
                        If the next belief-support is $q_{a+1}=\psi(q_a,i_a,s_{a+1})$, the attempt continues; if this occurs for every $a\in[0\until L-1]$, the attempt succeeds and reaches $q_L\in\Q_n$ for some $n\in[1\until N]$.
                        If, after playing $i_a$ at $q_a$, a belief-support different from $q_{a+1}$ is observed, then the attempt fails.
                        Since the action $i_a$ belongs to $\E_{h_m}(q_a)$ and $(\Q_{h_m},\E_{h_m})$ is closed, the new belief-support $q'$ belongs to $\Q_{h_m}$.
                        If $q'\in\bigcup_{n=1}^N\Q_n$, then the strategy has reached an end-component and proceeds as described below; otherwise, it immediately starts the path fixed above for $q'$.
                    \item
                        If the first revealing history $h_m$ satisfies $\D(h_m)=\emptyset$, then the strategy $\sigma'$ plays arbitrarily.
                \end{itemize}
    
                Assume that $\sigma'_{h_m}$ reaches a support in $\Q_n$ and consider the end-component $\C_n=(\Q_n,\E_n)$.
                Until observing a revealing signal, if the current support is $q$, then the strategy chooses an action $i\in \E_{h_m}(q)=\E_n(q)$.
                Since $\C_n$ is closed, this strategy is $\C_n$-safe.
                Moreover, since a revealing signal occurs with probability at least $p_{\min}$ at every stage, a revealing signal is observed almost surely.
                Since the strategy is $\C_n$-safe, if the state $k$ is revealed, then the current belief-support $\{k\}$ is in $\Q_n$.
                Then, the strategy $\sigma'_{h_m}$ plays $\com_k$.
                By definition of the Commit POMDP, \eqref{equation: relation of max-safe value with sigma*} gives $v_{\safe}(k)\ge \gamma(\delta_{\overline{k}},\sigma^*)$.
                Therefore, we obtain
                \begin{align}
                    \PP_{\sigma_{h_m}'}^{\delta_{k_{h_m}}}\left(\exists m'\in \NN^*\colon K'_{m'}=\top\right)\ge \gamma(\delta_{\overline{k}},\sigma^*)
                    \label{equation: reachability after revealing history}
                \end{align}
                After reaching $\top$ or $\bot$, the strategy $\sigma_{h_m}'$ plays arbitrarily.
    
                Define the first time at which a support enters one of the end-components by
                \[
                    T_\C\defas\inf\left\{m\in\NN^*\colon Q_m\in\bigcup_{n=1}^N\Q_n\right\}.
                \]
                We prove that a support in some $\Q_n$ is reached almost surely.
                At every stage before $T_\C$, the current belief-support $q$ belongs to $\Q_{h_m}$ and the strategy chooses an action in $\E_{h_m}(q)$.
                By the revealing property, the next belief is a revealed belief with conditional probability at least $p_{\min}$.
                Suppose that the revealed belief is $\delta_k$.
                If $\{k\}\in\bigcup_{n=1}^N\Q_n$, then an end-component has been reached.
                Otherwise, from the belief-support $\{k\}$, the strategy either continues the remaining part of the path currently being followed, if $\{k\}$ is the expected next belief-support, or starts the path fixed above for $\{k\}$.
                In both cases, denote the resulting path prescribed by the strategy by $q_0=\{k\},\ldots,q_L$, together with its actions $i_0,\ldots,i_{L-1}$ and signals $s_1,\ldots,s_L$.
                We have that $q_L\in\bigcup_{n=1}^N\Q_n$ and $L\le|\Q_{h_m}|$.
                Fix $k_L\in q_L$.
                By definition of $\psi$, there exist $k_0=k,k_1,\ldots,k_{L-1}$ such that $P_{k_a,k_{a+1}}(i_a,s_{a+1})>0$ for every $a\in[0\until L-1]$.
                Therefore, conditionally on the revealed belief $\delta_k$, the probability of following the chosen path and reaching $q_L\in\bigcup_{n=1}^N\Q_n$ is at least
                \[
                    p_{\min}^L\ge p_{\min}^{|\Q_{h_m}|}.
                \]
                It follows that, conditionally on every history before $T_\C$, the probability of reaching a support in $\bigcup_{n=1}^N\Q_n$ within the next $|\Q_{h_m}|+1$ stages is at least $p_{\min}^{|\Q_{h_m}|+1}$.
                Therefore, by induction on $\ell\in\NN$, we have that
                \[
                    \PP_{\sigma'_{h_m}}^{\delta_{k_{h_m}}}\left(T_\C>1+\ell\left(|\Q_{h_m}|+1\right)\right)
                    \le\left(1-p_{\min}^{|\Q_{h_m}|+1}\right)^\ell.
                \]
                Taking the limit as $\ell$ tends to infinity and using continuity from above gives
                \[
                    \PP_{\sigma'_{h_m}}^{\delta_{k_{h_m}}}(T_\C=\infty)
                    \le\lim_{\ell\to\infty}\left(1-p_{\min}^{|\Q_{h_m}|+1}\right)^\ell=0.
                \]
                Therefore, a support in some $\Q_n$ is reached almost surely.
    
                We deduce that
                \begin{align*}
                &\PP_{\sigma'}^{b'}(\exists m'\in\NN^*\colon K'_{m'}=\top)\\
                    &\quad=\sum_{k\in\K}b(k)\PP_{\sigma'}^{\delta_k}(\exists m'\in\NN^*\colon K'_{m'}=\top)
                        &&\text{(conditioning)}\\
                    &\quad\ge\sum_{k\in\supp(b^*)}b(k)\PP_{\sigma'}^{\delta_k}(\exists m'\in\NN^*\colon K'_{m'}=\top)
                        &&\text{(nonnegativity)}\\
                    &\quad=\sum_{k\in\supp(b^*)}b(k)\\
                    &\qquad\quad\cdot\sum_{m\ge2}\sum_{h_m\in\H_m}\PP_{\sigma^*}^{\delta_k}(T_{\rev}=m,H_m=h_m)\PP_{\sigma'_{h_m}}^{\delta_{k_{h_m}}}(\exists m'\in\NN^*\colon K'_{m'}=\top)
                        &&(\text{def. of } \sigma_{h_m}')\\
                    &\quad\ge\sum_{k\in\supp(b^*)}b(k)\sum_{m\ge2}\sum_{h_m\in\H_m}\PP_{\sigma^*}^{\delta_k}(T_{\rev}=m,H_m=h_m)\gamma(\delta_{\overline{k}},\sigma^*)
                        &&\text{(by Eq. \eqref{equation: reachability after revealing history})}\\
                    &\quad=\sum_{k\in\supp(b^*)}b(k)\sum_{m\ge2}\sum_{h_m\in\H_m}\PP_{\sigma^*}^{\delta_k}(T_{\rev}=m,H_m=h_m)\gamma(\delta_k,\sigma^*).
                        &&\text{(by Eq. \eqref{result: almost sure value is indep})}\\
                    &\quad=\sum_{k\in\supp(b^*)}b(k)\gamma(\delta_k,\sigma^*)
                        &&(T_{\rev}<\infty\;\PP_{\sigma^*}^{\delta_k}\text{-a.s.})
                \end{align*}
                which concludes the proof.
        \end{proof}
    
    \paragraph{Reduction}
        We now present the reduction of revealing POMDPs with long-run average objectives to Commit POMDPs with reachability objectives.
        
        \begin{Lemma}\label{Result: Reduction to Reachability Objectives}
            Consider a revealing POMDP $\POMDP$ with long-run average objectives and its Commit POMDP $\POMDP'$ with reachability objectives to the target state $\X\defas\{\top\}$.
            Then, for every $b_1\in \Delta(\K)$,
            \[
                v(b_1)=v_R'(b_1').  
            \]
        \end{Lemma}
        
        \begin{proof}[Proof of Lemma~\ref{Result: Reduction to Reachability Objectives}]
        
            Consider a revealing POMDP $\POMDP$ with initial belief $b_1\in \Delta(\K)$, its Commit POMDP $\POMDP'$ and its BS-MDP $\MDP_B$.
            Recall that
            \[
                v_{\safe}(k)\defas\max\{v_\C(\delta_k)\colon\; \C=(\Q,\E)\in \mathfrak{C} \text{ and } \{k\}\in \Q\}.
            \]
        
            \vspace{1em}
            \smallskip\noindent\textit{First inequality.}
                We prove that, for every $b_1\in \Delta(\K)$, 
                \begin{align}
                    v(b_1)\ge v_R'(b_1')
                    \label{equation: inequality 1 of reduction}
                \end{align}
                By the remark on pure strategies in the preliminaries, it is enough to fix an arbitrary pure strategy $\sigma'\in \Sigma'$ on the Commit POMDP $\POMDP'$ and a parameter $\eps>0$.
                Define the first commit time in the Commit POMDP $\POMDP'$ by
                \[
                    T_{\com}\defas\inf\left\{m\in \NN^*\colon \; I_m'\in \I_{\mathsf{com}}\right\}.
                \]
                with $T_{\com}=\infty$ if no commit action is ever played. On $\{T_{\com}=\infty\}$, define $K_{T_{\com}+1}'$ arbitrarily.
                
                The event $\{T_{\com}=m\}$ means that the controller has not played a commit action before stage $m$.
                Therefore, the histories and beliefs in $\POMDP$ and $\POMDP'$ up to stage $m$ are the same.
                In particular, the history $H_m'$ (resp., the belief $B_m'$) can be identified with $H_m$ (resp., $B_m$) in $\POMDP$.
                Moreover, on the event $\left\{T_{\com}=m,I'_m=\com_k\right\}$, the action $\com_k$ sends the process to $\top$ with probability $v_{\safe}(k)$ if the hidden state is $k$.
                Since the strategy selects $I_m'$ using only $H_m'$, for every $\another{k}\in\K$,
                \begin{equation}
                    \PP_{\sigma'}^{b_1'}\left(K_m'=\another{k}\givenm T_{\com}=m,H_m',I_m'=\com_k\right)=B_m'(\another{k}).
                    \label{equation: identity for belief in commit}
                \end{equation}        
        
                By construction of the Commit POMDP, the state $\top$ can be reached only after a commit action. 
                Since $\top$ and $\bot$ are absorbing, we get
                \begin{align}
                    &\PP_{\sigma'}^{b_1'}(\exists m \geq 1\colon K_m'=\top)\nonumber\\
                        & = \PP_{\sigma'}^{b_1'}(K_{T_{\com}+1}'=\top,T_{\com}<\infty)
                            &&\text{(def. of \(T_{\com}\))}\nonumber\\
                        & = \sum_{m\in \NN^*}\PP_{\sigma'}^{b_1'}(K_{m+1}'=\top,T_{\com}=m)
                            &&\text{(partition)}\nonumber\\
                        & = \sum_{m\in \NN^*}\sum_{k\in \K}\PP_{\sigma'}^{b_1'}\left(K_{m+1}'=\top, T_{\com}=m,I_{m}'=\com_k\right)
                            &&\text{(partition)}\nonumber\\
                        & = \sum_{m\in \NN^*}\sum_{k\in \K}\EE_{\sigma'}^{b_1'}\left(\mathbf{1}_{\left\{T_{\com}=m, I_m'=\com_k\right\}}\right.\nonumber\\
                        &\qquad\qquad\qquad\left.\PP_{\sigma'}^{b_1'}\left(K_{m+1}'=\top\givenm T_{\com}=m, H_m', I_{m}'=\com_k\right) \right)
                            &&\text{(tower rule)}\nonumber\\
                        & = \sum_{m\in \NN^*}\sum_{k\in \K}\EE_{\sigma'}^{b_1'}\left(\mathbf{1}_{\left\{T_{\com}=m, I_m'=\com_k\right\}}\right.\nonumber\\
                        &\quad\left.\cdot\sum_{\another{k}\in \K}\PP_{\sigma'}^{b_1'}\left(K_m'=\another{k}\givenm T_{\com}=m, H_m', I_{m}'=\com_k\right)p'(\top,\top\given \another{k},\com_k) \right)
                            &&\text{(Markov property)}\nonumber\\
                        & = \sum_{m\in \NN^*}\sum_{k\in \K}\EE_{\sigma'}^{b_1'}\left(\mathbf{1}_{\left\{T_{\com}=m, I_m'=\com_k\right\}}\sum_{\another{k}\in \K}B_m'(\another{k})v_{\safe}(k)\mathbf{1}_{\left\{\another{k}=k\right\}} \right)
                            &&\text{(by Eq.~\eqref{equation: identity for belief in commit} and def. of \(p'\))}\nonumber\\
                        & = \EE_{\sigma'}^{b_1'}\left(\sum_{m\in \NN^*}\sum_{k\in \K}\mathbf{1}_{\left\{T_{\com}=m, I_m'=\com_k\right\}}B_m'(k)v_{\safe}(k) \right).
                        \label{Equation: random time reachability}
                \end{align}
                The tail random variables below converge pointwise to $0$ and are bounded by $1$, because at most one summand is nonzero. Therefore, by the dominated convergence theorem, there exists $N\in \NN^*$ such that
                \[
                    \EE_{\sigma'}^{b_1'}\left(\sum_{m>N}\sum_{k\in \K}\mathbf{1}_{\left\{T_{\com}=m\right\}}\mathbf{1}_{\left\{I_m'=\com_k\right\}}B_m'(k)v_{\safe}(k)\right)\le \eps.
                \]
                
                We now construct a strategy $\sigma\in \Sigma$ in the revealing POMDP $\POMDP$.
                First, for every $k\in\K$, choose a strategy $\sigma_k\in \Sigma$ such that
                \begin{equation}
                    \gamma(\delta_k,\sigma_k)\geq v_{\safe}(k)-\eps.\label{Equation: existence of eps opt safe strategies}
                \end{equation}
                Indeed, by definition of $v_{\safe}(k)$, if $v_{\safe}(k)>0$, then there exists an end-component $\C=(\Q,\E)$ such that $\{k\} \in \Q$ and $v_\C(\delta_k)=v_{\safe}(k)$. We choose an $\eps$-optimal $\C$-safe strategy. 
                Moreover, if $v_{\safe}(k)=0$ then every strategy satisfies inequality \eqref{Equation: existence of eps opt safe strategies} and thus, a strategy $\sigma_k$ can be chosen arbitrarily.
                
                The strategy $\sigma$ in $\POMDP$ is defined as follows:
                \begin{itemize}
                    \item 
                        Up to stage $N$, $\sigma$ plays the same action as in $\sigma'$ until a commit action is played;
                    \item 
                        If the strategy $\sigma'$ selects a commit action $\com_k$ at some stage $m\leq N$, then the strategy $\sigma$ switches to the continuation strategy $\sigma_k$;
                    \item
                        If no commit action is played during stages $m\in [1\until N]$, then $\sigma$ plays arbitrarily from stage $N+1$ onward.        
                \end{itemize}
                
                We obtain that, for every $n\geq N$,
                \begin{align*}
                    \EE_{\sigma}^{b_1}\left(\dfrac{1}{n}\sum_{m=1}^nG_m\right)
                        \geq \sum_{m=1}^N\sum_{k\in \K}
                            \PP_{\sigma'}^{b_1'}(T_{\com}=m,I_m'=\com_k,K_m'=k)\dfrac{n-m+1}{n}\gamma_{n-m+1}(\delta_k,\sigma_k).
                \end{align*}
                Taking the limit inferior as $n\to\infty$,
                \begin{align*}
                    &\gamma(b_1,\sigma)\\
                        &\quad \geq \sum_{m=1}^N\sum_{k\in \K}
                        \PP_{\sigma'}^{b_1'}(T_{\com}=m,I_m'=\com_k,K_m'=k)
                        \gamma(\delta_k,\sigma_k)\\
                        &\quad \geq \sum_{m=1}^N\sum_{k\in \K}
                        \PP_{\sigma'}^{b_1'}(T_{\com}=m,I_m'=\com_k,K_m'=k)
                        (v_{\safe}(k)-\eps)
                            &&\text{(def. of \(\sigma_k\))}\\
                        &\quad = \sum_{m=1}^N\sum_{k\in \K}
                        \PP_{\sigma'}^{b_1'}(T_{\com}=m,I_m'=\com_k,K_m'=k)
                        v_{\safe}(k) \\
                        &\quad\qquad - \sum_{m=1}^N\sum_{k\in \K}
                        \PP_{\sigma'}^{b_1'}(T_{\com}=m,I_m'=\com_k,K_m'=k)
                        \eps\\
                        &\quad \geq \sum_{m=1}^N\sum_{k\in \K}
                        \PP_{\sigma'}^{b_1'}(T_{\com}=m,I_m'=\com_k,K_m'=k)
                        v_{\safe}(k)-\eps\\
                        &\quad = \EE_{\sigma'}^{b_1'}\left(
                        \sum_{m=1}^N\sum_{k\in \K}
                        \mathbf{1}_{\{T_{\com}=m\}}
                        \mathbf{1}_{\{I_m'=\com_k\}}
                        B_m'(k)v_{\safe}(k)\right)-\eps
                            &&\text{(def. of $B_m'$)}\\
                        &\quad = \PP_{\sigma'}^{b_1'}
                        (\exists m\in \NN^*\colon K_m'=\top)\\
                        &\quad\qquad - \EE_{\sigma'}^{b_1'}\left(
                        \sum_{m>N}\sum_{k\in \K}
                        \mathbf{1}_{\{T_{\com}=m\}}
                        \mathbf{1}_{\{I_m'=\com_k\}}
                        B_m'(k)v_{\safe}(k)\right)-\eps
                            &&\text{(eq.~\eqref{Equation: random time reachability})}\\
                        &\quad \geq \PP_{\sigma'}^{b_1'}
                        (\exists m\in \NN^*\colon K_m'=\top)-2\eps.
                            &&\text{(choice of \(N\))}
                \end{align*}
                Thus, for every pure $\sigma'\in\Sigma'$,
                \[
                    v(b_1)\ge \PP_{\sigma'}^{b_1'}(\exists m\in\NN^*\colon K_m'=\top)-2\eps.
                \]
                Taking the supremum over pure $\sigma'$ and using the preliminary remark gives $v(b_1)\ge v_R'(b_1')-2\eps$.
                Letting $\eps\to 0$ yields \eqref{equation: inequality 1 of reduction}.
        
            \vspace{1em}
            \smallskip\noindent\textit{Second inequality.}
                We prove that, for every $b_1\in \Delta(\K)$,
                \begin{align}
                    v(b_1)\le v_R'(b_1').
                    \label{result: second inequality coupling}
                \end{align}            
                
                Fix $\eps>0$. 
                By Lemma \ref{result: CSZ theorem}, there exist $m_\eps\in\NN^*$, a strategy $\sigma_\eps\in\Sigma$, and a belief $B^*\in\Delta(\K)$ determined by $H_{m_\eps}$, such that
                
                \begin{equation}
                    \label{equation: close beliefs}
                    \PP_{\sigma_\eps}^{b_1} \left(\|B_{m_\eps}-B^*\|_1\leq\eps\right)\ge 1-\eps
                \end{equation}
                and
                \begin{equation}
                    \label{equation: value lower bound}
                    \EE_{\sigma_\eps}^{b_1}[v(B^*)]
                    \geq v(b_1)-\eps.
                \end{equation}
                For every realization $b^*$ of $B^*$, consider a strategy $\sigma_{b^*}\in\Sigma$ such that, for every
                $k\in\supp(b^*)$,
                \[
                    \dfrac{1}{n}\sum_{m=1}^{n}G_m \xrightarrow[n\to\infty]{} \gamma(\delta_k,\sigma_{b^*}) \qquad \PP_{\sigma_{b^*}}^{\delta_k}\text{-almost surely},
                \]
                and
                \begin{equation}\label{equation: value preservation}
                    v(b^*)
                        = \gamma(b^*,\sigma_{b^*})
                        = \sum_{k\in\supp(b^*)}b^*(k)\gamma(\delta_k,\sigma_{b^*}).
                \end{equation}
                Since $B^*$ is determined by $H_{m_\eps}$, there exists a mapping $\phi_\eps\colon \H_{m_\eps}\to\Delta(\K)$ such that $B^*=\phi_\eps(H_{m_\eps})$ $\PP_{\sigma_\eps}^{b_1}$-almost surely.
                Define $\H_{m_\eps}^{\sigma_\eps}(b_1)\defas\{h_{m_\eps}\in \H_{m_\eps}\colon \PP_{\sigma_\eps}^{b_1}(H_{m_\eps}=h_{m_\eps})>0\}$.
                Given a history $h_{m_\eps}\in \H_{m_\eps}^{\sigma_\eps}(b_1)$, we will use the following notation 
                \[
                    b_{h_{m_\eps}}\defas b_{h_{m_\eps}}^{b_1}\text{ and }b_{h_{m_\eps}}^*=\phi_{\eps}(h_{m_\eps}).
                \]
                
                For every history $h_{m_\eps}\in \H_{m_\eps}^{\sigma_\eps}(b_1)$, Lemma~\ref{Lemma: connect almost sure limit and reachability}, applied with parameters $b=b_{h_{m_\eps}}$, $b^*=b_{h_{m_\eps}}^*$, and $\sigma^*=\sigma_{b_{h_{m_\eps}}^*}$, yields a continuation strategy $\sigma'_{h_{m_\eps}}\in \Sigma'$ such that
                \begin{align}
                    &\PP_{\sigma_{h_{m_\eps}}'}^{b_{h_{m_\eps}}'}(\exists m\in \NN^*\colon\; K_m'=\top)\nonumber\\
                        &\qquad \ge \sum_{k\in\supp\left(b_{h_{m_\eps}}^*\right)}b_{h_{m_\eps}}(k)\gamma\left(\delta_k,\sigma_{b_{h_{m_\eps}}^*}\right)
                            &&\text{(Lemma~\ref{Lemma: connect almost sure limit and reachability})}\nonumber\\
                        &\qquad = \sum_{k\in\supp\left(b_{h_{m_\eps}}^*\right)}\left(b_{h_{m_\eps}}(k)-b_{h_{m_\eps}}^*(k)+b_{h_{m_\eps}}^*(k)\right)\gamma\left(\delta_k,\sigma_{b_{h_{m_\eps}}^*}\right)\nonumber\\
                        &\qquad = \sum_{k\in\supp\left(b_{h_{m_\eps}}^*\right)}b_{h_{m_\eps}}^*(k)\gamma\left(\delta_k,\sigma_{b_{h_{m_\eps}}^*}\right)\nonumber\\
                        &\qquad\qquad
                            +
                            \sum_{k\in\supp\left(b_{h_{m_\eps}}^*\right)}\left(b_{h_{m_\eps}}(k)-b_{h_{m_\eps}}^*(k)\right)\gamma\left(\delta_k,\sigma_{b_{h_{m_\eps}}^*}\right)
                                &&\text{(linearity)}\nonumber\\
                        &\qquad \ge \sum_{k\in\supp\left(b_{h_{m_\eps}}^*\right)}b_{h_{m_\eps}}^*(k)\gamma\left(\delta_k,\sigma_{b_{h_{m_\eps}}^*}\right)-\left\|b_{h_{m_\eps}}-b_{h_{m_\eps}}^*\right\|_1
                            &&\text{($\gamma(\cdot)\in [0,1]$)}\nonumber\\
                        &\qquad = v\left(b_{h_{m_\eps}}^*\right)-\left\|b_{h_{m_\eps}}-b_{h_{m_\eps}}^*\right\|_1.
                            &&\text{(by Eq.~\eqref{equation: value preservation})}
                        \label{equation: continuation reachability and long run}
                \end{align}
                We construct a strategy $\sigma'\in \Sigma'$ as follows:
                \begin{itemize}
                    \item 
                        Up to stage $m_\eps-1$, the strategy $\sigma'$ follows the strategy $\sigma_\eps$.
                    \item 
                        At stage $m_\eps$, when observing $h_{m_\eps}\in \H_{m_\eps}^{\sigma_\eps}(b_1)$, the strategy switches to $\sigma'_{h_{m_\eps}}$.
                    \item 
                        After histories not belonging to $\H_{m_\eps}^{\sigma_\eps}(b_1)$, the strategy $\sigma'$ is defined arbitrarily.
                \end{itemize}
    
                Since no commit action is played before stage $m_\eps$, the state $\top$ cannot be reached before that stage.
                Moreover, $\sigma'$ in $\POMDP'$ and $\sigma_\eps$ in $\POMDP$ induce the same distribution over $H_{m_\eps}$.
                Therefore, for every $b_1\in \Delta(\K)$,
                \begin{align*}
                    &\PP_{\sigma'}^{b_1'}(\exists m\in\NN^*\colon K_m'=\top)\\
                    &\quad =\sum_{h_{m_\eps}\in\H_{m_\eps}^{\sigma_\eps}(b_1)}\PP_{\sigma_\eps}^{b_1}(H_{m_\eps}=h_{m_\eps})\PP_{\sigma'_{h_{m_\eps}}}^{b_{h_{m_\eps}}'}(\exists m\in\NN^*\colon K_m'=\top)
                        &&\text{(expectation)}\\
                    &\quad =\EE_{\sigma_\eps}^{b_1}\left(\PP_{\sigma'_{H_{m_\eps}}}^{B_{m_\eps}'}(\exists m\in\NN^*\colon K_m'=\top)\right)
                        &&\text{(def. of $\sigma_\eps$)}\\
                    &\quad \geq\EE_{\sigma_\eps}^{b_1}\left(\mathbf{1}_{\{\|B_{m_\eps}-B^*\|_1\leq\eps\}}\left(v(B^*)-\|B_{m_\eps}-B^*\|_1\right)\right)
                        &&\text{(by Eq. \eqref{equation: continuation reachability and long run})}\\
                    &\quad \geq\EE_{\sigma_\eps}^{b_1}\left(\mathbf{1}_{\{\|B_{m_\eps}-B^*\|_1\leq\eps\}}v(B^*)\right)-\eps\\
                    &\quad =\EE_{\sigma_\eps}^{b_1}\left(v(B^*)\right)-\EE_{\sigma_\eps}^{b_1}\left(\mathbf{1}_{\{\|B_{m_\eps}-B^*\|_1>\eps\}}v(B^*)\right)-\eps
                        &&\text{(decomposition)}\\
                    &\quad \geq\EE_{\sigma_\eps}^{b_1}\left(v(B^*)\right)-\PP_{\sigma_\eps}^{b_1}\left(\left\|B_{m_\eps}-B^*\right\|_1>\eps\right)-\eps
                        &&\text{($0\leq v(B^*)\leq 1$)}\\
                    &\quad \geq\EE_{\sigma_\eps}^{b_1}\left(v(B^*)\right)-2\eps
                        &&\text{(by Eq. \eqref{equation: close beliefs})}\\
                    &\quad \geq v(b_1)-3\eps
                        &&\text{(by Eq. \eqref{equation: value lower bound}).}
                \end{align*}
                Therefore, for every $b_1\in \Delta(\K)$,
                \[
                    v_R'(b_1')\ge v(b_1)-3\eps.
                \]
                Since $\eps>0$ was taken arbitrary, we conclude that, for every $b_1\in \Delta(\K)$,
                \[
                    v_R'(b_1')\ge v(b_1).
                \]
                By combining \eqref{equation: inequality 1 of reduction} and \eqref{result: second inequality coupling}, we obtain that, for every $b_1\in \Delta(\K)$, $v_{R}'(b_1')=v(b_1)$, which concludes the proof.
        \end{proof}
    

\subsection{Proof of the EXPTIME Upper Bound}\label{Section: Proof of Theorem}
    
    We can now prove the first item of \Cref{Result: Approximating the uniform value for revealing POMDPs is decidable}.
    
    \begin{proof}[Proof of the first item of \Cref{Result: Approximating the uniform value for revealing POMDPs is decidable}]
    
        Consider a revealing POMDP $\POMDP$ with initial belief $b_1\in \Delta(\K)$ and $\eps>0$.
        Since the statement trivially holds for $\eps\geq1$, we assume that $\eps\in (0,1)$.
        Set $\eta\defas\eps/4$.
        By~\Cref{Theorem: approximation of max safe value is EXPTIME}, in exponential time we can compute rational numbers $\widehat v(k)$, for all $k\in\K$, such that
        \[
            \max_{k\in\K}|\widehat v(k)-v_{\safe}(k)|\leq\eta.
        \]
        Replacing each $\widehat v(k)$ by its clipping to $[0,1]$ does not increase this error, so we henceforth assume $\widehat v(k)\in[0,1]$.
        Let $L\defas\lceil\log_2(4/\eps)\rceil$ and, for every $k\in\K$, round $\widehat v(k)$ to a dyadic number $w(k)\in[0,1]$ satisfying $|w(k)-\widehat v(k)|\leq2^{-L}$.
        Then
        \[
            \max_{k\in\K}|w(k)-v_{\safe}(k)|
            \leq\eta+2^{-L}\leq\eps/2.
        \]
        Consider the Approximate Commit POMDP $\POMDP'[w]$.
        Lemmas~\ref{Results: approximating reachability in Commit POMDP} and~\ref{Result: Reduction to Reachability Objectives} yield
        \[
            |v_{R,w}'(b_1')-v(b_1)|
            =|v_{R,w}'(b_1')-v_R'(b_1')|
            \leq\eps/2.
        \]
    
        By~\cite{asadi2026revealing}, in exponential time we can compute a number $z$ satisfying
        \[
            |z-v_{R,w}'(b_1')|\leq\eps/2,
        \]
        because $\POMDP'[w]$ is revealing.
        Consequently,
        \[
            |z-v(b_1)|
            \leq |z-v_{R,w}'(b_1')|+|v_{R,w}'(b_1')-v(b_1)|
            \leq\eps.
        \]
    
        It remains to check that the two exponential-time procedures compose without an exponential blow-up in the encoding passed to the second procedure.
        The finite-horizon construction in the proof of~\Cref{Theorem: approximation of max safe value is EXPTIME} uses rational arithmetic over exponentially many backups.
        At each successive backup, the numerator and denominator lengths grow by at most a polynomial amount; since the horizon and the number of backups are at most exponential, the rational arithmetic and the rounding above can be performed in exponential time.
        Each $w(k)$ has $O(\log(1/\eps))$ bits.
        Thus, $\POMDP'[w]$ has $|\K|+2$ states and an encoding length polynomial in the encoding length of $\POMDP$ and $\log(1/\eps)$.
        Only these rounded weights, rather than the possibly much longer intermediate values $\widehat v(k)$, are passed to the reachability algorithm.
        Hence the second procedure also takes time exponential in the original input size, and the full approximation algorithm runs in exponential time.
    \end{proof}


\section{Proof of the EXPTIME Lower Bound}\label{Section: Hardness}

This section proves the second item of \Cref{Result: Approximating the uniform value for revealing POMDPs is decidable} by a reduction from the almost-sure safety problem in POMDPs.
The reduction is inspired by the $\mathrm{EXPTIME}$-hardness proof for revealing POMDPs with parity objectives~\cite{belly2025revelations}.

\paragraph{Safety objective}
    Consider a POMDP $\POMDP=(\K,\I,\S,p,g)$, an initial state $k_1\in\K$, and a set of safe states $\F\subseteq\K$ with $k_1 \in \F$.
    Given a strategy $\sigma\in\Sigma$, the safety objective is
    \[
        \PP_{\sigma}^{\delta_{k_1}}\left(\forall m\in\NN^*,\;K_m\in\F\right).
    \]
    A strategy is almost-sure winning for the safety objective if this probability is one.
    The almost-sure safety problem asks whether such a strategy exists.\\

    We first recall the following classical result on POMDPs with safety objectives~\cite{chatterjee2016decidable}.

    \begin{Lemma}\label{Result: almost-sure safety hardness}
        The almost-sure safety problem for POMDPs is $\mathrm{EXPTIME}$-complete.
    \end{Lemma}

    The following lemma gives the finite-horizon consequence that we use in the reduction.

    \begin{Lemma}
        \label{Result: finite-horizon safety gap}
        Consider a POMDP $\POMDP$, an initial state $k_1\in \K$, and a set of safe states $\F\subseteq \K$ with $k_1\in \F$.
        Assume that no strategy is almost-sure winning for the safety objective.
        Then, there exist a horizon $n\in \NN^*$ and a constant $\eta>0$ such that, for every strategy $\sigma\in \Sigma$,
        \[
            \PP_{\sigma}^{\delta_{k_1}}\left(\exists m\in[1\until n]\colon\;K_m\notin \F\right)\geq\eta.
        \]
    \end{Lemma}

    \begin{proof}[Proof of Lemma~\ref{Result: finite-horizon safety gap}]
        Consider a POMDP $\POMDP$, an initial state $k_1\in \K$, and a set of safe states $\F\subseteq \K$ with $k_1\in \F$.
        For every horizon $n\in \NN^*$, define the minimal escape probability within the first $n$ stages by
        \begin{align}
            \alpha_n\defas\inf_{\sigma\in \Sigma}\PP_{\sigma}^{\delta_{k_1}}\left(\exists m\in[1\until n]\colon\;K_m\notin \F\right).
            \label{equation: escape probability}
        \end{align}
        The conclusion of the lemma holds if and only if there exists a horizon $n\in \NN^*$ such that $\alpha_n>0$.
        Indeed, if $\alpha_n>0$, then the horizon $n$ and the constant $\eta\defas\alpha_n$ satisfy the conclusion, and conversely every horizon $n$ and constant $\eta>0$ satisfying the conclusion give $\alpha_n\ge \eta>0$.
        Therefore, we prove the contrapositive, i.e., we assume that
        \begin{align}
            \alpha_n=0\qquad\text{for every }n\in \NN^*,
            \label{equation: vanishing escape probability}
        \end{align}
        and we construct an almost-sure winning strategy for the safety objective.

        We first prove that, for every horizon $n\in \NN^*$, there exists a pure strategy $\sigma\in \Sigma$ such that
        \begin{align}
            \PP_{\sigma}^{\delta_{k_1}}\left(\forall m\in[1\until n],\;K_m\in \F\right)=1.
            \label{equation: n-stage safe pure strategy}
        \end{align}
        Fix a horizon $n\in \NN^*$ and denote by $\Sigma_n$ the set of restrictions of the pure strategies to the histories before stage $n$.
        The probability in \eqref{equation: escape probability} depends on the strategy only through its restriction to the histories before stage $n$.
        Moreover, by the remark on pure strategies in the preliminaries, restricting the infimum in \eqref{equation: escape probability} to pure strategies does not change its value.
        Since the sets of actions and signals are finite, the set $\Sigma_n$ is finite, and thus the infimum $\alpha_n$ is attained by a pure strategy $\sigma$.
        By \eqref{equation: vanishing escape probability}, we deduce that
        \[
            \PP_{\sigma}^{\delta_{k_1}}\left(\exists m\in[1\until n]\colon\;K_m\notin \F\right)=\alpha_n=0,
        \]
        which proves \eqref{equation: n-stage safe pure strategy}.

        We now construct the strategy.
        Consider the tree whose nodes at depth $n\in \NN^*$ are the restrictions in $\Sigma_n$ of the pure strategies satisfying \eqref{equation: n-stage safe pure strategy}, and in which the parent of a node at depth $n+1$ is its restriction to the histories before stage $n$.
        The tree is well defined because every pure strategy satisfying \eqref{equation: n-stage safe pure strategy} at the horizon $n+1$ also satisfies it at the horizon $n$.
        Moreover, the tree has a node at every depth by \eqref{equation: n-stage safe pure strategy}, and it is finitely branching because the set $\Sigma_n$ is finite for every $n\in \NN^*$.
        Therefore, by K\"onig's lemma, the tree has an infinite branch $\left(\sigma_n^*\right)_{n\in \NN^*}$ with $\sigma_n^*\in \Sigma_n$ for every $n\in \NN^*$.
        Since $\sigma_{n+1}^*$ restricts to $\sigma_n^*$ for every $n\in \NN^*$, the branch defines a pure strategy $\sigma^*\in \Sigma$ that agrees with $\sigma_n^*$ on the histories before stage $n$.
        Since every $\sigma_n^*$ is the restriction of a pure strategy satisfying \eqref{equation: n-stage safe pure strategy}, and since the event in \eqref{equation: n-stage safe pure strategy} depends on the strategy only through its restriction to the histories before stage $n$, we obtain that, for every horizon $n\in \NN^*$,
        \begin{align}
            \PP_{\sigma^*}^{\delta_{k_1}}\left(\forall m\in[1\until n],\;K_m\in \F\right)=1.
            \label{equation: safe branch}
        \end{align}
        Finally, we deduce that
        \begin{align*}
            \PP_{\sigma^*}^{\delta_{k_1}}\left(\forall m\in \NN^*,\;K_m\in \F\right)
                &=\lim_{n\to\infty}\PP_{\sigma^*}^{\delta_{k_1}}\left(\forall m\in[1\until n],\;K_m\in \F\right)
                    &&\text{(continuity from above)}\\
                &=1.
                    &&\text{(by Eq.~\eqref{equation: safe branch})}
        \end{align*}
        Therefore, the strategy $\sigma^*$ is almost-sure winning for the safety objective, which proves the contrapositive and concludes the proof.
    \end{proof}

\paragraph{Reduction}
    Consider a POMDP $\POMDP=(\K,\I,\S,p,g)$, an initial state $k_1\in \K$, and a set of safe states $\F\subseteq \K$ with $k_1\in \F$.
    We define the POMDP $\another{\POMDP}=\left(\K,\I,\another{\S},\another{p},\another{g}\right)$ by
    \begin{itemize}
        \item
            $\K$ and $\I$ are the states and the actions of $\POMDP$, and every state in $\K\setminus \F$ is absorbing;
        \item
            $\another{\S}=\S\cup\left\{\rev_k\colon\; k\in \K\right\}\cup\left\{\reset\right\}$, where we assume without loss of generality that the signals $\rev_k$ with $k\in \K$ and the signal $\reset$ do not belong to $\S$;
        \item
            $\another{p}\colon \K\times \I\to \Delta\left(\K\times \another{\S}\right)$ is the transition function defined by
            \begin{itemize}
                \item
                    For every $k\in \F$, $k'\in \K$, $i\in \I$, and $s\in \S$,
                    \begin{align*}
                        \another{p}(k',s\given k,i)&=\dfrac{1}{3}p(k',s\given k,i),\\
                        \another{p}(k',\rev_{k'}\given k,i)&=\dfrac{1}{3}\sum_{s'\in \S}p(k',s'\given k,i),\\
                        \another{p}(k_1,\reset\given k,i)&=\dfrac{1}{3};
                    \end{align*}
                \item
                    For every $k\in \K\setminus \F$ and $i\in \I$, $\another{p}(k,\rev_k\given k,i)=1$;
                \item
                    All the transition probabilities not specified above are zero;
            \end{itemize}
        \item
            $\another{g}\colon \K\times \I\to [0,1]$ is the stage reward defined by $\another{g}(k,i)=\mathbf{1}_{\left\{k\in \F\right\}}$.
    \end{itemize}
    \noindent Intuitively, from a safe state $k\in \F$, after playing an action $i\in \I$, the transition of $\another{\POMDP}$ has three branches, each selected with probability $1/3$:
    \begin{enumerate}
        \item
            The original branch follows $p(\,\cdot \given k,i)$ and the controller observes the original signal;
        \item
            The revealing branch follows the marginal state transition of $p(\,\cdot \given k,i)$ and the controller observes the signal $\rev_{k'}$, which reveals the successor state $k'$;
        \item
            The reset branch moves to the initial state $k_1$ and the controller observes the signal $\reset$, which reveals the initial state $k_1$ because the successor of this branch is always $k_1$.
    \end{enumerate}
    The construction is polynomial in the size of $\POMDP$.
    The set of strategies in $\another{\POMDP}$ is denoted by $\another{\Sigma}$, and the matrices of $\another{\POMDP}$, defined as in \Cref{Section: Preliminaries}, are denoted by $\another{P}$.
    Given an initial belief $b_1\in \Delta(\K)$ and a strategy $\another{\sigma}\in \another{\Sigma}$, we denote by $\another{\PP}_{\another{\sigma}}^{b_1}$ the probability measure induced by $\another{\sigma}$ from $b_1$ in $\another{\POMDP}$, by $\another{\EE}_{\another{\sigma}}^{b_1}$ the corresponding expectation, by $\another{\gamma}$ the long-run average objective of $\another{\POMDP}$, and by $\another{v}$ its long-run average value.

    The POMDP $\another{\POMDP}$ is revealing.
    Indeed, fix an action $i\in \I$ and states $k,k'\in \K$ such that $\sum_{\another{s}\in \another{\S}}\another{P}_{k,k'}(i,\another{s})>0$.
    If $k\in \F$ and $\sum_{s\in \S}p(k',s\given k,i)>0$, then the definition of $\another{p}$ gives that $\another{P}_{k,k'}(i,\rev_{k'})>0$ and that the signal $\rev_{k'}$ is observed only when the successor state is $k'$, i.e., $\another{P}_{\underline{k},\overline{k}}\left(i,\rev_{k'}\right)=0$ for every $\underline{k}\in \K$ and $\overline{k}\in \K\setminus\{k'\}$.
    If $k\in \F$ and $\sum_{s\in \S}p(k',s\given k,i)=0$, then the definition of $\another{p}$ gives that $k'=k_1$, that $\another{P}_{k,k_1}(i,\reset)=1/3>0$, and that the signal $\reset$ is observed only when the successor state is $k_1$.
    If $k\in \K\setminus \F$, then the definition of $\another{p}$ gives that $k'=k$, that $\another{P}_{k,k}(i,\rev_k)=1>0$, and that the signal $\rev_k$ is observed only when the successor state is $k$.
    Therefore, the revealing property holds for every action and every feasible transition of $\another{\POMDP}$.

\begin{proof}[Proof of the second item of \Cref{Result: Approximating the uniform value for revealing POMDPs is decidable}]
    Consider a POMDP $\POMDP$, an initial state $k_1\in \K$, a set of safe states $\F\subseteq \K$ with $k_1\in \F$, and the revealing POMDP $\another{\POMDP}$ constructed above.
    We prove that $\another{v}(\delta_{k_1})=1$ if $\POMDP$ is almost-sure winning for the safety objective, and that $\another{v}(\delta_{k_1})=0$ otherwise.

    \vspace{1em}
    \smallskip\noindent\textit{$\POMDP$ is almost-sure winning.}
        Assume that $\POMDP$ is almost-sure winning for the safety objective.
        Denote by $\W\subseteq \beliefsupport$ the set of belief-supports from which $\POMDP$ is almost-sure winning, i.e., $q\in \W$ if and only if there exists a strategy $\sigma\in \Sigma$ such that $\PP_{\sigma}^{b}\left(\forall m\in \NN^*,\; K_m\in \F\right)=1$ for every belief $b\in \Delta(\K)$ with $\supp(b)=q$.
        By~\cite{chatterjee2016decidable}, the almost-sure safety problem is decided on the belief-support MDP $\MDP_B$, i.e., for every $q\in \W$ there exists an action $\iota(q)\in \I$ such that
        \begin{align}
            q\subseteq \F
            \qquad\text{and}\qquad
            \Post\left(q,\iota(q)\right)\subseteq \W.
            \label{equation: winning belief-supports}
        \end{align}
        Since $\POMDP$ is almost-sure winning from $\delta_{k_1}$, we have that $\{k_1\}\in \W$.
        Moreover, the set $\W$ is downward closed.
        Indeed, fix $q\in \W$ and a nonempty belief-support $q'\subseteq q$.
        Every strategy that remains in $\F$ almost surely from the beliefs with support $q$ also remains in $\F$ almost surely from the beliefs with support $q'$, and thus $q'\in \W$.

        Consider the strategy $\another{\sigma}\in \another{\Sigma}$ that plays the action $\iota(Q_m)$ at every stage $m\in \NN^*$, where $Q_m\defas \supp(B_m)$ denotes the current belief-support.
        We prove by induction on $m\in \NN^*$ that $Q_m\in \W$ holds $\another{\PP}_{\another{\sigma}}^{\delta_{k_1}}$-almost surely.
        For the base case, we have that $Q_1=\{k_1\}\in \W$.
        For the induction case, assume that $Q_m=q\in \W$ and write $i\defas \iota(q)$.
        Since $q\subseteq \F$ by \eqref{equation: winning belief-supports}, the definition of $\another{p}$ gives that every belief-support occurring at stage $m+1$ with positive probability is of one of the following three forms.
        \begin{itemize}
            \item
                $\psi(q,i,s)$ with $s\in \S$, which is obtained through the original branch.
                Then, $\psi(q,i,s)\in \Post(q,i)\subseteq \W$ by \eqref{equation: winning belief-supports}.
            \item
                $\{k'\}$ with $k'\in \K$, which is obtained through the revealing branch.
                Then, there exists a signal $s\in \S$ such that $k'\in \psi(q,i,s)$, and thus $\{k'\}\subseteq \psi(q,i,s)\in \W$.
                Since $\W$ is downward closed, we get that $\{k'\}\in \W$.
            \item
                $\{k_1\}$, which is obtained through the reset branch, and $\{k_1\}\in \W$.
        \end{itemize}
        Therefore, we have that $Q_{m+1}\in \W$, which proves the induction.

        Since $K_m\in Q_m$ almost surely and $q\subseteq \F$ for every $q\in \W$, we deduce that $K_m\in \F$ for every stage $m\in \NN^*$, $\another{\PP}_{\another{\sigma}}^{\delta_{k_1}}$-almost surely.
        Hence, $\another{g}(K_m,I_m)=1$ for every $m\in \NN^*$ and thus $\another{\gamma}\left(\delta_{k_1},\another{\sigma}\right)=1$.
        Since $\another{g}(\cdot)\in [0,1]$, we conclude that $\another{v}(\delta_{k_1})=1$.

    \vspace{1em}
    \smallskip\noindent\textit{$\POMDP$ is not almost-sure winning.}
        Assume that $\POMDP$ is not almost-sure winning for the safety objective.
        Fix the horizon $n\in \NN^*$ and the constant $\eta>0$ given by \Cref{Result: finite-horizon safety gap}, and fix an arbitrary strategy $\another{\sigma}\in \another{\Sigma}$.
        Define the first unsafe stage by
        \[
            T_{\unsafe}\defas\inf\left\{m\in \NN^*\colon\; K_m\notin \F\right\}.
        \]
        We partition the horizon into blocks of length $n$, i.e., block $j\in \NN^*$ is the set of stages $[m_j\until m_{j+1}-1]$ with $m_j\defas 1+(j-1)n$.
        In particular, $m_{j+1}=m_j+n$ for every $j\in \NN^*$.

        We first establish the following one-block inequality.
        For every block index $j\in \NN^*$ and admissible history $h_{m_j}\in \H_{m_j}(\delta_{k_1})$ such that $\another{\PP}_{\another{\sigma}}^{\delta_{k_1}}\left(H_{m_j}=h_{m_j},T_{\unsafe}>m_j\right)>0$,
        \begin{align}
            \another{\PP}_{\another{\sigma}}^{\delta_{k_1}}\left(T_{\unsafe}\le m_{j+1}\givenm H_{m_j}=h_{m_j},T_{\unsafe}>m_j\right)\ge \eta\, 3^{-n}.
            \label{equation: one-block escape}
        \end{align}
        Since the state at stage $m_j$ belongs to $\F$, the definition of $\another{p}$ gives that the reset branch is selected at stage $m_j$ with conditional probability $1/3$, i.e.,
        \[
            \another{\PP}_{\another{\sigma}}^{\delta_{k_1}}\left(S_{m_j+1}=\reset\givenm H_{m_j}=h_{m_j},T_{\unsafe}>m_j\right)=\dfrac{1}{3}.
        \]
        The signal $\reset$ reveals the initial state $k_1$, and thus the belief at stage $m_j+1$ is $\delta_{k_1}$.
        Moreover, the definition of $\another{p}$ gives that, at every stage at which the current state belongs to $\F$, the original branch is selected with conditional probability $1/3$ and is the only branch whose signal belongs to $\S$.
        Therefore, by induction on $r\in [2\until n]$, we obtain that
        \begin{align*}
            &\another{\PP}_{\another{\sigma}}^{\delta_{k_1}}\left(\bigcap_{r=2}^{n}\left(\left\{S_{m_j+r}\in \S\right\}\cup\left\{T_{\unsafe}<m_j+r\right\}\right)\right.\\
            &\qquad\qquad\qquad\qquad\left.\givenm S_{m_j+1}=\reset,H_{m_j}=h_{m_j},T_{\unsafe}>m_j\right)\ge 3^{-(n-1)}.
        \end{align*}
        On the intersection of the two events above, the states $\left(K_{m_j+r}\right)_{r\in [1\until n]}$ follow the dynamics of $\POMDP$ from the initial state $k_1$, as long as they belong to $\F$.
        Hence, the continuation of $\another{\sigma}$ induces a strategy $\sigma\in \Sigma$ such that the conditional law of $\left(K_{m_j+r}\right)_{r\in [1\until n]}$ coincides with the law of $\left(K_r\right)_{r\in [1\until n]}$ under $\PP_{\sigma}^{\delta_{k_1}}$.
        Since $m_j+n=m_{j+1}$, \Cref{Result: finite-horizon safety gap} gives that
        \begin{align*}
            &\another{\PP}_{\another{\sigma}}^{\delta_{k_1}}\left(T_{\unsafe}\le m_{j+1}\givenm \bigcap_{r=2}^{n}\left(\left\{S_{m_j+r}\in \S\right\}\cup\left\{T_{\unsafe}<m_j+r\right\}\right),\right.\\
            &\qquad\qquad\qquad\qquad\left.S_{m_j+1}=\reset,H_{m_j}=h_{m_j},T_{\unsafe}>m_j\right)\ge \eta.
        \end{align*}
        Multiplying the three inequalities above yields inequality \eqref{equation: one-block escape}.

        We now prove by induction on $r\in \NN$ that
        \begin{align}
            \another{\PP}_{\another{\sigma}}^{\delta_{k_1}}\left(T_{\unsafe}>m_{r+1}\right)\le \left(1-\eta\, 3^{-n}\right)^{r}.
            \label{equation: block escape}
        \end{align}
        We start by observing that the base case holds.
        When $r=0$, we have that $\another{\PP}_{\another{\sigma}}^{\delta_{k_1}}\left(T_{\unsafe}>m_1\right)\le 1=\left(1-\eta\, 3^{-n}\right)^0$ and thus \eqref{equation: block escape} holds.
        Assume now that \eqref{equation: block escape} holds for some $r\in \NN$.
        Since $\left\{T_{\unsafe}>m_{r+2}\right\}\subseteq \left\{T_{\unsafe}>m_{r+1}\right\}$, if $\another{\PP}_{\another{\sigma}}^{\delta_{k_1}}\left(T_{\unsafe}>m_{r+1}\right)=0$, then \eqref{equation: block escape} holds for $r+1$.
        Otherwise, using again this inclusion and the tower rule, we have that
        \begin{align*}
            \another{\PP}_{\another{\sigma}}^{\delta_{k_1}}\left(T_{\unsafe}>m_{r+2}\right)
                &=\left(1-\another{\PP}_{\another{\sigma}}^{\delta_{k_1}}\left(T_{\unsafe}\le m_{r+2}\givenm T_{\unsafe}>m_{r+1}\right)\right)\\
                &\qquad\qquad \cdot\another{\PP}_{\another{\sigma}}^{\delta_{k_1}}\left(T_{\unsafe}>m_{r+1}\right)\\
                &\le \left(1-\eta\, 3^{-n}\right)\another{\PP}_{\another{\sigma}}^{\delta_{k_1}}\left(T_{\unsafe}>m_{r+1}\right)
                    &&\text{(by Eq.~\eqref{equation: one-block escape})}\\
                &\le \left(1-\eta\, 3^{-n}\right)^{r+1},
                    &&\text{(induction hyp.)}
        \end{align*}
        which proves inequality \eqref{equation: block escape}.
        Then, since $m_{r+1}=1+rn$ tends to infinity, we deduce that
        \begin{align*}
            \another{\PP}_{\another{\sigma}}^{\delta_{k_1}}\left(T_{\unsafe}=\infty\right)
                &=\lim_{r\to\infty}\another{\PP}_{\another{\sigma}}^{\delta_{k_1}}\left(T_{\unsafe}>m_{r+1}\right)
                    &&\text{(continuity from above)}\\
                &\le \lim_{r\to\infty}\left(1-\eta\, 3^{-n}\right)^{r}
                    &&\text{(by Eq.~\eqref{equation: block escape})}\\
                &=0.
                    &&\text{($\eta\, 3^{-n}>0$)}
        \end{align*}
        In particular, we have that $T_{\unsafe}<\infty$ $\another{\PP}_{\another{\sigma}}^{\delta_{k_1}}$-almost surely.

        Since every state in $\K\setminus \F$ is absorbing and $\another{g}(k,i)=0$ for every $k\in \K\setminus \F$ and $i\in \I$, no reward is obtained from stage $T_{\unsafe}$ on, i.e., for every horizon $N\in \NN^*$,
        \begin{align}
            \dfrac{1}{N}\sum_{m=1}^N\another{g}(K_m,I_m)\le \min\left\{1,\dfrac{T_{\unsafe}}{N}\right\}.
            \label{equation: average after unsafe}
        \end{align}
        Therefore, we obtain that
        \begin{align*}
            \another{\gamma}\left(\delta_{k_1},\another{\sigma}\right)
                &=\liminf_{N\to\infty}\another{\EE}_{\another{\sigma}}^{\delta_{k_1}}\left(\dfrac{1}{N}\sum_{m=1}^N\another{g}(K_m,I_m)\right)
                    &&\text{(def. of $\another{\gamma}$)}\\
                &\le \liminf_{N\to\infty}\another{\EE}_{\another{\sigma}}^{\delta_{k_1}}\left(\min\left\{1,\dfrac{T_{\unsafe}}{N}\right\}\right)
                    &&\text{(by Eq.~\eqref{equation: average after unsafe})}\\
                &=0.
                    &&\text{(dominated convergence and $T_{\unsafe}<\infty$ a.s.)}
        \end{align*}
        Since $\another{g}(\cdot)\in [0,1]$ and the strategy $\another{\sigma}\in \another{\Sigma}$ was taken arbitrary, we conclude that $\another{v}(\delta_{k_1})=0$.

    \vspace{1em}
    \smallskip\noindent\textit{Conclusion.}
        The two cases above give that $\another{v}(\delta_{k_1})=1$ if $\POMDP$ is almost-sure winning for the safety objective and $\another{v}(\delta_{k_1})=0$ otherwise.
        Fix $\eps<1/2$ and consider $\overline{v}$ such that $\left|\overline{v}-\another{v}(\delta_{k_1})\right|\le \eps$.
        Then, $\overline{v}>1/2$ in the first case and $\overline{v}<1/2$ in the second case, and thus $\overline{v}$ decides the almost-sure safety problem.
        Since the construction of $\another{\POMDP}$ is polynomial in the size of $\POMDP$ and the almost-sure safety problem is $\mathrm{EXPTIME}$-hard by \Cref{Result: almost-sure safety hardness}, this proves the second item of \Cref{Result: Approximating the uniform value for revealing POMDPs is decidable}.
\end{proof}

\section*{Conclusion}
    This paper considered the problem of approximating the long-run average value in revealing POMDPs.
    We proved that approximating the long-run average value in revealing POMDPs is $\mathrm{EXPTIME}$-complete.
    Potential directions for future research include considering POMDPs with infinite sets.


\section*{Acknowledgements}
    This research was partially supported by Austrian Science Fund (FWF) 10.55776/COE12 and by the ERC CoG 863818 (ForM-SMArt) grant.
    We thank Raimundo Saona for helpful discussions during the early stages of this work.

\bibliographystyle{alpha}
\bibliography{Biblio}

@article{shani2013survey,
title={A survey of point-based {POMDP} solvers},
author={Shani, Guy and Pineau, Joelle and Kaplow, Robert},
journal={Autonomous Agents and Multi-Agent Systems},
volume={27},
number={1},
pages={1--51},
year={2013},
publisher={Springer}
}

@article{rabin1963probabilistic,
title={Probabilistic automata},
author={Rabin, Michael O.},
journal={Information and Control},
volume={6},
number={3},
pages={230--245},
year={1963},
publisher={Elsevier}
}

@article{chatterjee2016decidable,
title={What is decidable about partially observable {M}arkov decision processes with {$\omega$}-regular objectives?},
author={Chatterjee, Krishnendu and Chmelik, Martin and Tracol, Mathieu},
journal={Journal of Computer and System Sciences},
volume={82},
number={5},
pages={878--911},
year={2016},
publisher={Elsevier}
}

@article{smallwood1973optimal,
title={The optimal control of partially observable {M}arkov processes over a finite horizon},
author={Smallwood, Richard D. and Sondik, Edward J.},
journal={Operations Research},
volume={21},
number={5},
pages={1071--1088},
year={1973},
publisher={INFORMS}
}

@article{venel2016strong,
title={Strong uniform value in gambling houses and partially observable {M}arkov decision processes},
author={Venel, Xavier and Ziliotto, Bruno},
journal={SIAM Journal on Control and Optimization},
volume={54},
number={4},
pages={1983--2008},
year={2016},
publisher={SIAM}
}

@article{papadimitriou1987complexity,
title={The complexity of {M}arkov decision processes},
author={Papadimitriou, Christos H. and Tsitsiklis, John N.},
journal={Mathematics of Operations Research},
volume={12},
number={3},
pages={441--450},
year={1987},
publisher={INFORMS}
}

@article{madani2003undecidability,
title={On the undecidability of probabilistic planning and related stochastic optimization problems},
author={Madani, Omid and Hanks, Steve and Condon, Anne},
journal={Artificial Intelligence},
volume={147},
number={1-2},
pages={5--34},
year={2003},
publisher={Elsevier}
}

@article{chatterjee2022finite,
title={Finite-memory strategies in {POMDP}s with long-run average objectives},
author={Chatterjee, Krishnendu and Saona, Raimundo and Ziliotto, Bruno},
journal={Mathematics of Operations Research},
volume={47},
number={1},
pages={100--119},
year={2022},
publisher={INFORMS}
}

@inproceedings{wang2019inventory,
title={Inventory control with partially observable states},
author={Wang, Erli and Kurniawati, Hanna and Kroese, Dirk},
booktitle={23rd International Congress on Modelling and Simulation (MODSIM2019)},
pages={200--206},
address={Canberra, Australia},
organization={Modelling and Simulation Society of Australia and New Zealand},
year={2019}
}

@article{kaelbling1996reinforcement,
title={Reinforcement learning: A survey},
author={Kaelbling, Leslie Pack and Littman, Michael L. and Moore, Andrew W.},
journal={Journal of Artificial Intelligence Research},
volume={4},
pages={237--285},
year={1996}
}

@inproceedings{gimbert2014deciding,
title={Deciding the Value 1 Problem for {$\sharp$}-acyclic Partially Observable {M}arkov Decision Processes},
author={Gimbert, Hugo and Oualhadj, Youssouf},
booktitle={International Conference on Current Trends in Theory and Practice of Informatics},
pages={281--292},
year={2014},
organization={Springer}
}

@article{kaelbling1998planning,
title={Planning and acting in partially observable stochastic domains},
author={Kaelbling, Leslie Pack and Littman, Michael L. and Cassandra, Anthony R.},
journal={Artificial Intelligence},
volume={101},
number={1-2},
pages={99--134},
year={1998},
publisher={Elsevier}
}

@article{arapostathis1993discrete,
title={Discrete-time controlled {M}arkov processes with average cost criterion: a survey},
author={Arapostathis, Aristotle and Borkar, Vivek S. and Fern{\'a}ndez-Gaucherand, Emmanuel and Ghosh, Mrinal K. and Marcus, Steven I.},
journal={SIAM Journal on Control and Optimization},
volume={31},
number={2},
pages={282--344},
year={1993},
publisher={SIAM}
}

@article{rosenberg2002blackwell,
title={{B}lackwell optimality in {M}arkov decision processes with partial observation},
author={Rosenberg, Dinah and Solan, Eilon and Vieille, Nicolas},
journal={The Annals of Statistics},
pages={1178--1193},
year={2002},
publisher={Institute of Mathematical Statistics}
}

@book{bertsekas1976DynamicProgrammingStochastic,
title = {Dynamic Programming and Stochastic Control},
author = {Bertsekas, Dimitri P.},
year = {1976},
series = {Mathematics in Science and Engineering},
number = {v. 125},
publisher = {Academic Press},
address = {New York},
isbn = {978-0-12-093250-4}
}

@article{renault2017long,
title={Long-term values in {M}arkov decision processes and repeated games, and a new distance for probability spaces},
author={Renault, J{\'e}r{\^o}me and Venel, Xavier},
journal={Mathematics of Operations Research},
volume={42},
number={2},
pages={349--376},
year={2017},
publisher={INFORMS}
}

@book{paz1971introduction,
title={Introduction to Probabilistic Automata},
author={Paz, Azaria},
year={1971},
publisher={Academic Press},
address={New York},
series={Computer Science and Applied Mathematics}
}

@inproceedings{chatterjee2010probabilistic,
title={Probabilistic automata on infinite words: Decidability and undecidability results},
author={Chatterjee, Krishnendu and Henzinger, Thomas A.},
booktitle={International Symposium on Automated Technology for Verification and Analysis},
pages={1--16},
year={2010},
organization={Springer}
}

@book{puterman1994,
title={{M}arkov Decision Processes: Discrete Stochastic Dynamic Programming},
author={Puterman, Martin L.},
year={1994},
publisher={John Wiley \& Sons},
address={New York}
}

@article{feinberg1996,
title={On measurability and representation of strategic measures in {M}arkov decision processes},
author={Feinberg, Eugene A.},
journal={Lecture Notes-Monograph Series},
pages={29--43},
year={1996},
publisher={Institute of Mathematical Statistics}
}

@article{chatterjee2026approximating,
title={Approximating the Uniform Value in Hidden Stochastic Games with {D}oeblin Conditions},
author={Chatterjee, Krishnendu and Lurie, David and Saona, Raimundo and Ziliotto, Bruno},
journal={arXiv preprint arXiv:2602.06480},
year={2026}
}

@phdthesis{de1997formal,
title = {{Formal Verification of Probabilistic Systems}},
author = {de Alfaro, Luca},
year = {1997},
type = {{Ph.D.} diss.},
school = {Stanford University},
address = {Stanford, CA, USA}
}

@inproceedings{belly2025revelations,
title={Revelations: a decidable class of {POMDP}s with {$\omega$}-regular objectives},
author={Belly, Marius and Fijalkow, Nathana{\"e}l and Gimbert, Hugo and Horn, Florian and P{\'e}rez, Guillermo A. and Vandenhove, Pierre},
booktitle={Proceedings of the AAAI Conference on Artificial Intelligence},
volume={39},
pages={26454--26462},
year={2025}
}

@inproceedings{chatterjee2014partial,
title={Partial-observation stochastic reachability and parity games},
author={Chatterjee, Krishnendu},
booktitle={International Symposium on Mathematical Foundations of Computer Science},
pages={1--4},
year={2014},
organization={Springer}
}

@article{chatterjee2013survey,
title = {A survey of partial-observation stochastic parity games},
author = {Chatterjee, Krishnendu and Doyen, Laurent and Henzinger, Thomas A.},
year = {2013},
journal = {Formal Methods in System Design},
volume = {43},
pages = {268--284},
publisher = {Springer}
}

@article{chatterjee2012survey,
title={A Survey of Stochastic {$\omega$}-Regular Games},
author={Chatterjee, Krishnendu and Henzinger, Thomas A.},
journal={Journal of Computer and System Sciences},
volume={78},
number={2},
pages={394--413},
year={2012},
publisher={Elsevier}
}

@inproceedings{avrachenkov2025constrained,
title={Constrained Average-Reward Intermittently Observable {MDP}s},
author={Avrachenkov, Konstantin and Dhiman, Madhu and Kavitha, Veeraruna},
booktitle={2025 IEEE 64th Conference on Decision and Control (CDC)},
pages={338--344},
year={2025},
organization={IEEE}
}

@misc{chenIntermittentlyObservableMarkov2023,
title = {{Intermittently Observable Markov Decision Processes}},
author = {Chen, Gongpu and Liew, Soung-Chang},
year = {2023},
howpublished = {arXiv:2302.11761},
eprint = {2302.11761},
archiveprefix = {arXiv}
}

@inproceedings{asadi2026revealing,
title={Revealing {POMDP}s: Qualitative and quantitative analysis for parity objectives},
author={Asadi, Ali and Chatterjee, Krishnendu and Lurie, David and Saona, Raimundo},
booktitle={Proceedings of the AAAI Conference on Artificial Intelligence},
volume={40},
pages={36146--36154},
year={2026}
}

@article{VZ21,
title={History-dependent Evaluations in Partially Observable {M}arkov Decision Process},
author={Venel, Xavier and Ziliotto, Bruno},
journal={SIAM Journal on Control and Optimization},
volume={59},
number={2},
pages={1730--1755},
year={2021},
publisher={SIAM}
}

@article{NS10,
title={Repeated games with public uncertain duration process},
author={Neyman, Abraham and Sorin, Sylvain},
journal={International Journal of Game Theory},
volume={39},
pages={29--52},
year={2010},
publisher={Springer}
}

@article{chatterjee2025ergodic,
title = {Uniform value and decidability in ergodic blind stochastic games},
author = {Chatterjee, Krishnendu and Lurie, David and Saona, Raimundo and Ziliotto, Bruno},
journal = {Mathematics of Operations Research},
year = {2025},
publisher = {INFORMS}
}

@article{ni2009sensor,
  author  = {Ni, Kevin and Ramanathan, Nithya and Chehade, Mohamed Nabil Hajj
             and Balzano, Laura and Nair, Sheela and Zahedi, Sadaf
             and Kohler, Eddie and Pottie, Gregory J. and Hansen, Mark H.
             and Srivastava, Mani B.},
  title   = {Sensor Network Data Fault Types},
  journal = {ACM Transactions on Sensor Networks},
  volume  = {5},
  number  = {3},
  pages   = {25:1--25:29},
  year    = {2009},
  doi     = {10.1145/1525856.1525863}
}

\appendix

\section{Proof of Lemma~\ref{result: CSZ theorem}}\label{Appendix: Proof of Lemma result: CSZ theorem}

    For the purpose of applying~\cite[Lemma~5.3, p.~109]{chatterjee2022finite}, define the expected liminf average objective by
    \[
        \overline{\gamma}(b,\sigma)\defas\EE_{\sigma}^{b}\left(\liminf_{n\to\infty}\dfrac{1}{n}\sum_{m=1}^nG_m\right),
    \]
    and the expected liminf average value by $\overline{v}(b)\defas\sup_{\sigma\in\Sigma}\overline{\gamma}(b,\sigma)$.
    By~\cite[Lemma~5.3, p.~109]{chatterjee2022finite}, there exist $m_\eps\in\NN^*$, a strategy $\sigma_\eps\in\Sigma$, and a random belief $B^*\in\Delta(\K)$ satisfying the first item of the statement. Moreover, for every realization $b^*$ of $B^*$, there exists a strategy $\sigma_{b^*}\in\Sigma$ such that, for every $k\in\supp(b^*)$,
    \begin{align}
        \dfrac{1}{n}\sum_{m=1}^nG_m\xrightarrow[n\to\infty]{}\overline{\gamma}(\delta_k,\sigma_{b^*})\qquad\PP_{\sigma_{b^*}}^{\delta_k}\text{-almost surely},
        \label{equation: previous limit CSZ}
    \end{align}
    and $\overline{\gamma}(b^*,\sigma_{b^*})=\overline{v}(b^*)$ and $\EE_{\sigma_\eps}^{b_1}\left(\overline{v}(B^*)\right)\geq\overline{v}(b_1)-\eps$.
    
    Fix a realization $b^*$ of $B^*$ and let $\sigma_{b^*}$ be the corresponding strategy given by the lemma.
    Since the rewards are bounded in $[0,1]$, the dominated convergence theorem yields, for every $k\in \supp(b^*)$,
    \begin{align}
        \overline{\gamma}(\delta_k,\sigma_{b^*})
            &=\EE_{\sigma_{b^*}}^{\delta_k}\left(\liminf_{n\to\infty}\dfrac{1}{n}\sum_{m=1}^nG_m\right)
                &&\text{(def. of $\overline{\gamma}$)}\nonumber\\
            &=\EE_{\sigma_{b^*}}^{\delta_k}\left(\lim_{n\to\infty}\dfrac{1}{n}\sum_{m=1}^nG_m\right)
                &&\text{(by Eq. \eqref{equation: previous limit CSZ})}\nonumber\\
            &=\lim_{n\to\infty}\EE_{\sigma_{b^*}}^{\delta_k}\left(\dfrac{1}{n}\sum_{m=1}^nG_m\right)
                &&\text{(dominated convergence theorem)}\nonumber\\
            &=\liminf_{n\to\infty}\EE_{\sigma_{b^*}}^{\delta_k}\left(\dfrac{1}{n}\sum_{m=1}^nG_m\right)
                &&\text{(existence of the limit)}\nonumber\\
            &=\gamma(\delta_k,\sigma_{b^*})
                &&\text{(def. of $\gamma$).}
                \label{equation: equality CSZ and ours}
    \end{align}
    By \eqref{equation: previous limit CSZ} and \eqref{equation: equality CSZ and ours}, 
    \[
        \dfrac{1}{n}\sum_{m=1}^nG_m\xrightarrow[n\to\infty]{}\gamma(\delta_k,\sigma_{b^*})\qquad\PP_{\sigma_{b^*}}^{\delta_k}\text{-almost surely}.
    \]
    Since $\supp(b^*)$ is finite,
    \begin{align*}
        \overline{\gamma}(b^*,\sigma_{b^*})
            &=\sum_{k\in \K}b^*(k)\overline{\gamma}(\delta_k,\sigma_{b^*})
                &&\text{(linearity)}\\
            &=\sum_{k\in \K}b^*(k)\gamma(\delta_k,\sigma_{b^*})
                &&\text{(by Eq.~\eqref{equation: equality CSZ and ours})}\\
            &=\sum_{k\in \K}b^*(k)\lim_{n\to\infty}\EE_{\sigma_{b^*}}^{\delta_k}\left(\dfrac{1}{n}\sum_{m=1}^nG_m\right)\\
            &=\lim_{n\to\infty}\sum_{k\in \K}b^*(k)\EE_{\sigma_{b^*}}^{\delta_k}\left(\dfrac{1}{n}\sum_{m=1}^nG_m\right)
                &&\text{($\supp(b^*)$ is finite)}\\
            &=\lim_{n\to\infty}\EE_{\sigma_{b^*}}^{b^*}\left(\dfrac{1}{n}\sum_{m=1}^nG_m\right)
                &&\text{(linearity)}\\
            &=\liminf_{n\to\infty}\EE_{\sigma_{b^*}}^{b^*}\left(\dfrac{1}{n}\sum_{m=1}^nG_m\right)
                &&\text{(existence of the limit)}\\
            &=\gamma(b^*,\sigma_{b^*})
                &&\text{(def. of $\gamma$).}
    \end{align*}
    By Remark~\ref{remark: values coincides}, we have $\overline{v}(b)=v(b)$ for every belief $b\in\Delta(\K)$. 
    Hence,
    \[
        \gamma(b^*,\sigma_{b^*})=v(b^*)
    \]
    and $\EE_{\sigma_\eps}^{b_1}\left(v(B^*)\right)=\EE_{\sigma_\eps}^{b_1}\left(\overline{v}(B^*)\right)\geq\overline{v}(b_1)-\eps=v(b_1)-\eps$, which concludes the proof.

\end{document}